\documentclass[10pt,a4paper,twoside]{amsart} 
\usepackage[english]{babel}
\usepackage[latin1]{inputenc}
\usepackage[pdftex]{graphicx}
\usepackage{hyperref}
\usepackage{amssymb}
\usepackage{adjustbox}
\usepackage{amsmath}
\usepackage{latexsym}
\usepackage{amsthm}
\usepackage{verbatim}
\usepackage{tikz-cd}
\usepackage{enumitem}
\usepackage{afterpage}
\usepackage{xcolor,colortbl,soul}   
\usepackage{tabularx}
\counterwithin{figure}{section}
\newcolumntype{P}[1]{>{\centering\arraybackslash}p{#1}}
\usepackage{chngpage}
\usepackage{geometry}
\def\edgewidth {0.5 pt}
\def\vertexsize {1.2pt}   
\newcommand{\vertex}[2][1]{\fill (#2) circle [radius = #1 * \vertexsize];}
\newcommand{\divisor}[2][1]{\fill [red] (#2) circle [radius = 2* \vertexsize];}

\usepackage{xcolor,color,soul}

\definecolor{verde}{rgb}{0.01, 0.75, 0.24}

\newcommand{\val}{\operatorname{val}}

\theoremstyle{plain}                       
\newtheorem{theo}{Theorem}[section]
\newtheorem{prop}[theo]{Proposition}    
\newtheorem{coro}[theo]{Corollary}       
\newtheorem{lemm}[theo]{Lemma} 

\theoremstyle{definition}               
\newtheorem{defin}{Definition}

\theoremstyle{remark} 
\newtheorem{example}[theo]{Example} 
\newtheorem{rema}[theo]{Remark}  
\title{Tropical Trigonal Curves}
\date{}

\author{Margarida Melo}
\address{Universit\`a di Roma Tre\\ 
Dipartimento di Matematica e Fisica\\
Largo San Leonardo Murialdo 1\\
00146 Rome (Italy)}
\email{margarida.melo@uniroma3.it}

\author{Angelina Zheng}
\address{Universit\"at T\"ubingen\\
Fachbereich Mathematik\\
Auf der Morgenstelle 10\\
72076 T\"ubingen (Germany)}
\email{zheng@math.uni-tuebingen.de}

\begin{document}
\begin{abstract}
    We prove that the existence of a divisor of degree $3$ and Baker-Norine rank at least $1$ on a $3$-edge connected tropical curve is equivalent to the existence of a non-degenerate harmonic morphism of degree $3$ from a tropical modification of it to a tropical rational curve. 
    Using the second description, we define the moduli spaces of $3$-edge connected tropical trigonal covers and of $3$-edge connected tropical trigonal curves, the latter as a locus in the moduli space of tropical curves. Finally, we prove that the moduli space of $3$-edge connected genus $g$ tropical trigonal curves has the same dimension as the moduli space of genus $g$ algebraic trigonal curves.
\end{abstract}

\thanks{MM is supported by MIUR via the projects  PRIN2017SSNZAW (Advances in Moduli Theory and Birational Classification),  PRIN 2022L34E7W (Moduli spaces and birational geometry) and PRIN 2020KKWT53 (Curves, Ricci flat Varieties
and their Interactions), and is a member of the Centre for Mathematics of the University of
Coimbra -- UIDB/00324/2020, funded by the Portuguese Government through FCT/MCTES. 
AZ is supported by Alexander von Humboldt Foundation.
MM and AZ are members of the INDAM group GNSAGA}

\maketitle
\tableofcontents
\section{Introduction}
A smooth curve $C$ is said to be \emph{d-gonal} if it admits a $g_d^1$, i.e. a line bundle of degree~$d$ with non-trivial space of sections. 
The \emph{gonality} of $C$ is the smallest $d\in\mathbb Z_{>0}$ such that $C$ is $d$-gonal.
Since the datum of a $g_d^1$ on $C$  is equivalent to the existence of a morphism of degree $d$ from $C$ to $\mathbb P^1$, the gonality of $C$ is also the smallest degree of such a morphism.

The moduli space $M_g$ of smooth projective genus $g$ curves admits a stratification by \emph{gonality}, for $g\geq 3$:
$$ M_{g,2}^1\subseteq M_{g,3}^1\subseteq \dots M_{g,d}^1\subseteq\dots\subseteq M_g,$$
where $M^1_{g,d}:=\{\left[C\right] \in M_g: C \text{ has a }g^1_d\}$ is
an irreducible variety of dimension $2g+2d - 5$ when $d \leq (g + 2)/2$ and $M^1_{g,d}=M_g$ when $d \geq (g + 2)/2$ (see~\cite{AC81}).

Understanding the geometry of the locus of $M_g$ of curves with fixed gonality is a very natural problem, as the use of the additional data of the $g^1_d$ carries important information about the curve itself. Indeed, the locus $H_g=M_{g,2}^1$ of hyperelliptic curves is among the most studied loci of $M_g$. We know a lot about the geometry of this locus, e.g. its rational cohomology is completely known and it is tautological. 
The rational cohomology of the moduli space $H_{g,n}$ of $n$-pointed genus $g$ hyperelliptic curves is also known for $n\leq 2$ by work of Tommasi in~\cite{Tom}, and its $S_n$-equivariant Hodge-Euler characteristic for $n\leq 7$ has been computed by Bergstr\"om in~\cite{Ber}. 
Furthermore, the rational Chow ring of $H_{g,n}$ has been completely determined for $n\leq 2g+6$ in~\cite{CL_hyp} by Canning and Larson.

In recent years, there have been several works exploring connections between the topology of algebraic moduli spaces and their tropical counterparts. Such connections have been first explained for the moduli space of curves by Abramovich, Caporaso and Payne in~\cite{ACP}, where the authors prove that the dual boundary complex of the Deligne-Mumford compactification of $M_g$ is isomorphic to the moduli space of tropical curves $M_g^{\rm trop}$. Exploring this relation and applying Deligne's theory of weights, Chan, Galatius and Payne in~\cite{CGP} were able to prove breakthrough results for the cohomology of $M_g$ by studying a complex of graphs to compute the homology of $M_g^{\rm trop}$.

In the same spirit, Brandt, Chan, and Kannan have given in~\cite{BCK} a formula for the $S_n$-equivariant weight zero compactly supported Euler characteristic of $H_{g,n}$. This has been done by studying a graph complex defined by the dual complex of the boundary of the normal crossing compactification of $H_{g,n}$ via pointed admissible $\mathbb Z/2\mathbb Z$-covers, interpreted as a tropical moduli space.

The next natural case to consider is the case $d=3$: it is well known that the moduli space $T_g=M_{g,3}^1\setminus H_g$ of genus $g$ trigonal curves is irreducible of dimension $2g+1$, for $g\geq 4$, or of dimension $6$ if $g=3$.
Currently, much less is known about the geometry of $T_g$; for instance, while it is known that $H_g$ has the rational cohomology of a point, the rational cohomology of $T_g$ is only known for $g\leq5$,~\cite{Loo93},\cite{Tom},\cite{Z22}, while its rational Chow ring and stable cohomology, which turns out to be tautological, have been computed in~\cite{CL} and~\cite{Z24}, respectively.

In~\cite[Theorem 1.1]{CGP} it has been proved that the top-weight cohomology of $M_5$ does not vanish. From the description of the rational cohomology of $T_5$ computed in~\cite{Z22} and the spectral sequence determined by the inclusion $M_{5,3}^1\subseteq M_5$, one can deduce that the top-weight cohomology of $M_5$ is a non-zero multiple of that of $T_5$.
We wonder if this phenomenon generalizes to higher genera, and in order to do so it is natural to understand the boundary of $M_{g,3}^1,$ as a tropical moduli space.

In algebraic geometry, the moduli of stable trigonal curves $\overline{T}_g$ is the moduli space parameterizing stable trigonal curves, i.e., stable curves which are the stabilization of a nodal curve, which admits a degree $3$ admissible cover to a rational curve.

In this work, we study the analog of the boundary of $\overline{T}_g$ in $\overline{M}_g$, in the moduli space $M_g^{\operatorname{trop}}$.
For combinatorial graphs, 
the divisor theory has been introduced in~\cite{BN07} and shown to have remarkable similarities to divisor theory on curves; namely, there is a well-defined notion of rank for which the Brill-Noether theorem holds.
Likewise, there is a notion of harmonic morphism of graphs, introduced by the same authors in~\cite{BN}.

Based on these two notions, there have been introduced (at least) two different definitions of gonality for graphs, namely \emph{geometric gonality} and \emph{divisorial gonality}.
A graph is geometrically $d$-gonal if it admits a non-degenerate harmonic morphism of degree $d$ to a tree, and it is divisorially $d$-gonal if there exists a divisor on the graph of degree $d$ and rank $1$,
see~\cite{LC} for the precise definitions. 
However, unlike the algebraic geometric case, the equivalence between the two definitions is known to hold only for $d=2$,~\cite{MC}, or $d=3,$~\cite{ADMYY}, in the $3$-edge connected case.

Notice that unlike the algebraic case, when defining a (divisorally) $d$-gonal graph, it might also be $(d-1)$-gonal. This is because we want to describe the boundary of the closure of the locus of algebraic $d$-gonal curves, which contains curves of lower gonality. However, we will also see that $3$-edge connectivity provides a lower bound on the gonality of graphs, hence this difference will not play a role when studying trigonality for $3$-edge connected graphs.

Both the divisor theory and the notion of harmonic morphism have been generalized to metric graphs.
For degree $2$, Melody Chan showed in~\cite{MC} that both notions still coincide for metric graphs; more precisely, a metric graph admits a divisor of degree~$2$ and rank $1$ if and only if it admits a non-degenerate harmonic morphism of degree~$2$ to a metric tree; such metric graphs are therefore called hyperelliptic.
However, it is not difficult to find examples that show that this is not true for degree $3$, see Example~\ref{not_trigonal}.

In the present work, we consider analogous definitions of gonality for metric graphs and we give a precise relation between the two for the case $d=3$. More precisely, we show
that the result for hyperelliptic curves proved in~\cite{MC} generalizes to $d=3$ for $3$-edge connected metric graphs, if we allow tropical modifications:

\begin{theo}\label{th:main}
    Let $\Gamma$ be a $3$-edge connected metric graph with canonical loopless model $(G_{-},l_{-})$. The following are equivalent.
    \begin{itemize}
        \item[A.] $|V(G_{-})|=2,3$ or there exists a non-degenerate harmonic morphism of degree $3$ from a tropical modification of $\Gamma$ to a metric tree.
        \item[B.] $\Gamma$ is divisorially trigonal.
    \end{itemize}
\end{theo}

In the special case of loopless $3$-edge connected metric graphs (i.e. there is no loop in any model of the graph), the equivalence in Theorem~\ref{th:main} is actually stronger as we show that $\Gamma$ itself admits a non-degenerate harmonic morphism of degree $3$ from $\Gamma$ to a tree (see Theorem~\ref{th:Main_Theo}).
In general, whenever $\Gamma$ has loops we might need to perform tropical modifications at $\Gamma$ by inserting trees at suitable vertices in order to guarantee the existence of the morphism: the case of $3-$edge connected graphs with loops is considered in Theorem~\ref{th:Main_Theo_loops}.

Even though the restriction to $3$-edge connected graphs is natural, one may ask if Theorem~\ref{th:main} holds if we remove this connectivity condition. The answer is no in general: see for instance~\cite[Example 5.13]{ABBR}, where the authors exhibit an example of a divisorially trigonal $2$-edge connected metric graph, which admits no non-degenerate harmonic morphism of degree $3$ to a tree, even if we allow tropical modifications.
However, the theorem can be extended in many cases, even though the situation can be combinatorially much more complicated. For instance, some metric graphs may be hyperelliptic or contain hyperelliptic subgraphs which might be used to perform tropical modifications of the curves which admit a harmonic morphism of degree $3$ to a metric tree.
We will give a generalization of Theorem~\ref{th:main} with no restriction on the edge connectivity of the graph in a follow-up to this paper.

On the positive side, for any $d\in\mathbb Z_{>0}$ and any metric graph $\Gamma$ (with no assumptions on its edge-connectivity), the existence of a non-degenerate harmonic morphism of degree $d$ from $\Gamma$ (or a tropical modification of it) to a tree always determines a divisor on $\Gamma$ of degree $d$ and rank at least $1$, see Lemma~\ref{lm:Lemma_preim}.

Notice that we can think of the statement of Theorem~\ref{th:main} as a tropical analog of the description of the stable trigonal locus $\overline T_g\subset \overline{M}_g$. Indeed, in Harris-Mumford's theory of admissible covers, it is allowed to modify the original curve by adding rational tails (see~\cite{HM82}): the tropical modifications we perform can be interpreted as tropical analogs of those modifications.

Based on the correspondence described in Theorem~\ref{th:main} we construct the moduli spaces $H_{g,3}^{\operatorname{trop},(3)}$ and $T_{g}^{\operatorname{trop},(3)}$, of $3$-edge connected tropical trigonal covers and $3$-edge connected tropical trigonal curves, respectively, as generalized cone complexes. 
In particular, the moduli space $T_g^{\operatorname{trop},(3)}$ has pure dimension $2g+1$ for $g\geq 4$, and $6$ for $g=3$, as expected in analogy with the classical algebraic case. 
In order to prove that $T_g^{\operatorname{trop},(3)}$ has the expected dimension, we give an explicit construction of its maximal cells, called \emph{$3$-ladders} and prove that $T_g^{\operatorname{trop},(3)}$ is also connected through codimension~$1$.

Finally, as a consequence of our construction, we will also show that $3$-ladders, together with a suitable morphism, are in fact tropical admissible covers, as defined by~\cite{CMR}.

Regarding the connection with the situation of algebraic curves, the hope would be to relate $T_g^{\operatorname{trop},(3)}$ with the boundary of $M_{g,3}^1$ in $\overline{M}_{g,3}^1$ and use it to understand part of the topology of $T_g.$ 
Notice that the homology of the trigonal locus is expected to coincide with its restriction to $3$-edge connected graphs due to the presence of separating vertices, or multiple edges.
Indeed, as shown in~\cite{ACP22}, the restriction of the (link of the) moduli space of tropical curves to the locus of curves with bridges, cut vertices, loops, weights or multiple edges is contractible. Therefore, studying the topology of $T_g^{\operatorname{trop},(3)}$ should be enough to obtain results on the topology of $T_g$ itself. 
This would also be significant in understanding the cohomology of $M_g.$

In Section~\ref{sc: preliminaries}, we recall the definitions and main properties of harmonic morphisms of graphs and metric graphs, along with tropical modifications of these. Moreover, we also recall the divisor theory on metric graphs and discuss possible definitions of $d$-gonality for metric graphs. For $d=2$, we recall the results obtained by Melody Chan in~\cite{MC} and we observe similarities and differences between the cases $d=2$ and $d=3$.

In Section~\ref{sc:3_connected}, we prove Theorem~\ref{th:main} first in the case of loopless graphs and then in the case with loops.

Finally, in Section~\ref{sc:moduli} we construct the moduli spaces of $3$-edge connected tropical trigonal covers and $3$-edge connected tropical trigonal curves. For the latter, we will also provide a construction for the maximal cells which will allow us to compute the dimension of the moduli space, which coincides with that of the moduli space of algebraic trigonal curves. 

\medskip

\noindent 
{\bf Acknowledgements.} It is our pleasure to thank Lucia Caporaso for her insights and enlightening discussions at an early stage of the paper. We also thank Melody Chan, Hannah Markwig and Felix R\"ohrle for helpful comments. We are grateful to Matthew Baker and Erwan Brugall\'e for informing us of Luo's example in~\cite[Example 5.13]{ABBR}, which helped us to correct a mistake in the statement of our main theorem in a previous version of this paper. Finally, we thank the anonymous referees for their valuable comments and suggestions.
\section{Preliminaries}\label{sc: preliminaries}
Let us recall the definitions of the main objects we will discuss. We will mostly follow~\cite{MC}.

\subsection{Graphs}
Given a graph $G$, we denote by $V(G),E(G)$ the sets of vertices and edges, respectively, and by $E_v(G)$ the set of edges incident to a given vertex $v\in V(G)$. Given an edge $e\in E(G)$, we will abuse notation and write $e=uv$ to indicate that $u$ and $v$ are the endpoints of $e$. The valence of a vertex $v\in V(G)$, $\val(v),$ is the cardinality of $E_v(G)$, with loops counted twice. Our graphs are assumed to be finite and connected.

A \textbf{weighted graph} is a pair $(G,w)$, where $G$ is a graph and $w$ is a weight function $w:V(G)\to \mathbb Z_{\geq 0}$. Given a graph $G$, we will often consider $G$ to be a weighted graph by endowing $G$ with the trivial weight function $\underline 0$.
Given a (weighted) graph $G=(G,w)$, the \textbf{genus} of $G$ is set to be
$$g(G):=\sum_{v\in V(G)}w(v)+b_1(G)=g,$$
where $b_1(G)=|E(G)|-|V(G)|+1$ is the first Betti number of $G$.

Given an edge $e$ in a weighted graph $G=(G,w)$, the weighted contraction of $e$ is the weighted graph $G/e=(G/e, w/e)$ such that:
\begin{itemize}
\item the edge $e$, together with its endpoints $u,v\in V(G)$, is identified with a vertex $v_e\in V(G/e)$ of weight $w(u)+w(v)$ if $u$ and $v$ are different, and $w(u)+1$ if $u=v$ (i.e., if $e$ is a loop-edge).
\item $G$ and $G/e$ are identified outside of $e$.
\end{itemize}
It follows immediately from the definition that a weighted contraction preserves the genus.
Given a set $S=\{e_1,\dots, e_n\}\subset E(G)$, the contraction of $S$ is obtained by composing the single edge contractions of $e_1,\dots,e_n$.

A graph $G$ is said to be \textbf{stable} if   $\forall v\in V(G)$ with $w(v)=0$, we have that $\val(v)\geq 3$ and if $w(v)=1$, we have that $\val(v)\geq 1$. It is easy to see that weighted contractions preserve stability.

Given a graph $G$, the stable model of $G$ is the graph $G^{\operatorname{st}}$ obtained by contracting all leaves and then, for any vertex of weight $0$ and valence $2$, contracting one of the two adjacent edges.

We denote by $G_{-}$ the loopless model of $G,$ obtained by adding in $G^{\operatorname{st}}$ a vertex in the interior of each loop. Notice that if $G^{\operatorname{st}}$ has loops, then $G_{-}$ is not stable. 

Finally, to any (weighted) graph $(G,w)$ we can associate a weightless graph $(G^w,\underline{0})$ obtained by attaching loops at each vertex $v\in V(G)$. The weighted contraction of all added loops in $G^w$ yields again $(G,w)$.

\subsection{Metric graphs}

\begin{defin}
    A \textbf{metric graph} $\Gamma$ is a metric space such that there exists a graph~$G$ and a length function $l:E(G)\to\mathbb R_{>0}$ so that $\Gamma$ is obtained by gluing intervals $\left[0,l(e)\right]$ for any $e\in E(G)$ at their endpoints, as prescribed by the combinatorial data of $G$. We write $\Gamma=(G,l)$ and we say that $(G,l)$ is a model for $\Gamma$.

    The \textbf{valence} $\operatorname{val}(x)$ of a point $x\in \Gamma$ is the number of connected components in~$U_x\setminus \{x\},$ where $U_x$ is a small neighborhood of $x$. In particular, notice that all but finitely many points $x\in \Gamma$ have valence $2$.
    The \textbf{distance} $d(x,y)$ between two points $x,y\in \Gamma$ is the length of the shortest path between them.
\end{defin}

\begin{defin}
    The \textbf{canonical model} $\Gamma_0=(G_0,l_0)$ of a metric graph $\Gamma=(G,l)$ with $g(G)\geq2$ is the unique model obtained by iteratively removing all vertices in $V(G)$ of valence 1 and their adjacent edges and all vertices of valence $2$ and replacing their adjacent edges $e_1,e_2$ by an edge $e$ of length $l_0(e)=l(e_1)+l(e_2).$  
\end{defin}

Notice that the canonical model $\Gamma_0=(G_0,l_0)$ of $\Gamma=(G,l)$ is such that $G_0=G^{\operatorname{st}}.$ Moreover, if $G$ has leaves, then $\Gamma_0\subsetneq\Gamma.$ 

\begin{defin}
    Let $G$ be any graph. A \textbf{refinement} $G'$ of $G$ is a graph obtained from $G$ by inserting vertices in the interior of some edges. 
    We say that $(G',l')$ is a \textbf{refinement} of $(G,l)$ if $G'$ is a refinement of $G$ and, for any $e\in E(G),$ $l(e)=\sum_{i}l'(e_i),$ where $\{e_i\}$ are the edges in $G'$ whose union gives $e.$
    Observe that $(G',l')$ and $(G,l)$ coincide as metric spaces.
    
    The \textbf{canonical loopless model} $(G_{-},l_{-})$ of a metric graph $(G,l)$ is a refinement of the canonical model $(G_0,l_0)$ where $G_{-}$ is the loopless model of $G_0$ and the vertices are added at the midpoint of each loop.
\end{defin}

\begin{defin}
    An \textbf{abstract tropical curve} is a triple $(G,w,l)$ where $(G,l)$ is a metric graph and $(G,w)$ is a stable graph.
\end{defin}

\begin{defin}
Two metric graphs are \textbf{tropically equivalent} if they have the same canonical model.
A \textbf{tropical modification} of a metric graph $\Gamma$ is a metric graph $\Gamma'$ tropically equivalent to $\Gamma,$ obtained by gluing trees at some points. In particular, $\Gamma$ is a topological retraction of any tropical modification.
\end{defin}
\begin{rema}
    Notice that if $\Gamma'$ is a tropical modification of $\Gamma$, then both have the same canonical model. 
\end{rema}

\subsection{Harmonic morphisms of graphs and metric graphs}

\begin{defin}
A \textbf{morphism of graphs} $\varphi:G\rightarrow G'$ is a map of sets 
$$\varphi:V(G)\cup E(G)\rightarrow V(G')\cup E(G')$$ such that
\begin{enumerate}
    \item $\varphi(V(G))\subseteq V(G');$
    \item if $e=xy\in E(G)$ is such that $\varphi(e)\in V(G'),$ then $\varphi(e)=\varphi(x)=\varphi(y);$
    \item if $e=xy\in E(G)$ is such that $\varphi(e)\in E(G'),$ then $\varphi(e)=\varphi(x)\varphi(y).$
\end{enumerate}
An \textbf{indexed morphism} is a morphism $\varphi:G\rightarrow G'$ that assigns to each $e\in E(G)$ a non-negative integer $\mu_{\varphi}(e)$ with $\mu_{\varphi}(e)=0$ if and only if $\varphi(e)\in V(G').$

An indexed morphism is \textbf{harmonic}\footnote{This coincides with the definition of pseudo-harmonic morphism in~\cite[Definition 6(B)]{LC}.} if for any $x\in V(G)$ the quantity $$m_{\varphi}(x)=\sum_{\substack{e\in E_x(G)\\ \varphi(e)=e'}}\mu_{\varphi}(e)$$ 
is constant for any $e'\in E_{\varphi(x)}(G')$.
If the sum is taken over the empty set, we set $m_\varphi(x)=0.$ Moreover we say that the morphism is $\textbf{non-degenerate}$ if it contracts no loops and if $m_\varphi(x)\geq1$ for any $x\in V(G).$ 
The \textbf{degree} of a harmonic morphism $\varphi$ is 
$$\operatorname{deg}\varphi=\sum_{\substack{e\in E(G)\\ \varphi(e)=e'}}\mu_{\varphi}(e)$$
for any $e'\in E(G').$ Note that for a harmonic morphism the degree is independent of the choice of $e',$~\cite[Lemma 2.12]{U}.
\end{defin}

\begin{defin}\label{def_metric}
Let $\Gamma=(G,l)$ and $ \Gamma'=(G',l')$ be metric graphs.
Suppose that $\varphi$ is an indexed morphism of graphs with 
\begin{equation}\label{index}
\mu_{\varphi}(e)=\begin{cases}
\frac{l'(\varphi(e))}{l(e)}\in\mathbb{Z}, &\text{if }\varphi(e)\in E(G')\\
0,&\text{if }\varphi(e)\in V(G')
\end{cases}\end{equation}
for any $e\in E(G)$.
We say that $\varphi$ induces a \textbf{morphism of metric graphs}, as topological spaces, $$\tilde{\varphi}:(G,l)\to(G',l'),$$ which is continuous and linear along each edge $e$ mapping to an edge $e'$.  
Such a morphism of metric graphs is \textbf{non-degenerate}, resp. \textbf{harmonic}, if the morphism of graphs $\varphi$ is. We also define its degree $\operatorname{deg}\tilde{\varphi}$ to be equal to $\operatorname{deg}\varphi$ and the \textbf{horizontal multiplicity} at $x\in \Gamma$ as
\begin{equation}\label{hor_mult}
    m_{\tilde{\varphi}}(x)=\begin{cases}
    m_{\varphi}(x),&x\in V(G)\\
    0,&x\in\mathring{e}; \varphi(e)\in V(G')\\
    \mu_{\varphi}(e),&x\in\mathring{e}; \varphi(e)\in E(G').
    \end{cases}
\end{equation}
\end{defin}

\begin{rema}\label{rk:lengths}
Notice that the definition of harmonic morphism on metric graphs depends on the choice of a model.
Let us stress that given an indexed morphism of graphs $\varphi:G\to G'$ and a metric graph $\Gamma'=(G',l')$, we can always find a metric graph $\Gamma=(G,l)$ and a morphism $\tilde{\varphi}:(G,l)\to(G',l')$ such that if $\varphi$ is non-degenerate (resp. harmonic of degree $d$), then $\tilde{\varphi}$ is non-degenerate (resp. harmonic of degree $d$).
Indeed, from the definition, it follows that the length function $l:E(G)\rightarrow \mathbb{R}_{>0}$ is uniquely determined by the indices $\mu_{\varphi}(e)$ at the edges which are not contracted by $\varphi$.

However, if we fix a metric graph $\Gamma=(G,l)$ and an indexed morphism $\varphi:G\to G'$, it may be impossible to find a metric graph with underlying graph $G'$ with an induced morphism of metric graphs $\tilde{\varphi}:(G,l)\to (G',l'):$  
the relations~\eqref{index} between the indices of the morphism and the lengths of the models have to be satisfied. 
\end{rema}

\begin{example}
Here are some examples of morphisms of (metric) graphs. If not specified, all indices are $1$ (edges of the same color are sent to edges of the same color and they all have the same length).
\end{example}

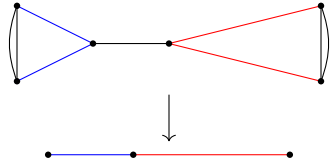
\begin{figure}[ht]
\centering
\begin{tikzcd}
    \begin{tikzpicture}
    \draw (0,0) --  (1, 0);
    \draw [blue] (-1,0.5) -- (0,0);
    \draw [blue] (-1,-0.5) -- (0,0);
    \draw [red] (3,0.5) -- (1,0);
    \draw [red] (3,-0.5) -- (1,0);
    \draw (-1,0.5)--(-1,-0.5);
    \draw (3,0.5)--(3,-0.5);
    \draw (-1,0.5) to [out=250, in=110](-1,-0.5);
    \draw (3,0.5) to [out=290, in=70](3,-0.5);
    \vertex{-1,0.5}
    \vertex{-1,-0.5}
    \vertex{3,0.5}
    \vertex{3,-0.5}
    \vertex{0,0}
    \vertex{1,0}
    \end{tikzpicture}
    \arrow[d] \\
    \begin{tikzpicture}
    \draw [blue] (-1.12,0) -- (0,0);
    \draw [red] (2.06,0) -- (0,0);
    \vertex{-1.12,0}
    \vertex{0,0}
    \vertex{2.06,0}
    \end{tikzpicture}
\end{tikzcd}
\caption{A non-degenerate non-harmonic morphism.}\label{figure1} 
\end{figure}

\begin{figure}[ht]
\centering
\begin{tikzcd}
    \begin{tikzpicture}
    \draw [blue] (0,0) -- (2,0);
    \draw [blue] (0,0.8) -- (2,0.8);
    \draw [blue] (0,1.6) -- (2,1.6);
    \draw (0,1.6) -- (0,0.8);
    \draw (0,0) -- (0,0.8);
    \draw (2,1.6) -- (2,0.8);
    \draw (2,0) -- (2,0.8);
    \draw (0,0) --(-0.5,0.8);
    \draw (0,0.8) --(-0.5,0.8);
    \draw (0,1.6) --(-0.5,0.8);
    \vertex{-0.5,0.8}
    \begin{scope}
		\clip (2,0) rectangle ++(0.5,1.6);
		\draw (2,0.8) ellipse (0.5 and 0.8);
	\end{scope}
    
    \foreach \i in {0,2} {
    \foreach \j in {0,0.8,1.6} {
    	    \vertex{\i, \j}
         }
    	}
    
    \end{tikzpicture}
    \arrow[d] \\
    \begin{tikzpicture}
    \draw [blue] (0,0) -- (2,0);
    \vertex{0,0}
    \vertex{2,0}
    \end{tikzpicture}
\end{tikzcd}
\caption{A degenerate and harmonic morphism of degree~$3$.}\label{figure2} 
\end{figure}

\begin{figure}[ht]
\centering
\begin{tikzcd}
    \begin{tikzpicture}
    \draw[line width=2*\edgewidth, blue](0,0) --  (1, 0);
    \draw (0.5,0) node[anchor=south] {$2$};
    \draw [blue] (-2,0.5) -- (0,0);
    \draw [blue] (-2,-0.5) -- (0,0);
    \draw [red] (3,0.5) -- (1,0);
    \draw [red] (3,-0.5) -- (1,0);
    \draw (-2,0.5)--(-2,-0.5);
    \draw (3,0.5)--(3,-0.5);
    \draw (-2,0.5) to [out=250, in=110](-2,-0.5);
    \draw (3,0.5) to [out=290, in=70](3,-0.5);
    \vertex{-2,0.5}
    \vertex{-2,-0.5}
    \vertex{3,0.5}
    \vertex{3,-0.5}
    \vertex{0,0}
    \vertex{1,0}
    \end{tikzpicture}
    \arrow[d] \\
    \begin{tikzpicture}
    \draw [blue] (-2,0) -- (0,0);
    \draw [blue] (0,0) -- (2,0);
    \draw [red] (2,0) -- (4.06,0);
    \vertex{0,0}
    \vertex{2,0}
    \vertex{-2,0}
    \vertex{4.06,0}
    \end{tikzpicture}
\end{tikzcd}\caption{A non-degenerate harmonic morphism of degree $2.$}\label{figure3}
\end{figure}

\begin{defin}
    Let $\varphi:(G,l)\rightarrow (G',l')$ be a harmonic morphism of metric graphs. A~\textbf{tropical modification} of $\varphi$ is a harmonic morphism of metric graphs $\tilde \varphi:(\tilde G,\tilde l)\rightarrow (\tilde G',\tilde l')$ such that $(\tilde G,\tilde l),$ resp. $(\tilde G',\tilde l')$, is a tropical modification of $(G,l)$, resp. $(G',l')$, and the following diagram is commutative.
\begin{center}
\begin{tikzcd}
(G,l) \arrow[leftarrow]{r}  \arrow[d,"\varphi"] &(\tilde G,\tilde l)\arrow[d,"\tilde \varphi"]\\
(G',l')\arrow[leftarrow]{r}  &(\tilde G',\tilde l')
\end{tikzcd}
\end{center} 
Notice that the horizontal arrows in the diagram induce topological retractions of the associated metric spaces.
\end{defin}

\begin{prop}\label{prp: harm_contractions}
Given a non-degenerate harmonic morphism of metric graphs $\varphi:(G,l)\rightarrow (G',l')$ of degree $d\geq 2$, there is a tropical modification $\tilde{\varphi}:(\tilde{G},\tilde{l})\rightarrow (\tilde{G}',\tilde{l}')$ which is non-degenerate, harmonic of the same degree but with no contractions. 
\end{prop}
\begin{proof}
    Let $e=uv\in E(G)$ be such that $\varphi(e)=v'\in V(G').$
    By non-degeneracy, $m_{\varphi}(x)\geq 1$ for any $x\in V(G).$ Define the graph $(G_e,l_e)$ by adding a vertex $w$ in the mid-point of $e$ and by adding 
    $m_{\varphi}(x)$ leaves of length $l(e)/2$ at any $x\in (\varphi^{-1}(v')\cap V(G))\setminus \{u,v\}$ and $m_{\varphi}(x)-1$ leaves of the same length at $x=u$ and $x=v.$ 
    Define $(G'_e,l_e')$ by adding only one leaf incident to $v'$, again with the same length. 
    Finally define $\varphi_e:({G}_e,l_e)\rightarrow({G}'_e,l'_e)$ as the morphism which coincides with $\varphi$ over $(V(G)\setminus\{u,v\})\cup E(G)\setminus\{e\}$ sending $w$ and all the added vertices of valence $1$ to the endpoint of the leaf added in $G_e'$, and all the adjacent edges and the added leaves in $G_e$ to the added leaf in $G_e'$. By construction, for any $x\in (\varphi_e^{-1}(v')\cap V(G))\setminus \{u,v\}$, the quantity $m_{\varphi_e}(x)=m_{{\varphi}}(x)$ is equal to the number of added leaves at $x,$ therefore it is constant for any $e'\in E(G_e').$
    For $x\in\{u,v\},$ the number of edges sent to the leaf is again $m_{{\varphi}}(x):$ we added $m_{\varphi}(x)-1$ edges plus the edge whose other endpoint is $w.$
    
    Thus by construction, the morphism ${\varphi}_e$ is again harmonic. No contraction has been added, therefore the resulting morphism is also non-degenerate. Repeating the construction for any $e\in G$ such that $\varphi(e)\in V(G')$ yields metric graphs  $(\tilde{G},\tilde{l}), (\tilde{G}',\tilde{l}')$ and a morphism $\tilde\varphi:(\tilde{G},\tilde{l})\rightarrow (\tilde{G}',\tilde{l}')$ which is non-degenerate, harmonic with the same degree of $\varphi$ and with no contractions.
\end{proof}

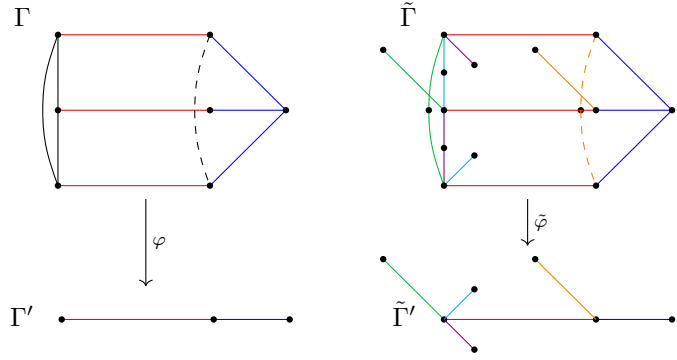
\begin{figure}[ht]
\begin{tikzcd}
\begin{tikzpicture}
\draw (-0.2,1.2) node[anchor=east] {$\Gamma$};
    \draw[red](0,0)--(2,0);
    \draw[red](0,1)--(2,1);
    \draw[red](0,-1)--(2,-1);
    \draw[blue](2,0)--(3,0);
    \draw[blue](2,1)--(3,0);
    \draw[blue](2,-1)--(3,0);
    \draw(0,1) to [out=250, in=90] (-0.2,0);
    \draw(-0.2,0) to [out=270, in=110] (0,-1);
    \draw(2,1)[dashed] to [out=250, in=90] (1.8,0);
    \draw(1.8,0)[dashed] to [out=270, in=110] (2,-1);
    \draw(0,1)--(0,0);\draw(0,-1)--(0,0);
    \foreach \i in {0,2} {
    \foreach\j in {-1,0,1}{
    \vertex{\i, \j}}
    }  
    \vertex{3,0}
    \end{tikzpicture}\arrow[d,"\varphi"] &\begin{tikzpicture}
\draw (-0.2,1.2) node[anchor=east] {$\tilde \Gamma$};
    \vertex{1.8,0}
    \draw[red](0,0)--(2,0);
    \draw[red](0,1)--(2,1);
    \draw[red](0,-1)--(2,-1);
    \draw(2,1)[dashed,orange] to [out=250, in=90] (1.8,0);
    \draw(1.8,0)[dashed,orange] to [out=270, in=110] (2,-1);
    \draw[orange](2,0)--(1.2,0.8);\vertex{1.2,0.8}
    \draw[blue](2,0)--(3,0);
    \draw[blue](2,1)--(3,0);
    \draw[blue](2,-1)--(3,0);
    \draw(0,1)[verde] to [out=250, in=90] (-0.2,0);
    \draw(-0.2,0)[verde] to [out=270, in=110] (0,-1);
    \draw[verde](0,0)--(-0.8,0.8);
    \draw[cyan](0,1)--(0,0.5);
    \draw[cyan](0,0)--(0,0.5);\vertex{0,0.5}
    \draw[cyan](0,-1)--(0.4,-0.6);\vertex{0.4,-0.6}
    \draw[violet](0,-1)--(0,-0.5);
    \draw[violet](0,0)--(0,-0.5);\vertex{0,-0.5}
    \draw[violet](0,1)--(0.4,0.6);\vertex{0.4,0.6}
    \vertex{-0.2,0}
    \foreach \i in {0,2} {
    \foreach\j in {-1,0,1}{
    \vertex{\i, \j}}
    }  
    \vertex{3,0}\vertex{-0.8,0.8}
    \end{tikzpicture}\arrow[d,"\tilde\varphi"]\\
    \begin{tikzpicture}
    \draw (-0.2,0) node[anchor=east] {$\Gamma'$};
    \draw[red](0,0) -- (2, 0);
    \draw[blue](2,0) -- (3, 0);
    \foreach \i in {0,2,3} {
    	    \vertex{\i, 0}
    	    }  
    \end{tikzpicture}&\begin{tikzpicture}
    \draw (-0.2,0) node[anchor=east] {$\tilde \Gamma'$};
    \draw[red](0,0) -- (2, 0);
    \draw[blue](2,0) -- (3, 0);
    \draw[verde](0,0)--(-0.8,0.8);
    \foreach \i in {0,2,3} {
    	    \vertex{\i, 0}
    	    }  
    \vertex{-0.8,0.8}
    \draw[orange](2,0)--(1.2,0.8);\vertex{1.2,0.8}
    \draw[cyan](0,0)--(0.4,0.4);\vertex{0.4,0.4}
    \draw[violet](0,0)--(0.4,-0.4);\vertex{0.4,-0.4}
    \end{tikzpicture}\end{tikzcd}
    \caption{A non-degenerate harmonic morphism of degree $3$ and a tropical modification without edge contractions.}\label{fg:harm_morph}
\end{figure}

\subsection{Divisor theory on metric graphs}

\begin{defin}
Let $\Gamma$ be a metric graph. A \textbf{divisor} of $\Gamma$ is a formal sum
$$D:=\sum_{x\in \Gamma}D(x)x,$$
where $D(x)\in\mathbb{Z}$ and $D(x)\neq0$ for finitely many $x\in\Gamma.$ 
The \textbf{support} of $D$ is the set $\operatorname{Supp}(D)=\{x\in\Gamma|\,D(x)\neq0\}$.

We denote by $\operatorname{Div}(\Gamma)$ the group of all divisors in $\Gamma.$

The \textbf{degree} of $D$ is $\operatorname{deg}(D)=\sum_{x\in \Gamma}D(x).$ The divisor group decomposes as 
$$\operatorname{Div}(\Gamma)=\cup_{d\in \mathbb Z} \operatorname{Div}^d(\Gamma),$$
where $\operatorname{Div}^d(\Gamma)$ stands for divisors of degree $d$.

We say that $D$ is \textbf{effective}, and write $D\geq 0,$ if $D(x)\geq 0,$ for any $x\in\Gamma.$ 

A \textbf{rational function} on $\Gamma$ is a continuous and piecewise linear function $f:\Gamma\rightarrow\mathbb{R},$ with integer slopes along its domains of linearity.

A divisor $D$ is \textbf{principal} if $D=\operatorname{div}(f),$ for some rational function $f,$ where $\operatorname{div}(f)$ is the divisor with $\operatorname{div}(f)(x)$ equal to the sum of all slopes of $f$ along edges emanating from $x.$ Principal divisors have degree $0$ and form a subgroup $\operatorname{Prin}(\Gamma)$ of $\operatorname{Div}^0(\Gamma)$.

Two divisors are \textbf{linearly equivalent} if their difference is principal. In particular, linearly equivalent divisors have the same degree.

The \textbf{Jacobian} of $\Gamma$ is the quotient $\operatorname{Jac}(\Gamma)=\operatorname{Div}(\Gamma)/\operatorname{Prin}(\Gamma)$. 
It also has a decomposition
$$\operatorname{Jac}(\Gamma)=\cup_{d\in \mathbb Z} \operatorname{Jac}^d(\Gamma),$$
with $\operatorname{Jac}^d(\Gamma)=\operatorname{Div}^d(\Gamma)/\operatorname{Prin}(\Gamma).$

The \textbf{rank} of $D$ is set to be  $\operatorname{rk}(D)=-1$ if $D$ is not equivalent to an effective divisor; otherwise $$\operatorname{rk}(D)=\operatorname{max}\{k\in\mathbb{Z}_{\geq 0}:\forall E\geq 0\text{ with }\operatorname{deg}(E)=k, \exists\,E'\geq 0 \text { such that } D-E\sim E'\}.$$

Finally, we define the $(r,d)$ \textbf{Brill-Noether locus} of $\Gamma$:
$$W_d^r(\Gamma)=\{\left[D\right]\in \operatorname{Jac}^d(\Gamma)|\,\operatorname{rk}(D)\geq r\}.$$ 
\end{defin}

For simplicity, we will refer to $D\in W_d^r(\Gamma)$, with $r\geq 0$, to be always effective.  

\begin{rema}\label{rk:Dhar}
In the following, we will also make extensive use of Dhar's burning algorithm,~\cite{D},
as a sufficient condition to determine when a divisor cannot have rank $\geq1$.

Given an effective divisor $D\in\operatorname{Div}(\Gamma)$ and a point $w\in \Gamma$, Dhar's burning algorithm consists of starting a fire at $w,$ which spreads along each edge incident to $w$. Any $x\in \Gamma$ with $D(x)>0$ can control the fire in $D(x)$ directions: if the $D(x)$ is bigger or equal than the number of burnt edges incident to $x$ then the fire stops at $x$, otherwise it burns $x$ and continues to spread along all remaining edges incident to $x$.
Following~\cite{BS}, if the entire graph burns, then $D-w$ is a $w\mbox{-}$ reduced divisor with $\operatorname{rk}(D-w)=-1$, which implies that it cannot be linearly equivalent to an effective divisor; therefore $r(D)<1$.
\end{rema}

\begin{rema}\label{rk:tropmod}Given a tropical modification $\Gamma'$ of $\Gamma,$ the natural inclusion of metric spaces $i:\Gamma\hookrightarrow \Gamma'$ induces a morphism $i_*:\operatorname{Div}(\Gamma)\to\operatorname{Div}(\Gamma')$ which preserves the rank. Indeed, it is easy to see that the Jacobian of a metric tree is trivial.
\end{rema}

A harmonic morphism induces a map on rational functions and on divisors.
\begin{defin}
Let ${\varphi}:\Gamma\rightarrow \Gamma'$ be a harmonic morphism of degree $d$. 
\begin{enumerate}
    \item Given a rational function $f:\Gamma'\to \mathbb R$, the pullback of $f$ via $\varphi$ is the function $\varphi^*(f):\Gamma\to \mathbb R$ such that $\varphi^*(f)=g\circ \varphi$. One may check that $\varphi^*(f)$ is again rational.
\item The \textbf{pullback}
$\varphi^*:\operatorname{Div}(\Gamma')\rightarrow\operatorname{Div}(\Gamma)$ is the map that, given a divisor $D'$ in $\Gamma',$ defines a divisor $\varphi^*D'\in\operatorname{Div}(\Gamma),$ with $(\varphi^*D')(x)=m_{\varphi}(x)\cdot D'(\varphi(x)),$
for any $x\in\Gamma.$ One can easily check that $\varphi^*(\operatorname{Div}^d(\Gamma'))\subset \operatorname{Div}^{\operatorname{deg}(D)\cdot d}(\Gamma)$.
\end{enumerate}
\end{defin}

\begin{rema}
It follows from~\cite[Proposition 4.2(ii)]{BN} and~\cite[Proposition 2.7]{MC} that the pullback morphism on divisors respects linear equivalence; therefore, it induces a map
$\varphi^*:\operatorname{Jac}(\Gamma')\to \operatorname{Jac}(\Gamma)$.
\end{rema}

\subsection{Gonality of graphs and metric graphs}
\begin{defin}
    We say that a graph is \textbf{$k\mbox{-}$edge connected} if, by removing any $k-1$ edges, we still get a connected graph.
\end{defin}

\begin{defin}
    We say that a metric graph is \textbf{$k\mbox{-}$edge connected} if its canonical model is $k\mbox{-}$edge connected.
\end{defin}

\begin{defin}
    A \textbf{$k$-edge cut} of a graph is a set of $k$ distinct edges such that their removal disconnects the graph, but the removal of any proper subset of it does not.
    A $1$-edge cut is called a \textbf{bridge}. A $k$-edge cut of a metric graph is a $k$-edge cut of its canonical model.
\end{defin}

\begin{rema}
    Recall that a graph is $k$-edge connected if and only if given any pair of distinct vertices $u,v\in V(G)$, there are at least $k$ edge-disjoint paths connecting $u$ and $v$. 
    By Menger's theorem on edge connectivity (see~\cite[Theorem 9.7]{BM08}), $k$-edge connectivity is equivalent to the following property: an edge cut separating $u$ and $v$ must have cardinality at least $k$.
\end{rema} 

Here are different definitions of gonality for graphs and for metric graphs.
\begin{defin}\label{def:gonal}
    A graph $G$ is \textbf{$d$-gonal} if there is a non-degenerate harmonic morphism of degree $d$ from $G$ to a tree.
\end{defin}

\begin{defin}\label{def:metric_gonal}
    A metric graph $\Gamma$ is \textbf{$d$-gonal} if there is a non-degenerate harmonic morphism of degree $d$ from a tropical modification $\Gamma'$ of $\Gamma,$ to a metric tree. We will say that $\Gamma$ is \textbf{strictly $d$-gonal} if the tropical modification $\Gamma'$ coincides with $\Gamma$, i.e., there is a non-degenerate harmonic morphism of degree $d$ from  $\Gamma$  to a metric tree.
\end{defin}

\begin{defin}\label{def: div_gonal}
    A metric graph is \textbf{divisorially $d$-gonal} if $W_d^1(\Gamma)\neq \emptyset.$
\end{defin}

In the following we will say that a (metric) graph is \textbf{hyperelliptic}, respectively \textbf{trigonal}, if it is $2$-gonal, respectively $3$-gonal.
Similarly, a metric graph is \textbf{divisorially hyperelliptic}, respectively \textbf{divisorially trigonal}, if $W_2^1(\Gamma)\neq \emptyset,$ respectively
$W_3^1(\Gamma)\neq \emptyset.$

Hyperelliptic metric graphs have been intensively studied in~\cite{MC} and~\cite{L}, where such objects were shown to share several similarities with (algebraic) hyperelliptic curves. In the trigonal case, some of those features don't hold anymore, as we now explain. 

In~\cite[Theorem 3.6]{L} it has been proved that, given a (divisorially) hyperelliptic metric graph $\Gamma,$ all divisors of degree $2$ and rank $1$ are linearly equivalent. 
This is a tropical version of Martens' theorem,~\cite[Theorem 5.1]{ACGH}, for algebraic curves.

Notice that this is in contrast with the trigonal case.
A smooth irreducible non-hyperelliptic algebraic curve of genus $g\geq5$ cannot have two distinct $g_3^1$'s but
instead this is not true for divisorially trigonal metric graphs: we include an example to illustrate this phenomenon in Figure~\ref{fg:trig_no_unique}. The figure shows infinitely many divisors of degree $3$ and rank $1$ which are not linearly equivalent to each other. In Figure~\ref{fg:theta} we also provide an example, of genus $4$, where the divisors are finite, up to linear equivalence. This is not in contrast with the smooth algebraic case; indeed, a smooth irreducible non-hyperelliptic algebraic curve of genus $4$ whose canonical model lies on a smooth quadric surface has two $g^1_3$'s. 

\begin{figure}[ht]
\begin{tikzcd}
\begin{tikzpicture}  
    \draw[blue] (-1,0.5) -- (0, 0.5);
    \draw[verde] (1,0.5) -- (0, 0.5);
    \draw (-0.5,0.5) node[anchor=south] {$l_1$};
    \draw (0.5,0.5) node[anchor=south] {$l_2$};
    \draw[blue] (-1,0) -- (0, 0);
    \draw[verde] (0,0) -- (1, 0);
    \draw (-1.5,-0.5) -- (2,-0.5);
    \draw (-1.5,-0.5) node[anchor=north] {$x$};
    \draw (2,-0.5) node[anchor=north] {$y$};
    \draw (0,0.5) -- (0,0);
    \divisor{0,0}\divisor{0,0.5}
    \draw (-1,0.5) -- (-1,0);
    \draw (-1,0.5) -- (-1.5,-0.5);
    \draw (-1,0) -- (-1.5,-0.5);
    \vertex{-1,0}\vertex{-1,0.5}\vertex{-1.5,-0.5}
    \draw (1,0.5) --(2,-0.5);
    \draw (1,0) --(2,-0.5);
    \draw (1,0.5) --(1,0);
    \vertex{1,0}\vertex{1,0.5}\vertex{2,-0.5}\divisor{1,-0.5}
\end{tikzpicture}&
\begin{tikzpicture}
    \draw[blue] (-1,0.5) -- (0, 0.5);
    \draw[verde] (1,0.5) -- (0, 0.5);
    \draw (-0.5,0.5) node[anchor=south] {$l_1$};
    \draw (0.5,0.5) node[anchor=south] {$l_2$};
    \draw[blue] (-1,0) -- (0, 0);
    \draw[verde] (0,0) -- (1, 0);
    \draw (-1.5,-0.5) -- (2,-0.5);
    \draw (-1.5,-0.5) node[anchor=north] {$x$};
    \draw (2,-0.5) node[anchor=north] {$y$};
    \draw (0,0.5) -- (0,0);
    \divisor{0,0}\divisor{0,0.5}
    \draw (-1,0.5) -- (-1,0);
    \draw (-1,0.5) -- (-1.5,-0.5);
    \draw (-1,0) -- (-1.5,-0.5);
    \vertex{-1,0}\vertex{-1,0.5}\vertex{-1.5,-0.5}
    \draw (1,0.5) --(2,-0.5);
    \draw (1,0) --(2,-0.5);
    \draw (1,0.5) --(1,0);
    \vertex{1,0}\vertex{1,0.5}\vertex{2,-0.5}\divisor{0,-0.5}
    \end{tikzpicture}
\end{tikzcd}\caption{A divisorially trigonal metric graph of genus $5$ with two non-linearly equivalent divisors of degree $3$ and rank $1$. Edges of the same color have the same length and $d(x,y)\geq l_1+l_2$. The top $2$ red points and any point $p\in xy$ such that $d(x,p)\geq l_1$ and $d(y,p)\geq l_2$ yields a divisor of degree $3$ and rank $1$. In particular there are infinitely many such divisors.}\label{fg:trig_no_unique}
\end{figure}

\begin{figure}[ht]
\begin{tikzcd}
\begin{tikzpicture}  
    \draw(-1,0)--(0,0);
    \draw(-1,0) to [out=45, in=135](0,0);
    \draw(-1,0) to [out=-45, in=-135](0,0);
    \draw(1,0)--(0,0);
    \draw(0,0) to [out=45, in=135](1,0);
    \draw(0,0) to [out=-45, in=-135](1,0);
    \divisor{0,0}\divisor{1,0}\divisor{-1,0}
\end{tikzpicture}&
\begin{tikzpicture}
    \draw(-1,0)--(0,0);
    \draw(-1,0) to [out=45, in=135](0,0);
    \draw(-1,0) to [out=-45, in=-135](0,0);
    \draw(1,0)--(0,0);
    \draw(0,0) to [out=45, in=135](1,0);
    \draw(0,0) to [out=-45, in=-135](1,0);
    \divisor{-1,0}\vertex{0,0}\vertex{1,0}
    \draw(-1,0) node[anchor=east] {3};
    \end{tikzpicture}
\end{tikzcd}\caption{A divisorially trigonal metric graph with two non-linearly equivalent divisors of degree $3$ and rank $1$. Up to linear equivalence, there are the only two divisors of degree $3$ and rank $1$.}\label{fg:theta}
\end{figure}

Let us also recall Chan's main result on hyperelliptic metric graphs. 
\begin{theo}[{\cite[Theorem 3.12]{MC}}]
    Let $\Gamma$ be a metric graph with no points of valence $1$, and let $(G_{-},l_{-})$ denote its canonical loopless model. Then the following are equivalent:
    \begin{itemize}
        \item[(i)] $\Gamma$ is divisorially hyperelliptic.
        \item[(ii)] There exists an involution $i : G_{-} \to G_{-}$ such that $G_{-}/i$ is a tree.
        \item[(iii)] There exists a non-degenerate harmonic morphism of degree 2 from $G_{-}$ to a tree, or $|V (G_{-})| = 2.$
    \end{itemize}
\end{theo}

This shows that a metric graph is (divisorially) hyperelliptic if and only if the underlying graph of its canonical model is hyperelliptic, but we can actually say more on the non-degenerate harmonic morphism $\varphi$ of degree $2$ from $G_{-}$ to a tree $T.$

From~\cite[Theorem 3.2, Claim 3.3]{MC}, the involution $i$ is an isometry, therefore edges sent to the same edge of the tree have the same length on the metric graph. Then we can define a length function $l_T$ on $E(T),$ such that $l_T(\varphi(e))=l_0(e)$ for any $e\in E(G_{-})$ with $\varphi(e)\in E(T).$

In particular, this shows that, given a metric graph $\Gamma$ with a canonical loopless model $(G_{-},l_{-})$ such that $|V (G_{-})|>2,$ the following holds:
$$\Gamma \text{ is divisorially hyperelliptic }\Leftrightarrow \Gamma \text{ is strictly hyperelliptic}.$$

We would like to extend this relation to the trigonal case. In the case of (loopless) graphs, this problem has been considered by  Aidun, Dean, Morrison,  Yu and Yuan in~\cite{ADMYY}, where the authors prove that, for a $3$-edge connected graph $G$, the existence of a non-degenerate harmonic morphism of degree $3$ from $G$ to a tree is equivalent to the existence of a divisor in $G$ of degree $3$ and rank at least $1$\footnote{Divisors on graphs are supported at the vertices and their ranks coincide with the ranks of the associated metric graphs with all edges of the same length (see for instance~\cite{BN07} and~\cite{AC}).}.
However, things are much more complicated when we consider metric graphs, as we will now discuss.

On the positive side, we can prove that the easier implication holds in general for metric graphs and any $d$; i.e., a metric graph that is strictly $d$-gonal is also divisorially $d$-gonal. 

\begin{lemm}\label{lm:Lemma_preim}
    Let $\Gamma=(G,l)$ and $\Gamma_T=(T,l_T)$ be metric graphs with $T$ a tree and let $\varphi:(G,l)\rightarrow (T,l_T)$ be a non-degenerate harmonic morphism of degree $d$. Then, for any point $t\in\Gamma_T$, $D:=\varphi^*(t)\in W_d^1(\Gamma).$
\end{lemm}

\begin{proof}
The proof follows the same ideas of~\cite[Proof of Theorem 3.2]{MC}. 

Let $t\in\Gamma_T$ and $D:=\varphi^*(t)=\sum_{i=1}^d x_i$ for some $x_i\in\Gamma$. The divisor $D$ has clearly degree $d$. It suffices to prove that it has rank greater or equal to $1$.
Let $w\in\Gamma.$ Since $\operatorname{Jac}(T)=1,$ we have $\varphi(w)\sim t$ and $D\sim \varphi^*(\varphi(w))=m_{{\varphi}}(w)\cdot w+ E$
for some effective divisor $E.$

If $w\in V(G)$ or $w\in\mathring{e},$ for some $e\in E(G)$ such that $\varphi(e)\in E(T),$ then, by~\eqref{hor_mult} in Definition~\ref{def_metric}, $m_{{\varphi}}(w)>0$ and $D-w\sim E'$ for some effective divisor $E'.$

If instead $w\in\mathring{e},$ for some $e\in E(G)$ such that $\varphi(e)\in V(T),$ then $m_{\varphi}(w)=0,$ but 
\begin{align*}
({\varphi})^*({\varphi}(w))&=\sum_{x\in \Gamma; {\varphi}(x)={\varphi}(w)}m_{{\varphi}}(x)\cdot x\\
&=\sum_{x\in V(G); {\varphi}(x)={\varphi}(w)}m_{{\varphi}}(x)\cdot x.\\
\end{align*}
Let $y,y'$ be the endpoints of $e,$ then $y+y'\sim w+w'$ for some $w'\in e$ and we have that also in this case $D-w\sim E''$ for some effective divisor $E'',$ which concludes the proof.
\end{proof}

From Remark~\ref{rk:tropmod} then it follows that a $d$-gonal metric graph is also divisorially $d$-gonal.

Here is instead an important example that shows that the opposite implication is trickier.

\begin{example}\label{not_trigonal}Consider the metric graph $(G_0,l_0)$ in Figure~\ref{fig:div_trigonal}. The figure represents a graph in 3D, where the dashed edges do not intersect the interior of any other edge and the lengths $l_1,l_2,l_3$ are all distinct.    
The red points define a divisor of degree 3 and rank $1$. Observe that there exists a non-degenerate harmonic morphism of graphs of degree $3$ from $G_0$ to the tree $T=K_2$, namely the one sending the three horizontal edges $e_1,e_2,e_3$ to the same edge and contracting everything else. 

However this morphism does not induce a non-degenerate harmonic morphism of degree $3$ from $(G_0,l_0)$ to any model $(T,l_T):$ by Remark~\ref{rk:lengths}, any length function $l_T$ on $T$ would require $l_0(e_1)=l_0(e_2)=l_0(e_3),$ which is a contradiction.

We do obtain instead a morphism of metric graphs which is non-degenerate, harmonic and of degree $3$ if instead we consider the models $(G,l)$, $(T',l_{T'})$ as in Figure~\ref{fig:div_trigonal_ref}.
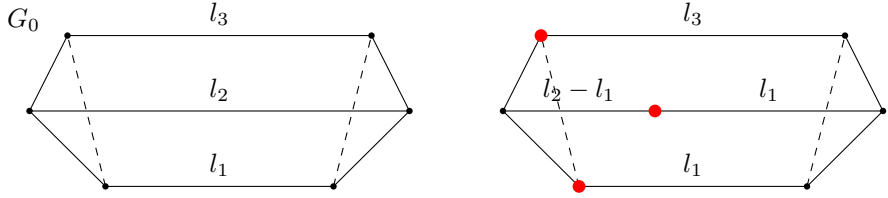
\begin{figure}[ht]
\begin{tikzcd}
\begin{tikzpicture}
   \draw (-0.7,2.2) node[anchor=east] {$G_0$};      
   \path[draw] (0,0) -- +(3, 0);
    \path[draw] (-1,1) -- +(5, 0);
        \path[draw] (-0.5,2) -- +(4,0);
        \draw [dashed] (0,0) -- (-0.5,2);
        \draw [dashed] (3,0) -- (3.5,2);
        \draw (0,0) -- (-1,1);
        \draw (-1,1) -- (-0.5,2);
        \draw (3.5,2) -- (4,1);
        \draw (3,0) -- (4,1);
    \draw (1.5,0) node[anchor=south] {$l_1$};
    \draw (1.5,1) node[anchor=south] {$l_2$};
    \draw (1.5,2) node[anchor=south] {$l_3$};
    	\foreach \i in {0,3} {
    	    \vertex{\i, 0}
    	    }
        \foreach \i in {-1,4} {
    	    \vertex{\i, 1}
    	    }
        \foreach \i in {-0.5,3.5} {
    	    \vertex{\i, 2}
    	    }
\end{tikzpicture}&
\begin{tikzpicture}
        \path[draw] (0,0) -- +(3, 0);
    	\path[draw] (-1,1) -- +(5, 0);
        \path[draw] (-0.5,2) -- +(4,0);
        \draw [dashed] (0,0) -- (-0.5,2);
        \draw [dashed] (3,0) -- (3.5,2);
        \draw (0,0) -- (-1,1);
        \draw (-1,1) -- (-0.5,2);
        \draw (3.5,2) -- (4,1);
        \draw (3,0) -- (4,1);
    	\foreach \i in {0,3} {
    	    \vertex{\i, 0}
    	    }
        \foreach \i in {-1,4} {
    	    \vertex{\i, 1}
    	    }
        \foreach \i in {-0.5,3.5} {
    	    \vertex{\i, 2}
    	    }
        \divisor{1, 1}
        \divisor{0, 0}
        \divisor{-0.5,2}
    \draw (1.5,0) node[anchor=south] {$l_1$};
    \draw (0,1) node[anchor=south] {$l_2-l_1$};
    \draw (2.5,1) node[anchor=south] {$l_1$};
    \draw (1.5,2) node[anchor=south] {$l_3$};
    \end{tikzpicture}
\end{tikzcd}\caption{A metric graph $(G_0,l_0)$ with a divisor of degree 3 and rank 1.}\label{fig:div_trigonal}
\end{figure}

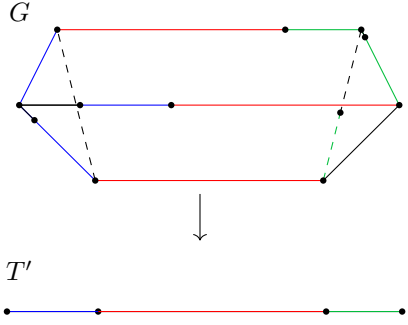
\begin{figure}[ht]
\begin{tikzcd}
\begin{tikzpicture}
        \draw (-0.7,2.2) node[anchor=east] {$G$};
        \path[draw,red] (0,0) -- +(3, 0);
        \draw (-1,1) -- (-0.2,1);
        \draw[blue] (1,1) -- (-0.2,1);
        \draw[red] (4,1) -- (1,1);
        \path[draw, red] (-0.5,2) -- +(3,0);
        \draw [verde] (2.5,2) -- (3.5,2);
        \draw [dashed] (0,0) -- (-0.5,2);
        \draw [dashed,verde] (3,0) -- (3.225,0.9);
        \draw[blue] (0,0) -- (-1,1);
        \draw[blue] (-1,1) -- (-0.5,2);
        \draw[verde] (3.5,2) -- (4,1);
        \draw (3,0) -- (4,1);
    \vertex{1, 1}
    \vertex{0,0}
    \vertex{-0.5,2}
    \vertex{3, 0}\vertex{4,1}\vertex{2.5,2}
    \vertex{3.5, 2}
    \vertex{3.55,1.9}
    \vertex{3.225,0.9}
    \draw[dashed](3.5, 2)--(3.225,0.9);
    \draw(3.5, 2)--(3.55,1.9);
    \vertex{-1,1}\vertex{-0.2,1}\vertex{-0.8,0.8}
    \draw (-1,1)--(-0.2,1);
    \draw (-1,1)--(-0.8,0.8);
    \end{tikzpicture}\arrow[d]\\
    \begin{tikzpicture}
    \draw (-0.7,0.5) node[anchor=east] {$T'$};
    \draw (-1.2,0)--(4,0);
    \vertex{0,0}
    \draw[red](0,0)--(3,0);
    \draw[blue](-1.2, 0)--(0,0);
    \draw[verde](4,0)--(3,0);
    \vertex{3, 0}
    \vertex{4, 0}
    \vertex{-1.2, 0}
    \end{tikzpicture}
\end{tikzcd}\caption{A non-degenerate harmonic morphism of degree $3$, from a refinement $(G,l)$ of $(G_0,l_0)$ to a metric tree $(T',l_{T'})$.}\label{fig:div_trigonal_ref}
\end{figure}

\end{example}

This example suggests that the divisorial trigonality of a metric graph $\Gamma$ is not always equivalent to graph trigonality of its canonical model $G_0$, but instead it might be equivalent to the trigonality of a graph defining a metric graph which is a tropical modification of $\Gamma$.

\begin{defin}
    Let $D\in W_d^1(\Gamma),$ and let $(G_0,l_0)$ be the canonical model of $\Gamma.$ An \textbf{admissible representative} for $D$ with respect to $x\in V(G_0)$ is a divisor $D_x\sim D$ such that $D_x=x+\sum_{i=1}^{d-1}x_i,$ with $x_i,x_j$ not in the interior of the same edge of $G_0$ for $i\neq j$.
\end{defin}

It follows from the following remarks that, given $D\in W_d^1(\Gamma)$, we can always find an admissible representative for $D$.

\begin{rema}\label{rk:common_edge}
    Let $D$ be an effective divisor of degree $d\geq 2$. We can always write $D\sim \sum_{i=1}^d y_i,$ with $y_i\in \Gamma$ such that no two points belong to the interior of the same edge of $G_0$. In fact, if $D\sim \sum_{i=1}^d x_i,$ then for any pair of points $x_i,x_j$ both contained in the interior of the same edge $e,$ we can consider the following rational function. Denote by $v_i,v_j$ the endpoints of $e,$ where $v_i$ is closest to $x_i$ and $v_j$ to $x_j,$ and let $f$ be the rational function with slope $-1$ from both of $x_i,x_j$ towards $v_i,v_j$, respectively, on a path of length $\operatorname{min}\{d(x_i,v_i),d(x_j,v_j)\},$
    and constant everywhere else. This gives  $x_i+x_j\sim y_i+y_j$ where at least one of the points $y_i,y_j$ is a vertex.
\end{rema}

\begin{rema}\label{rk:vertices}
Let $D\in W_d^1(\Gamma)$. Then, for any $x\in \Gamma,$ we can always assume $D\sim x+\sum_{i=1}^{d-1}x_i$. Indeed, since $D$ has rank at least $1$, we have $D-x\sim E,$ with $E$ effective. Therefore it is sufficient to pick $\sum_{i=1}^{d-1}x_i$ as the support of $E.$ 
\end{rema}

Let us conclude the section with a remark on the rank of a divisor of degree $d.$
\begin{rema}\label{rk:rank1}
    In~\cite[Theorem 3.6]{AC} and~\cite[Theorem A.1]{L}, respectively, \emph{Riemann's theorem} and \emph{Clifford's theorem} have been generalized to metric graphs. In particular, given a metric graph $\Gamma$ of genus $g$ and a divisor $D\in\operatorname{Div}^d(\Gamma),$ with $d\geq 1$,
    \begin{enumerate}
        \item \emph{(Riemann's theorem)} if $\operatorname{deg}(D)\geq 2g-1$ then $\operatorname{rk}(D)=\operatorname{deg}(D)-g;$ 
        \item \emph{(Clifford's theorem)} if $\operatorname{deg}(D)<2g-2$ then $2\operatorname{rk}(D)\leq\operatorname{deg}(D)$.
    \end{enumerate}
    From {(1)}, if $g=1$ then $\operatorname{rk}(D)=d-1$ and if $g=2$ and $d\geq 3$ then $\operatorname{rk}(D)=d-2.$
    
    It follows also from {(2)}, that if $d=3$ and $g\geq3$ then $\operatorname{rk}D\leq 1.$
    From now on we will assume $g\geq 2$ and therefore $ W_3^1(\Gamma)= W_3^1(\Gamma)\setminus W_3^2(\Gamma).$
\end{rema}

\section{Trigonal and divisorially trigonal metric graphs: the \texorpdfstring{$3$}{3}-edge connected case}\label{sc:3_connected}

Let us consider $3$-edge connected metric graphs.
We first observe that none of such graphs can be hyperelliptic. In fact, it was proved  in~\cite[Lemma 5.3]{BN} that the edge connectivity of a hyperelliptic (combinatorial) graph must be at most $2$, and, by the correspondence in~\cite[Theorem 3.12]{MC}, the same holds for metric graphs.

Here is an alternative proof for metric graphs, using Dhar's burning algorithm.
\begin{lemm}\label{lm:no_hyp}
    Let $\Gamma=(G,l)$ be a metric graph with canonical model $\Gamma_0=(G_0,l_0)$ such that $|V(G_0)|\geq 3$. If $\Gamma$ is $3$-edge connected, then it is not (divisorially) hyperelliptic.
\end{lemm}

\begin{proof}
    By contradiction, assume $H\in W_2^1(\Gamma)\neq \emptyset.$
    From Remark~\ref{rk:common_edge}, we can consider $H=x+y,$ with $x,y$ not in the interior of the same edge. 
    Start Dhar's burning algorithm from any vertex $w\in V(G_0);$ $w\neq x,y$. By $3$-edge connectivity, there are three distinct paths from $w$ to $x$. Even if one of such paths stops at $y$, the other two will then reach $x$ and again $y$, burning the whole graph.
    \end{proof}

\begin{rema}\label{rk:tree_gon}
This also shows that if a $3$-edge connected metric graph $\Gamma$ is trigonal, then its \emph{tree gonality}~\cite{CD},\cite{DV}, i.e. the minimum degree of a tropical morphism from a tropical modification of $\Gamma$ to a metric tree, is $3.$ 
In fact, they define a tropical morphism as an harmonic morphism with no contractions satisfying the Riemann-Hurwitz inequality at any point.
We can indeed replace the harmonic morphism with its tropical modification with no contractions as in Proposition~\ref{prp: harm_contractions}, which is of degree $3$ and thus satisfies the Riemann-Hurwitz inequality, as observed in~\cite[Remark 6]{LC}. 
\end{rema}

In Lemma~\ref{lm:Lemma_preim}, we proved that a non-degenerate harmonic degree $3$ morphism $\varphi:(G,l)\to(T,l_T),$ with $T$ a tree, determines a divisor in $W_3^1(\Gamma)$, where $\Gamma=(G,l).$ 

Instead, to determine a trigonal graph from a divisor in $W_3^1(\Gamma)$ is more difficult. We consider first the case without loops and then extend our construction to the general case with loops.

\subsection{The loopless case}\label{sc:loopless}
The relation between the definitions of gonality that we want to prove is the following.

\begin{theo}\label{th:Main_Theo}
    Let $\Gamma$ be a loopless $3\mbox{-}$edge connected metric graph with canonical model $(G_0,l_0).$ The following are equivalent.
    \begin{itemize}
        \item[A.] $|V(G_0)|=2,3$ or $\Gamma$ is strictly trigonal.
        \item[B.]  $\Gamma$ is divisorially trigonal.
    \end{itemize}
\end{theo}

\begin{rema}
     The above result has already been proved in terms of combinatorial graphs in~\cite{ADMYY}, for a $3$-edge connected graph $G$.
     In the metric case, one needs to consider also edge lengths, which adds a layer of complexity. However, if we restrict to the loopless case, the results in the metric and non-metric case are comparable.
     We will indeed see in this section that when $\Gamma$ is a $3\mbox{-}$edge connected metric graph without loops, if it is trigonal, then the tropical modification $\Gamma',$ from which a non-degenerate harmonic morphism of degree 3 to a tree exists, is such that $\Gamma'=\Gamma,$ i.e. no trees are added in this case.
\end{rema}

\begin{proof}[Proof of Theorem~\ref{th:Main_Theo}, $A.\Rightarrow B$ ]
The proof of this first implication when $|V(G_0)|>3$ and $\Gamma$ is trigonal follows from Lemma~\ref{lm:Lemma_preim}.

Instead, if $|V(G_0)|=2,$ define $D:=x+2y,$ s.t $x,y\in V(G_0)$ and $x\neq y.$
Then, for any $w\in \Gamma,$ set $d=\operatorname{min}\{d(x,w),d(y,w))\}$ and define $f_w:\Gamma\rightarrow \mathbb{R}$ with slope $-1$ from both $x,y$ along a path of length $d$ over the edge containing $w$ and constant everywhere else.
This gives $D-w+\operatorname{div}(f_w)=z+w'$, with $z\in\{x,y\}$ and $w'\in \Gamma$ which proves that $D$ has rank $1.$

Finally, if $|V(G_0)|=3,$ define instead $D:=x+y+z,$ s.t $x,y,z\in V(G_0)$ and $x\neq y\neq z.$
By a similar procedure as the above, we can check that also in this case $D$ has rank 1.
\end{proof}

We will now focus on the implication $B.\Rightarrow A.$ and explain the strategy of the proof.
Throughout the rest of the section, $\Gamma$ will be a $3\mbox{-}$edge connected metric graph with canonical model $(G_0,l_0)$ with $|V(G_0)|>3.$

Take $D\in W_3^1(\Gamma)$ and let us recall that given $x\in V(G_0)$, an admissible representative for $D$ with respect to $x$ is a divisor 
$D_x=x+x_1+x_2$ with $x_1,x_2$ not in the interior of the same edge of $G_0.$ We will start by proving that such admissible representatives are unique in the $3$-edge connected case.

\begin{lemm}\label{lm:disj_rep}
    Let $D\in W^1_3(\Gamma),$ with $\Gamma$ a $3$-edge connected metric graph.
    For any $x\in V(G_0)$ there is a unique admissible representative $D_x$ for $D.$
\end{lemm}

\begin{proof}
    Fix $x\in V(G_0)$ and assume by contradiction that $D_x=x+x_1+x_2$ is linearly equivalent to $x+x_1'+x_2'$ with $\{x_1,x_2\},\{x_1',x_2'\}$ disjoint and $x_1',x_2'$ not in the interior of the same edge. If $x+x_1+x_2\sim x+x_1'+x_2'$ then $x_1+x_2\sim x_1'+x_2'$ which means that there exists a rational function $f$ such that $\operatorname{div}(f)=x_1+x_2-x_1'-x_2'.$

    Let us denote by $\mathbf m\subset\Gamma,$ $\mathbf M\subset\Gamma,$ the set of points of $\Gamma$ over which $f$ is minimized and maximized, respectively, and by  $\partial \mathbf m$,  $\partial \mathbf M$ their boundaries. 
    
    Notice that $3$-edge connectivity implies $2$-edge connectivity. Therefore by~\cite[Lemma 3.1. (i)]{MC} $\partial \mathbf m=\{x_1,x_2\}$ and $\partial \mathbf M=\{x_1',x_2'\}$ and there are precisely $2$ edges leaving $\partial \mathbf m$ one at each $x_i,$ and precisely $2$ edges leaving $\partial \mathbf M$ one at each $x_i'$. The minimizing and maximizing sets $\mathbf m$ and $\mathbf M$ cannot be disconnected, otherwise these edges would be disconnecting, which is impossible by our connectivity assumption.

    Notice that the loci $\mathbf m$ and $\mathbf M$ must contain vertices, otherwise the two pairs $x_1,x_2$ and $x_1',x_2'$ are all contained in the same edge, which is not possible by the definition of admissible representative.
    
    By $3$-edge connectivity, there must be a path connecting a certain vertex in $\mathbf m$ to a certain vertex in $\mathbf M,$ not passing through the above mentioned edges. 
    However, such path cannot exist because by the continuity of $f$, it should pass through points in $\partial \mathbf m$ and $\partial \mathbf M$.
\end{proof}

\begin{rema}\label{rk:disj_supp}
As a consequence of the proof of Lemma~\ref{lm:disj_rep}, any two distinct admissible representatives $D_x,$ $D_y$ of $D$ are such that $\operatorname{Supp}(D_x)\cap\operatorname{Supp}(D_y)=\emptyset.$ Indeed, in the proof, the fact that $x,y\in V(G_0)$ does not play any role.
\end{rema}

Take $D\in W_3^1(\Gamma)$ and consider the set of admissible representatives for $D:$ 
$$A_D:=\{D_x|\, D_x\sim D \text{ and $D_x$ is admissible with respect to $x\in V(G_0)$}\},$$ i.e. the set of all divisors which are linearly equivalent to $D$ and of the form $D_x=x+x_1+x_2$ with $x\in V(G_0)$ and $x_1,x_2$ not in the interior of the same edge of $G_0.$

Clearly, since $V(G_0)$ is finite, the above lemma also yields that $A_D$ is finite.

Given $D\in W_3^1(\Gamma),$ we want to construct a model $(G_D,l_D)$ which is a refinement of the canonical model $(G_0,l_0)$ and a morphism $\varphi_D:G_D\to T_D$ where $T_D$ is a tree.
Furthermore we will define a length function $l_{T_D}$ on $E(T_D)$ such that $\varphi_D$ induces a morphism of metric graphs $\tilde{\varphi}_D:(G_D,l_D)\rightarrow (T_D,l_{T_D}).$ 

Let us construct first $(G_D,l_D).$
Define $G_D$ as the refinement of $G_0$ obtained by adding a vertex at $x_1$ and $x_2$ (if they were not already vertices) for any admissible representative $x+x_1+x_2$ of $D$, and define the lengths of the new edges accordingly, as in Figure~\ref{fig:refin}.

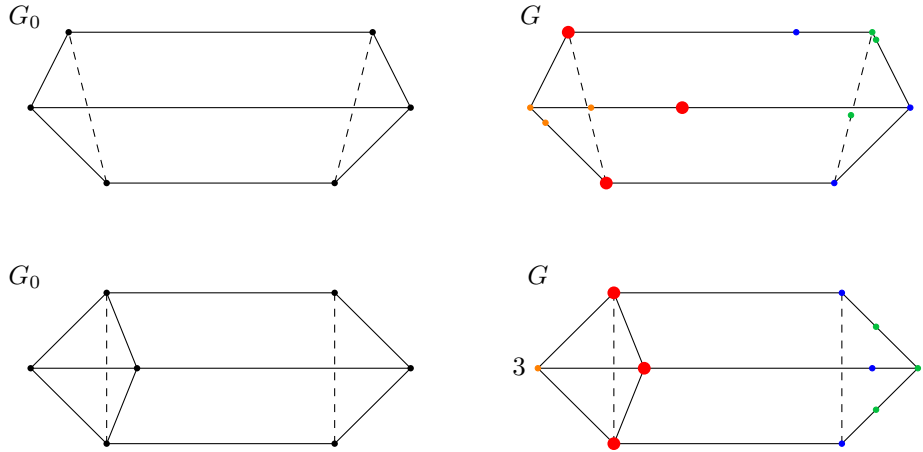
\begin{figure}[ht]
\begin{tikzcd}
\begin{tikzpicture}
   \draw (-0.7,2.2) node[anchor=east] {$G_0$};      
   \path[draw] (0,0) -- +(3, 0);
    	\path[draw] (-1,1) -- +(5, 0);
        \path[draw] (-0.5,2) -- +(4,0);
        \draw [dashed] (0,0) -- (-0.5,2);
        \draw [dashed] (3,0) -- (3.5,2);
        \draw (0,0) -- (-1,1);
        \draw (-1,1) -- (-0.5,2);
        \draw (3.5,2) -- (4,1);
        \draw (3,0) -- (4,1);
    	\foreach \i in {0,3} {
    	    \vertex{\i, 0}
    	    }
        \foreach \i in {-1,4} {
    	    \vertex{\i, 1}
    	    }
        \foreach \i in {-0.5,3.5} {
    	    \vertex{\i, 2}
    	    }
    \end{tikzpicture}&
        \begin{tikzpicture}
        \draw (-0.7,2.2) node[anchor=east] {$G$};
        \path[draw] (0,0) -- +(3, 0);
    	\path[draw] (-1,1) -- +(5, 0);
        \path[draw] (-0.5,2) -- +(4,0);
        \draw [dashed] (0,0) -- (-0.5,2);
        \draw [dashed] (3,0) -- (3.5,2);
        \draw (0,0) -- (-1,1);
        \draw (-1,1) -- (-0.5,2);
        \draw (3.5,2) -- (4,1);
        \draw (3,0) -- (4,1);
    \divisor{1, 1}
    \divisor{0,0}
    \divisor{-0.5,2}
    \filldraw[blue](3, 0)circle (1pt);
    \filldraw[blue](4,1)circle (1pt);
    \filldraw[blue](2.5,2)circle (1pt);
    \filldraw[verde](3.5, 2)circle (1pt);
    \filldraw[verde](3.55,1.9)circle (1pt);
    \filldraw[verde](3.22,0.9)circle (1pt);
    \filldraw[orange](-1,1)circle (1pt);
    \filldraw[orange](-0.2,1)circle (1pt);
    \filldraw[orange](-0.8,0.8)circle (1pt);
    \end{tikzpicture}\\
    \begin{tikzpicture}
   \draw (0.3,1.2) node[anchor=east] {$G_0$};
    \path[draw] (1,1) -- +(3, 0);
    \path[draw] (0,0) -- +(5, 0);
    \path[draw] (1,-1) -- +(3,0);
    \draw (4,1) -- (5, 0);
    \draw (4,-1) -- (5, 0);
    \draw [dashed] (4,1) -- (4,-1);
    \draw [dashed] (1,1) -- (1,-1);
    \draw (0,0) -- (1,1);
    \draw (0,0) -- (1,-1);
        \vertex{1.4,0}
        \vertex{0,0}
        \vertex{1,1}
        \vertex{1,-1}
        \vertex{4,1}
        \vertex{4,-1}
        \vertex{5,0}
        \draw (1.4,0) -- (1,1);
        \draw (1.4,0) -- (1,-1);
    \end{tikzpicture}&
    \begin{tikzpicture}
    \draw (0.3,1.2) node[anchor=east] {$G$};
    \path[draw] (1,1) -- +(3, 0);
    \path[draw] (0,0) -- +(5, 0);
    \path[draw] (1,-1) -- +(3,0);
    \draw (4,1) -- (5, 0);
    \draw (4,-1) -- (5, 0);
    \draw [dashed] (4,1) -- (4,-1);
    \draw [dashed] (1,1) -- (1,-1);
    \draw (0,0) -- (1,1);
    \draw (0,0) -- (1,-1);
        \vertex{1.4,0}
    \draw (0,0) node[anchor=east] {$3$};
    \filldraw[orange](0,0) circle (1pt);
        \draw (1.4,0) -- (1,1);
        \draw (1.4,0) -- (1,-1);
    \divisor{1, 1}
    \divisor{1.4,0}
    \divisor{1,-1}
    \filldraw[blue](4, 1)circle (1pt);
    \filldraw[blue](4.4,0)circle (1pt);
    \filldraw[blue](4,-1)circle (1pt);
    \filldraw[verde](4.45, 0.55)circle (1pt);
    \filldraw[verde](5,0)circle (1pt);
    \filldraw[verde](4.45,-0.55)circle (1pt);
    \end{tikzpicture}
\end{tikzcd}\caption{Examples of refinements $G$ of $G_0$, defined by all admissible representatives of $D\in W_3^1(\Gamma).$ In both cases, the divisor $D$ is defined by red points, and the admissible representatives by triples, with multiplicity, of the same color.}
\end{figure}\label{fig:refin}

For any admissible representative $D_x,$ we define a vertex $t_x \in V(T_D)$ and set
$\varphi_D(x)=\varphi_D(x_1)=\varphi_D(x_2)=t_x.$
From Remark~\ref{rk:disj_supp}, the supports of distinct admissible representatives are disjoint, therefore the map $\varphi_D|_{V(G_D)}:V(G_D)\to V(T_D)$ is well defined and $|V(T_D)|=|A_D|.$

In order to define the edges of the tree and the morphism over them, we need the following definition. 
\begin{defin}\label{def:consecutive}
Let $D_x,D_y\in A_D,$ such that $D_x\neq D_y.$ We say that they are \textbf{consecutive} if there exists an edge in $E(G_D)$ with an endpoint in $\operatorname{Supp}(D_x)$ and the other in $\operatorname{Supp}(D_y).$
\end{defin}

For any pair of vertices $t_x,t_y\in V(T_D)$, we set $t_xt_y\in E(T_D)$ if and only if the corresponding admissible representatives $D_x$ and $D_y$ are consecutive.
This defines the graph $T_D.$ We will later prove that $T_D$ is a tree.

Finally, let us construct the morphism $\varphi_D$ over the edges of $G_D.$
First of all, we set $\varphi_D(e)=t\in V(T_D)$ for any edge that has both endpoints in the support of the same admissible representative, which is sent to $t$ via $\varphi|_{V(G_D)}$.
For any pair of consecutive admissible representatives $D_x,D_y$, we set $\varphi_D(e)=t_xt_y$ for any $e\in E(G_D)$ such that $e$ has an endpoint in $\operatorname{Supp}(D_x)$ and the other in $\operatorname{Supp}(D_y).$
By construction the morphism $\varphi_D$ is a well defined morphism of graphs.

The fact that $\varphi_D$ is non-degenerate follows immediately from the following result.

\begin{lemm}\label{lm:existence}
    For any $x\in V(G_0),$ there exists $y\in V(G_0)$ such that $D_x,D_y$ are consecutive admissible representatives for the same divisor $D.$
\end{lemm}
\begin{proof}
If $D_x$ were not consecutive to any $D_y$, then
the subgraph of $G_D$ with vertices in $\operatorname{Supp}(D_x)$ would be a union of connected components of $G_D$, which is absurd.
\end{proof}

\begin{lemm}\label{lm:3_cut}
    Let $D_x,$ $D_y$ be two consecutive admissible representative of $D.$
    There are exactly  three distinct edges with an endpoint in $\operatorname{Supp}(D_x)$ and the other in $\operatorname{Supp}(D_y).$ Moreover, the three edges form a $3$-edge cut and have the same length in $(G_D,l_D).$
\end{lemm}

The proof relies on the following generalization of~\cite[Lemma 3.1. (i)]{MC}, of which we will not include a proof since it would follow word by word from the proof in loc.~cit.
\begin{lemm}\label{lm:div_3cut}
    Let $\Gamma$ be a $3$-edge connected metric graph and let $f$ be a rational function on $\Gamma$ such that 
    $$\operatorname{div}(f)=x_1+x_2+x_3-y_1-y_2-y_3,$$
    with $\{x_i\}_{i=1,\dots,3}\subset \Gamma,$ $\{y_i\}_{i=1,\dots,3}\subset \Gamma,$ such that $\{x_i\}_{i=1,\dots,3}\cap \{y_i\}_{i=1,\dots,3}=\emptyset.$
    Let~$\mathbf m,\mathbf M$ denote the sets of points in $\Gamma$ over which $f$ is minimized, maximized, respectively, and by $\partial\mathbf m,\partial\mathbf M$ their boundaries.
    Then $\partial\mathbf m=\{x_i\}_{i=1,\dots,3}$, $\partial\mathbf M=\{y_i\}_{i=1,\dots,3}$ and there are precisely three edges leaving $\partial\mathbf m,$ one at each $x_i,$ and precisely three edges leaving $\partial\mathbf M,$ one at each $y_i.$  
\end{lemm}
\begin{proof}[Proof of Lemma~\ref{lm:3_cut}]
    Write $D_x=x+x_1+x_2$ and $D_y=y+y_1+y_2.$
    From the linear equivalence $D_x\sim D_y,$ there exists a rational function 
    $f$ such that $\operatorname{div}(f)=x+x_1+x_2-y-y_1-y_2,$ where the sets $\{x,x_1,x_2\}, \{y,y_1,y_2\}$ are disjoint from Remark~\ref{rk:disj_supp}. Since $D_x,D_y$ are consecutive, the three edges leaving $\partial\mathbf{m}$, defined as in Lemma~\ref{lm:div_3cut}, are precisely the three edges leaving $\partial\mathbf{M}.$
    These three edges have to be a $3$-edge cut: they are the unique edges over which $f$ is non-constant therefore their removal has to disconnect the graph. Finally, they also have to have equal length because $f$ is continuous. 
\end{proof}

\begin{rema}\label{rk:multiplicity}
    Let $D_x=x+x_1+x_2$ be an admissible divisor, and consider the $3$-edge cut defined by any other admissible divisor that is consecutive to $D_x.$ It is a consequence of Lemmas~\ref{lm:existence},~\ref{lm:3_cut} and ~\ref{lm:div_3cut} that the number of edges in such a $3$-edge cut and incident to $z\in \operatorname{Supp}(D_x)$ is equal to $D_x(z).$
\end{rema}

We assign to any edge $e$ not contracted via $\varphi_D$  the index $\mu_{\varphi}(e)=1.$ 
We also set $l_{T_D}(\varphi(e))=l_D(e)$ for any $e\in E(G_D)$ such that $\varphi (e)\in E(T_D)$. Then by Lemma~\ref{lm:3_cut} we have that for any $e'\in E(T_D)$ the edges in $\varphi^{-1}(e')$ are precisely three and they have same length in $\Gamma.$ This means that the relations~\eqref{index} are always satisfied and we have an induced morphism $\tilde{\varphi}_D:(G_D,l_D)\to(T_D,l_{T_D}).$

Let us now conclude with the proof of Theorem~\ref{th:Main_Theo}.
\begin{proof}[Proof of Theorem~\ref{th:Main_Theo}, $B.\Rightarrow A$ ]
Let $D\in W_3^1(\Gamma)$, where $W_3^1(\Gamma)=W_3^1(\Gamma)\setminus W_3^2(\Gamma)$ by Remark~\ref{rk:rank1}.

We use the above construction to construct a refinement $(G_D,l_D)$ of $(G_0,l_0)$ and a morphism of graphs $\varphi_D: G_D\to T_D$ which induces a morphism of metric graphs ${\varphi}_D:(G_D,l_D)\to(T_D,l_{T_D}).$ By abuse of notation we use the same notation for the induced morphism.
Let us also recall that, as a consequence of Lemma~\ref{lm:existence}, $\varphi_D$ is non-degenerate.

What is left to prove is that the graph $T_D$ is a tree and that ${\varphi}_D$ is harmonic of degree $3$.
By construction, $T_D=\varphi_D (G_D)$ and hence it is connected since $G_D$ is. Recall also that a graph is a tree if and only if any of its edge is a bridge. By construction any $e\in E(T_D)$ is such that the edges in $\varphi_D^{-1}(e)$ form a $3$-edge cut and therefore the removal of $e$ must disconnect $T_D.$

By definition, the morphism $\varphi_D$ is harmonic if, for any $x\in V(G_D),$
$$m_{\varphi_D}(x)=\sum_{\substack{e\in E_x(G)\\ \varphi_D(e)=e'}}\mu_{\varphi_D}(e)$$ is constant for any $e'\in E_{t}(T_D),$ where $t=\varphi_D(x)$.

Since by construction we have that $\mu_{\varphi_D}(e)=1$ for $e$ such that $\varphi_D(e)\in E(T_D),$ 
then harmonicity follows if we prove instead that the quantity 
$$|\{e\in E_x(G_D)|\,\varphi_D(e)=e'\}|$$
is constant for any $e'\in E_t(T_D).$

This is a consequence of Lemma~\ref{lm:3_cut}. In fact, from the construction of $\varphi_D,$ the quantity 
$|\{e\in E_x(G_D)|\,\varphi_D(e)=e'\}|$ is equal to the number of edges in the $3$-cut defined by any admissible representative consecutive to $D_{x'},$ where $D_{x'}$ is the unique admissible divisor such that $x\in\operatorname{Supp}(D_{x'}).$ From Remark~\ref{rk:multiplicity}, the number of edges is simply the multiplicity of $x$ in $D_{x'},$ which does not depend on $e'\in E_t(T_D).$

Clearly, the harmonic morphism will have degree $3$.
\end{proof}

Let us observe that if one considers a loopless graph which is $k$-edge connected and divisorially $k$-gonal for any $k\geq 3$, then one can repeat the above construction step by step. This yields the following.
\begin{coro}\label{}
    Let $\Gamma$ be a loopless $k\mbox{-}$edge connected metric graph with canonical model $(G_0,l_0).$ The following are equivalent.
    \begin{itemize}
        \item[A.] $|V(G_0)|=2,3,\dots,k$ or $\Gamma$ is strictly $k$-gonal.
        \item[B.]  $\Gamma$ is divisorially $k$-gonal.
    \end{itemize}
\end{coro}

\subsection{The case with loops}\label{sc:loops}

Let us now consider a metric graph $\Gamma$ which is still $3$-edge connected, but possibly with loops.

Given $D\in W_3^1(\Gamma),$ we would like to extend the construction of the non-degenerate harmonic morphism $\varphi_D$ in the previous subsection to a metric graph with loops.

In the hyperelliptic case in~\cite{MC}, the two halves of a loop, in the canonical loopless model, are sent to the same edge.
The same construction would not work in our case since this would yield a morphism of degree $2$ over the edges in the canonical loopless model defined by a loop, and of degree $3$ elsewhere, as in Figure~\ref{fig:loop1}. 
\begin{figure}[ht]
\begin{tikzcd}
\begin{tikzpicture}
        \draw (-0.2,1.2) node[anchor=east] {$G_{-}$};
        \path[draw][red] (-0.24,0) -- (2, 0);
    	\path[draw][red] (0,-1) -- (2, 0);
        \path[draw][red] (0,1) -- (2,0);
        \draw[dashed] (0,-1)--(0,1);
        \draw (0,-1) -- (-0.24,0);
        \draw (0,1) -- (-0.24,0);
        \draw[verde] (2,0) to [out=140, in=180] (2,1);
        \draw[verde] (2,0) to [out=40, in=360] (2,1);
	  \path[draw][blue] (4.42,0) -- (2, 0);
    	\path[draw][blue] (4.2,-1) -- (2, 0);
        \path[draw][blue] (4.2,1) -- (2,0);
        \draw [dashed] (4.2,1) -- (4.2,-1);
        \draw (4.2,-1) -- (4.42,0);
        \draw (4.2,1) -- (4.42,0);
    \vertex{2,1}
    \vertex{2,0}
    \vertex{-0.24,0}
    \vertex{0,1}
    \vertex{0,-1}
    \vertex{4.42,0}
    \vertex{4.2,-1}
    \vertex{4.2,1}
    \end{tikzpicture}\arrow[d,"\varphi"]\\
    \begin{tikzpicture}
    \draw (-0.44,0.2) node[anchor=east] {$T$};
    \draw[red](-0.24,0) -- (2, 0);
    \draw (0.88,0) node[anchor=north] {$e_1$};
    \draw (3.21,0) node[anchor=north] {$e_2$};
    \draw (2.5,0.5) node[anchor=east] {$e_3$};
    \draw[blue](4.42,0) -- (2, 0);
    \draw[verde] (2, 0)--(2.7,0.7);
    \vertex{2,0}
    \vertex{2.7,0.7}
    \vertex{-0.24,0}   
    \vertex{4.42,0}
    \end{tikzpicture}
\end{tikzcd}\caption{The degree of the morphism is not well-defined: we have   $\sum_{\protect\substack{e\in E(G_{-})\\ \varphi(e)=e_i}}\mu_{\varphi}(e)=3$ for $i=1,2$ but $\sum_{\protect\substack{e\in E(G_{-})\\ \varphi(e)=e_3}}\mu_{\varphi}(e)=2$.}\label{fig:loop1}
\end{figure}
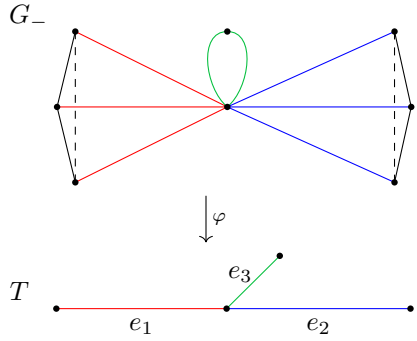

In order to get a morphism which has degree $3$ over the edges obtained from a loop, one could define instead a tropical modification, which is obtained from the canonical model by adding a vertex in the interior of a loop $e$, but such that the vertex divides the loop edge into two edges of lengths $\frac{l_0(e)}{3}$ and $\frac{2l_0(e)}{3}$. Unfortunately, this solution also does not work since the resulting morphism may not always be harmonic. An example is represented in Figure~\ref{fig:loop2}.

\begin{figure}[ht]
\begin{tikzcd}
\begin{tikzpicture}
        \draw (-1.2,1.2) node[anchor=east] {$G'$};
        \path[draw][red] (-1,1) -- (1, 0.5);
        \path[draw][red] (-1,0) -- (1, 0.5);
    	\path[draw][red] (-1.12,-0.7) -- (1, -0.7);
        \path[draw][blue] (1,-0.7) -- (3.6, -0.7);
        \path[draw][blue] (1,0.5) -- (3.5,1);
        \path[draw][blue] (1,0.5) -- (3.5,0);
        \draw (-1,1) -- (-1,0);
        \draw (-1,0) -- (-1.12,-0.7);
        \begin{scope}
		\clip (-1.5,0) rectangle ++(0.5,1);
		\draw (-1,0) ellipse (0.3 and 1);
	  \end{scope}
        \begin{scope}
		\clip (-1.5,-0.7) rectangle ++(0.38,0.7);
		\draw (-1.12,0) ellipse (0.18 and 0.7);
	  \end{scope}
        \begin{scope}
		\clip (3.5,0) rectangle ++(0.5,1);
		\draw (3.5,0) ellipse (0.3 and 1);
	  \end{scope}
        \begin{scope}
		\clip (3.6,-0.7) rectangle ++(0.38,0.7);
		\draw (3.6,0) ellipse (0.2 and 0.7);
	  \end{scope}
        \draw[verde] (1,0.5) to [out=140, in=180] (1,1.5);
        \draw[verde] (1,0.5) to [out=40, in=360] (1,1.5);
        \begin{scope}
		  \clip (0.7,0.5) rectangle ++(0.3,0.7);
		  \draw[thick,verde] (1,0.5) to [out=140, in=180] (1,1.5);
	  \end{scope}
        \draw (1.2,1.2) node[anchor=west] {$1$};
        \draw (0.8,0.8) node[anchor=east] {$2$};
        \draw (1,0.5) node[anchor=north] {$x$};
        \draw (3.5,1) -- (3.5,0);
        \draw (3.5,0) -- (3.6, -0.7);
    \vertex{-1,1}
    \vertex{-1,0}
    \vertex{1,0.5}
    \vertex{-1.12,-0.7}
    \vertex{1,-0.7}
    \vertex{3.5,1}
    \vertex{3.5,0}
    \vertex{3.6, -0.7}
    \vertex{0.72,1.2}
    \end{tikzpicture}\arrow[d,"\varphi"]\\
    \begin{tikzpicture}
    \draw (-0.32,0.2) node[anchor=east] {$T$};
    \draw[red](-0.12,0) -- (2, 0);
    \draw[blue](4.6,0) -- (2, 0);
    \draw[verde] (2, 0)--(3,1);
    \vertex{2,0}
    \vertex{3,1}
    \vertex{-0.12,0}   
    \vertex{4.6,0}\draw (0.88,0) node[anchor=north] {$e_1$};
    \draw (3.21,0) node[anchor=north] {$e_2$};
    \draw (2.5,0.5) node[anchor=east] {$e_3$};
    \end{tikzpicture}\end{tikzcd}\caption{We have $\sum_{e;\varphi(e)=e_3}\mu_{\varphi}(e)=3$ and $\sum_{e;\varphi(e)=e_1}\mu_{\varphi}(e)=\sum_{e;\varphi(e)=e_2}\mu_{\varphi}(e)=2$.}\label{fig:loop2}
\end{figure}

We will prove in this subsection that, in order to extend $\varphi_D$ over the possible loops of the graph to a morphism which is still non-degenerate and harmonic, we will have to consider a tropical modification which is obtained from the canonical loopless model by adding a leaf, for each loop, at a point that will be determined by $D$.
An example is shown in Figure~\ref{fig:loop3}.

\begin{figure}[ht]
\begin{tikzcd}
\begin{tikzpicture}
        \draw (-1.2,1.2) node[anchor=east] {$G'$};
        \path[draw][red] (-1,1) -- (1, 0.5);
        \path[draw][red] (-1,0) -- (1, 0.5);
    	\path[draw][red] (-1.12,-0.7) -- (1, -0.7);
        \path[draw][blue] (1,-0.7) -- (3.6, -0.7);
        \path[draw][blue] (1,0.5) -- (3.5,1);
        \path[draw][blue] (1,0.5) -- (3.5,0);
        \draw (-1,1) -- (-1,0);
        \draw (-1,0) -- (-1.12,-0.7);
        \begin{scope}
		\clip (3.5,0) rectangle ++(0.5,1);
		\draw (3.5,0) ellipse (0.3 and 1);
	  \end{scope}
        \begin{scope}
		\clip (3.6,-0.7) rectangle ++(0.38,0.7);
		\draw (3.6,0) ellipse (0.2 and 0.7);
	  \end{scope}
        \begin{scope}
		\clip (-1.5,0) rectangle ++(0.5,1);
		\draw (-1,0) ellipse (0.3 and 1);
	  \end{scope}
        \begin{scope}
		\clip (-1.5,-0.7) rectangle ++(0.38,0.7);
		\draw (-1.12,0) ellipse (0.18 and 0.7);
	  \end{scope}
        \draw[verde] (1,0.5) to [out=140, in=180] (1,1.5);
        \draw[verde] (1,0.5) to [out=40, in=360] (1,1.5);
        \draw (3.5,1) -- (3.5,0);
        \draw (3.5,0) -- (3.6, -0.7);
        \draw[verde] (1, -0.7)--(1.7,0);
    \vertex{-1,1}
    \vertex{-1,0}
    \vertex{1,0.5}
    \vertex{-1.12,-0.7}
    \vertex{1,-0.7}
    \vertex{3.5,1}
    \vertex{3.5,0}
    \vertex{3.6, -0.7}
    \vertex{1,1.5}
    \vertex{1.7,0}\end{tikzpicture}\arrow[d,"\varphi"]\\
    \begin{tikzpicture}
    \draw (-0.32,0.2) node[anchor=east] {$T$};
    \draw[red](-0.12,0) -- (2, 0);
    \draw[blue](4.6,0) -- (2, 0);
    \draw[verde] (2, 0)--(2.7,0.7);
    \vertex{2,0}
    \vertex{2.7,0.7}
    \vertex{-0.12,0}   
    \vertex{4.6,0}
    \end{tikzpicture}\end{tikzcd}\caption{A non-degenerate harmonic morphism of degree $3$.}\label{fig:loop3}
\end{figure}
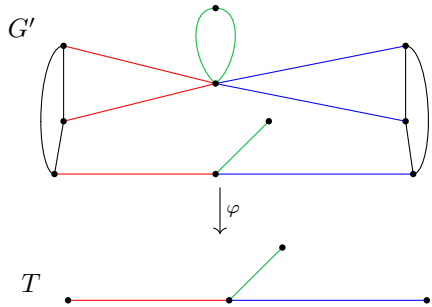

\begin{rema}\label{rk:supp_loop}
Let $D\in W_3^1(\Gamma)$ and denote by $\Gamma^\circ\subset\Gamma$ the metric sub-graph obtained by removing the loops in $\Gamma.$ Then we can always find $D'\in W_3^1(\Gamma)$ such that $D'\sim D$ and $\operatorname{Supp}(D')\subset \Gamma^\circ.$

In fact, take $D'$ admissible for $D$, then its support contains at most one point in the interior of a loop $e.$ If by contradiction there is precisely one point in $\operatorname{Supp}(D'),$ contained in the interior of $e$, we start Dhar's algorithm from any vertex of valence at least $3,$ not contained in $\operatorname{Supp}(D').$ The fire would either burn the other two points in the support and then the whole graph, or it can be controlled at the endpoint of~$e$, in at most 2 directions by the two other points in the support of $D'$.
However, by $3$-edge connectivity the valence at the endpoint of $e$ is at least $5,$ which means that the graph would burn also in this case.
\end{rema}

\begin{lemm}\label{lm:mult_loop}
Let $D\in W_3^1(\Gamma)$ and let $x$ be the endpoint of a loop $e$ in $\Gamma.$ 
Then $D$ is linearly equivalent to an admissible divisor $D_x= 2x+y$ for some $y\in \Gamma\setminus e.$
\end{lemm}
\begin{proof}
As observed in the previous remark, since the valence at $x$ is at least $5,$ then $x\in V(G_0),$ with $(G_0,l_0)$ the canonical model of $\Gamma.$ 

If by contradiction $D_x=x+y+z$ with $y,z\neq x$ and $y,z\not\in e$ from the above remark, then, starting Dhar's algorithm from any point in the interior of the loop, would burn $x$ and then the whole graph again by $3$-edge connectivity.
\end{proof}

\begin{lemm}\label{lm:vtx}
    Let $\Gamma$ be a metric graph with canonical loopless model $(G_{-},l_{-})$ such that $|V(G_{-})|>3.$
    Denote by $\Gamma^\circ$ the metric graph obtained from $\Gamma$ removing all loops and by $(G^{\circ},l^{\circ})$ its canonical model.

    If $\Gamma$ is divisorially trigonal and $|V(G^\circ)|\in \{2,3\},$ then $\Gamma^\circ$ is trigonal.
\end{lemm}

\begin{proof}
Let us first observe that $|V(G_{-})|=|V(G^\circ)|+k,$ with $k$ the number of loops in $\Gamma.$

Consider first the case $|V(G^\circ)|=2,$ then since $|V(G_{-})|>3,$ there are at least 2 loops in $\Gamma.$ 

Let $D\in W_3^1(\Gamma),$ then, by Remark~\ref{rk:supp_loop}, there exists $D'\in W_3^1(\Gamma^\circ)$ which by Lemma~\ref{lm:mult_loop} is such that $D'\sim 2x+y$ for any $x$ which is an endpoint of a loop in $\Gamma$ and $y\in \Gamma^\circ.$

By $3$-edge connectivity and using Dhar's burning algorithm, it is easy to see that all divisorially trigonal metric graphs $\Gamma$ with $|V(G^\circ)|=2$ have one of the form represented in Figure~\ref{fg:vtx2}, where all edges of the same color have same length. The morphism from $\Gamma^\circ$ sending these edges to a tree consisting of an edge of same length is clearly non-degenerate and harmonic of degree $3.$

\begin{figure}[ht]
\begin{tikzcd}
\begin{tikzpicture}
    \draw(0,0) to [out=170, in=190] (-0.5,0.5);
    \draw(0,0) to [out=95, in=10] (-0.5,0.5);
    \draw(0,0) to [out=10, in=350] (0.5,0.5);
    \draw(0,0) to [out=85, in=170] (0.5,0.5);
    \draw(0,0)[blue] to [out=240, in=120] (0,-1);
    \draw(0,0)[blue] to [out=300, in=60] (0,-1);
    \draw[blue](0,0)--(0,-1);
    \vertex{0,0}\vertex{0,-1}
    \end{tikzpicture}&
    \begin{tikzpicture}
        \draw(0,0) to [out=170, in=190] (-0.5,0.5);
    \draw(0,0) to [out=95, in=10] (-0.5,0.5);
    \draw(0,0) to [out=10, in=350] (0.5,0.5);
    \draw(0,0) to [out=85, in=170] (0.5,0.5);
    \draw(0,0)[blue] to [out=240, in=120] (0,-1);
    \draw(0,0)[blue] to [out=330, in=50] (0.2,-0.8);
    \draw(0,-1) to [out=30, in=50] (0.2,-0.8);
    \draw[blue](0,0)--(0,-1);
    \vertex{0,0}\vertex{0,-1}\vertex{0.2,-0.8}
    \end{tikzpicture}&\begin{tikzpicture}
    \draw(0,0) to [out=150, in=180] (0,0.5);
    \draw(0,0) to [out=30, in=0] (0,0.5);
    \draw(0,-1) to [out=330, in=0] (0,-1.5);
    \draw(0,-1) to [out=210, in=180] (0,-1.5);
    \draw(0,0)[blue] to [out=240, in=120] (0,-1);
    \draw(0,0)[blue] to [out=300, in=60] (0,-1);
    \draw[blue](0,0)--(0,-1);
    \vertex{0,0}\vertex{0,-1}
    \end{tikzpicture}\end{tikzcd}
    \caption{All possible divisorially trigonal 
 and $3$-edge connected metric graphs with $2$ loops and 2 vertices in the canonical model. Graphs with more loops are obtained by adding as many loops at vertices of valence at least $3$. In the left picture, if no loops are added in the vertex in the bottom, the three edges are also allowed to have arbitrary lengths.}\label{fg:vtx2}
\end{figure}

If instead $|V(G^\circ)|=3$ then there is at least a loop in $\Gamma.$ 

As in the previous case, given $D\in W_3^1(\Gamma),$ then we consider $D'\in W_3^1(\Gamma^\circ)$ such that $D'\sim 2x+y$ for any $x$ which is an endpoint of a loop in $\Gamma$ and $y\in \Gamma^\circ.$

Similarly, by $3$-edge connectivity and using Dhar's burning algorithm, it is easy to see that all metric graphs $\Gamma$ with $|V(G^\circ)|=3$ are all represented in Figure~\ref{fg:vtx3}. Here all edges of the same color still have same length and the morphism sending them to an edge of same length is again non-degenerate, harmonic of degree $3.$
\end{proof}

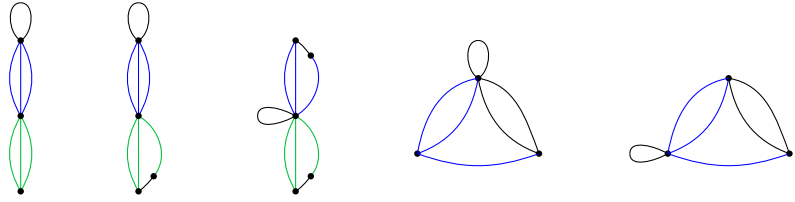
\begin{figure}[ht]
\begin{tikzcd}
\begin{tikzpicture}
    \draw(0,0) to [out=150, in=180] (0,0.5);
    \draw(0,0) to [out=30, in=0] (0,0.5);
    \draw(0,0)[blue] to [out=240, in=120] (0,-1);
    \draw(0,0)[blue] to [out=300, in=60] (0,-1);
    \draw[blue](0,0)--(0,-1);
    \draw(0,-1)[verde] to [out=240, in=120] (0,-2);
    \draw(0,-1)[verde] to [out=300, in=60] (0,-2);
    \draw[verde](0,-1)--(0,-2);
    \vertex{0,0}\vertex{0,-1}\vertex{0,-2}
    \end{tikzpicture}&
    \begin{tikzpicture}
    \draw(0,0) to [out=150, in=180] (0,0.5);
    \draw(0,0) to [out=30, in=0] (0,0.5);
    \draw(0,0)[blue] to [out=240, in=120] (0,-1);
    \draw(0,0)[blue] to [out=300, in=60] (0,-1);
    \draw[blue](0,0)--(0,-1);
    \draw[verde](0,-1)--(0,-2);
    \draw(0,-1)[verde] to [out=240, in=120] (0,-2);
    \draw(0,-1)[verde] to [out=330, in=50] (0.2,-1.8);
    \draw(0,-2) to [out=30, in=50] (0.2,-1.8);
\vertex{0,0}\vertex{0,-1}\vertex{0.2,-1.8}\vertex{0,-2}
    \end{tikzpicture}&\begin{tikzpicture}
    \draw(0,0)[blue] to [out=240, in=120] (0,-1);
    \draw(0,0) to [out=330, in=130] (0.2,-0.2);
\draw(0,-1)[blue] to [out=30, in=310] (0.2,-0.2);
    \draw[blue](0,0)--(0,-1);
    \draw[verde](0,-1)--(0,-2);
    \draw(0,-1)[verde] to [out=240, in=120] (0,-2);
    \draw(0,-1)[verde] to [out=330, in=50] (0.2,-1.8);
    \draw(0,-2) to [out=30, in=50] (0.2,-1.8);
    \draw(-0.5,-1) to [out=90, in=150] (0,-1);
    \draw(-0.5,-1) to [out=270, in=210] (0,-1);
    \vertex{0,0}\vertex{0,-1}\vertex{0.2,-1.8}\vertex{0,-2}\vertex{0.2,-0.2}
    \end{tikzpicture}&\begin{tikzpicture}
    \draw(0.8,-0.5) to [out=150, in=180] (0.8,0);
    \draw(0.8,-0.5) to [out=30, in=0] (0.8,0);
    \vertex{0,-1.5}\vertex{0.8,-0.5}\vertex{1.6,-1.5}
    \draw[blue] (0,-1.5) to [out=340, in=200] (1.6,-1.5);
    \draw[blue] (0,-1.5) to [out=80, in=190] (0.8,-0.5);
    \draw[blue] (0,-1.5) to [out=20, in=260] (0.8,-0.5);
    \draw[] (1.6,-1.5) to [out=110, in=350] (0.8,-0.5);
    \draw[] (1.6,-1.5) to [out=160, in=280] (0.8,-0.5);
    \end{tikzpicture}
    &\begin{tikzpicture}
    \draw(-0.5,-1.5) to [out=90, in=150] (0,-1.5);
    \draw(-0.5,-1.5) to [out=270, in=210] (0,-1.5);    \vertex{0,-1.5}\vertex{0.8,-0.5}\vertex{1.6,-1.5}
    \draw[blue] (0,-1.5) to [out=340, in=200] (1.6,-1.5);
    \draw[blue] (0,-1.5) to [out=80, in=190] (0.8,-0.5);
    \draw[blue] (0,-1.5) to [out=20, in=260] (0.8,-0.5);
    \draw[] (1.6,-1.5) to [out=110, in=350] (0.8,-0.5);
    \draw[] (1.6,-1.5) to [out=160, in=280] (0.8,-0.5);
    \end{tikzpicture}
    \end{tikzcd}
    \caption{All possible divisorially trigonal 
 and $3$-edge connected metric graphs with a loop and 3 vertices in the canonical model. The graphs with more loops are obtained by adding as many loops at the vertices incident to triples of edges of the same color, in the first three graphs on the left, or by adding loops at the vertices incident to (at least) two edges of the same color, in the two graphs on the right.}\label{fg:vtx3}
\end{figure}

The statement of Theorem~\ref{th:Main_Theo} can now be generalized by dropping the assumption on the loops.
\begin{theo}\label{th:Main_Theo_loops}
    Let $\Gamma$ be a $3\mbox{-}$edge connected metric graph with canonical loopless model $(G_{-},l_{-}).$ The following are equivalent.
    \begin{itemize}
        \item[A.] $|V(G_{-})|=2,3$ or $\Gamma$ is trigonal.
        \item[B.]  $\Gamma$ is divisorially trigonal.
    \end{itemize}
\end{theo}

\begin{proof}
The proof of $A.\Rightarrow B.$ again follows by Lemma~\ref{lm:Lemma_preim} if $|V(G_{-})|>3$.
If instead $|V(G_{-})|\leq3$, then define $D$ as in the proof of Theorem~\ref{th:Main_Theo} $A.\Rightarrow B.$ and then statement will follow.

To prove $B.\Rightarrow A.$, instead, 
let us consider $\Gamma^\circ$ the metric graph obtained from $\Gamma$ removing all loops.
If $\Gamma=\Gamma^\circ$ then the proof just follows from that of Theorem~\ref{th:Main_Theo}.

Otherwise, take $D\in W_3^1(\Gamma)$ and by Remark~\ref{rk:supp_loop} we may assume $\operatorname{Supp}(D)\subset\Gamma^\circ.$ Repeat the construction of $(G_D,l_D), (T_D,l_{T_D}),\varphi_D$ in Subsection~\ref{sc:loopless} over $\Gamma^\circ.$ 
Notice here from Lemma~\ref{lm:vtx} such a morphism exists even if the vertices in the canonical model of $\Gamma^\circ$ are less or equal than $3.$ 

Construct then the tropical modification $(G_D',l_D')$ of $\Gamma$ such that it coincides with 
$(G_D,l_D)$ over $\Gamma^\circ$ and over the loops it is obtained by placing a vertex at the midpoint of each loop and adding a leaf. The endpoints of such leaves are determined by the divisor $D.$

From Lemma~\ref{lm:mult_loop}, for any loop attached to a vertex $x,$ we have $D\sim 2x+y.$ Insert a leaf at $y$ and set the length of the leaf equal to half of that of the loop.

Define now $T_D'$ from $T_D$ by adding a leaf $t'$ for any added leaf $l$ in $G_D'$ and set $l_{T_D}'(t')=l_D'(l).$ Finally extend $\varphi_D$ over $l$ and the two edges in $G_D'$, obtained from the loop, sending them all to $t',$ and denote by $\varphi_D'$ this morphism.

Clearly $T_D'$ is still a tree since it has been obtained by adding leaves to a tree.

What is left to prove is that $\varphi_D':(G_D',l_D')\to(T_D',l_{T_D}')$ is again non-degenerate and harmonic of degree $3.$
The non-degeneracy follows by construction, since no contraction has been defined.

To prove that the morphism is again harmonic, it suffices to observe that, as $\varphi_D$, all edges not contracted have index $1$. Therefore, we need to check that for any $x\in V(G_{D}')$ the quantity $$|\{e\in E_x(G_D')|\,\varphi_D'(e)=e'\}|$$
is constant for any $e'\in E_{\varphi'_D(x)}(T_D'),$ and in particular for $x$ the endpoint of a loop or any added leaf.

If $x$ is the endpoint of a loop, from Lemma~\ref{lm:mult_loop}, we have $D\sim D_x=2x+y,$ with $D_x$ admissible. From the construction of the leaves, then we have that the above quantity is $2$ if $y\neq x,$ or $3$ if $x=y$. Thus it is the multiplicity of $x$ in $D_{x}$. If instead  we consider the endpoint of the added leaf the 
above quantity would be instead the multiplicity of $y$.
As in the case without loops, such multiplicities do not depend on $e'\in E_{\varphi'_D(x)}(T_D)$ and therefore the morphism is harmonic, of degree $3.$
\end{proof}

\section{Moduli of tropical trigonal curves and tropical trigonal covers}\label{sc:moduli}

Based on the results obtained in the preceding sections, we now aim to construct the moduli space of $3$-edge connected trigonal tropical curves. 
Our approach follows the lines of the construction of the moduli space of tropical hyperelliptic curves in~\cite{MC}, which we adapt to our situation.

In particular, given a trigonal abstract tropical curve, i.e. an abstract tropical curve $(G,w,l)$ whose underlying metric graph $(G^w,l)$ is divisorially trigonal, we want to associate to it some combinatorial data, and finally use it to construct the moduli space.

\begin{defin}\label{def:type}
    A $3$-edge connected \textbf{trigonal type} is a tuple $(G,w,\varphi)$ such that:
\begin{enumerate}
    \item $(G,w)$ is a $3$-edge connected stable graph,
    \item $\varphi:G_{\varphi}\rightarrow T$ is a morphism from a graph $G_{\varphi}$ with $(G_{\varphi})^{\operatorname{st}}=G^w$ to a tree $T$, with $\varphi$ non-degenerate harmonic  of degree $3$ and, for any $t\in V(T),$ $\varphi^{-1}(t)\cap V(G)\neq\emptyset$.
\end{enumerate}
Furthermore, for any $3$-edge connected trigonal type $(G,w,\varphi)$, we say that the pair $(G,w)$ is a $3$-edge connected \textbf{trigonal graph}.
\end{defin}

Notice that given a trigonal type $(G,w,\varphi)$, a refinement of $G$ would induce a refinement of $G_{\varphi}$ along with a non-degenerate harmonic of degree $3$ from it to a refinement of $T$.
However, such a refinement does not define a trigonal type because condition (2) would not be satisfied.

\begin{rema}\label{rk:rel_edges}
Let $(G,w,\varphi)$ be a $3$-edge connected trigonal type. The morphism $\varphi:G_{\varphi}\rightarrow T$ induces an equivalence relation $\sim_{\varphi}$ on $E(G_{\varphi}):$  we say that $e_1 \sim_{\varphi} e_2$ if and only if $e_1,e_2\in \varphi^{-1}(e)$ for $e\in E(T).$ 
\end{rema}

\begin{defin}
    Let $(G,w,\varphi)$ be a $3$-edge connected trigonal type. A \textbf{${\varphi}$-contraction} of $(G,w,\varphi)$ is the weighted contraction in $G_{\varphi}$ of a set of edges $S_{\varphi}$ which is the union of equivalence classes of $\sim_{\varphi}.$
\end{defin}

In other words, a $\varphi$-contraction of a $3$-edge connected trigonal type $(G,w,\varphi)$ is the weighted contraction of all edges in $G_{\varphi}$ sharing the same image via $\varphi,$ or the contraction of any edge in the pre-image of a vertex in the tree.

Let us now show that $3$-edge connected trigonal types are well-behaved with respect to $\varphi$-contractions.

\begin{prop}\label{prp:contr}
    Let $(G,w,\varphi)$ be a $3$-edge connected trigonal type, and let $\pi:G_{\varphi}\to G_{\varphi}/S_{\varphi}$ be a $\varphi$-contraction. 
    The morphism $\pi$ then determines a $3$-edge connected trigonal type $(\tilde G,\tilde w, \tilde \varphi)$ where $(G_{\varphi}/S_{\varphi})^{\operatorname{st}}=(\tilde G_{\tilde \varphi})^{\operatorname{st}}$.
\end{prop}

\begin{proof}
  Let us first notice that a $\varphi$-contraction is the finite composition of contractions of equivalence classes for $\sim_{\varphi}.$ Then we may assume that $\pi$ contracts only one equivalence class $S_{\varphi}.$

    By definition of trigonal type, $\varphi:G_{\varphi}\rightarrow T$ is a non-degenerate, harmonic morphism of degree $3$ from a graph $G_{\varphi}$ such that $(G_{\varphi})^{st}=G^w$ onto a tree $T$. 
    
    In each case, we set $S_G:=S_{\varphi}\cap E(G)$ and define $(\tilde G,\tilde w)$ as the weighted contraction of $S_G$ in $(G,w)$.
    \begin{itemize}
        \item If $S_{\varphi}=\{e\}$ and it is a loop edge, then by our assumption on $3$-edge connectivity we have $\varphi(e)=v\in V(T).$
        Let $\tilde{G}_{\tilde \varphi}=G_{\varphi}/S_{\varphi}$ and $\tilde \varphi: \tilde G_{\tilde\varphi}\to \tilde T$
        to be the morphism which coincides with $\varphi$ away from $e.$ Then clearly $\tilde \varphi$ is non-degenerate and harmonic of degree $3$.
        
        \item If $S_{\varphi}=\{e\}$ and it is a non-loop edge, we have again $\varphi(e)=v\in V(T).$
        Let $k$ be the number of edges which are parallel to $e,$ i.e. sharing the same endpoints, $x$ and $y$. The contraction of $e$ then yields $k$ loops at $\pi(e).$
        
        Let $\tilde{G}_{\tilde \varphi}$ be the graph obtained from $G_{\varphi}/S_{\varphi}$ by adding $k$ leaves at $\pi(e)$ if $m_\varphi(x)+m_\varphi(y)=3$ or at the other distinct vertex $z$; $\varphi(z)=v$. Let $\tilde T$ be the tree obtained by adding $k$ leaves at $v$. We then define $\tilde \varphi: \tilde G_{\tilde\varphi}\to \tilde T$
        to be the morphism which coincides with $\varphi$ away from $e$ and all its parallel edges, and which maps each pair of edges originated from a subdivided loop at $\pi(e)$ together with the added leaf to the corresponding leaf added to the tree $T$ at~$v$. One easily checks that the obtained morphism is non-degenerate, harmonic of degree $3,$ and $\tilde G^{\tilde w}=(\tilde G_{\tilde\varphi})^{\operatorname{st}}.$ 

        \item If $S_{\varphi}=\{e_1,e_2,e_3\},$ with (w.l.o.g.) $\{e_1,e_2\}\subset G_{\varphi}$ such that $\{e_1,e_2\}$ is obtained as a refinement of a loop of $G$. 
        Then necessarily $e_3$ is a leaf and we proceed analogously as in the case $S_{\varphi}=\{e\}$ is a loop edge.
        Here $\tilde G^{\tilde w}=G^w$ therefore the same morphism shows that $(\tilde G,\tilde w,\varphi)$ a trigonal type.  
        
        \item If finally $S_{\varphi}=\{e_1,e_2,e_3\},$ with $\{e_i,e_j\}\subset G_{\varphi}$ with no $\{e_i,e_j\}$ is obtained as a refinement of a loop of $G$. Then $\varphi(e_i)=e\in E(T)$ for any $i=1,2,3.$ Their contraction might again generate new loops, say $k$ loops. Let $\tilde G_{\tilde \varphi}$ be the graph obtained by adding $k$ leaves at $G_{\varphi}$, as in the second case, and contracting all~$e_i$'s, and let $\tilde T$ be the tree obtained by contracting $e$ to a vertex $v$ and by adding $k$ leaves at $v$. 
        Also in this case the morphism $\tilde\varphi:\tilde G_{\tilde \varphi}\rightarrow \tilde T,$ defined as in the second case, has the required properties and $(\tilde G_{\tilde\varphi})^{\operatorname{st}}=\tilde G^{\tilde w}.$ \qedhere 
    \end{itemize}    
\end{proof}

As a consequence of the above proposition, a $\varphi$-contraction then determines a map of $3$-edge connected trigonal types $(G,w,\varphi)\to(\tilde G,\tilde w, \tilde \varphi)$, which by abuse of notation we will denote in the same way. 

Moreover, it is easy to see that, when the $\varphi$-contraction $\pi:(G,w,\varphi)\to (\tilde G,\tilde w,\tilde \varphi)$ does not produce loops, it is then the datum of the following commutative diagram.
\begin{center}
\begin{tikzcd}
G_{\varphi} \arrow[r,"\pi"]  \arrow[d,"\varphi"] & \tilde{G}_{\tilde\varphi}\arrow[d,"\tilde\varphi"]\\
T\arrow[twoheadrightarrow]{r}  &\tilde T
\end{tikzcd}
\end{center}
In particular, a $\varphi$-contraction of an equivalence class of edges whose image via $\varphi$ is $e\in E(T)$ corresponds to the contraction of $e,$ otherwise it is the identity morphism and $\tilde T=T.$ An example is represented in Figure~\ref{fg:contr}.

Instead, if the $\varphi$-contraction produces loops we do not have a similar diagram: the source space of the morphism $\tilde\varphi$ will not be $\pi(G_{\varphi}),$ but a graph $\tilde G_{\tilde\varphi}$ such that $(\tilde G_{\tilde\varphi})^{\operatorname{st}}=\pi(G_{\varphi}).$

\begin{figure}[ht]
\centering
\begin{tikzcd}
    \begin{tikzpicture}
    \draw[blue](-0.5,0)--(0.5,0);
    \draw[blue](-0.4,0.2)--(0,1);
    \draw(-0.4,0.2)--(-0.5,0);\vertex{-0.4,0.2}\vertex{-0.5,0}
    \draw[blue](-0.3,0.8)--(0,2);\vertex{-0.3,0.8}\vertex{2.5,0}
    \draw(-0.3,0.8)--(-0.5,0);
\vertex{0.5,0}
    \draw[red](2.5,0)--(0.5,0);
    \draw[red](0,2)--(2,2);
    \draw[red](0,1)--(2,1);
    \draw(0,2)--(0,1);
    \draw(2,2)--(2,1);
    \draw(2.5,0)--(2,1);\draw(2.5,0)--(2,2);
    \foreach \i in {0,2} {\foreach \j in {1,2} {\vertex{\i, \j}}}
    \draw[->](1,-0.3) to (1,-0.8);
    \draw (1,-0.5) node[anchor=west] {$\varphi$};
    \draw (-0.2,2.2) node[anchor=south ] {$G'$};
    \draw (-0.7,-0.8) node[anchor=south ] {$T$};
    \draw[blue](-0.5,-1)--(0.5,-1);
    \draw[red](0.5,-1)--(2.5,-1);
    \vertex{-0.5,-1}\vertex{0.5,-1}\vertex{2.5,-1}
    \end{tikzpicture}&
    \begin{tikzpicture}
    \draw[->](0,1)to (1,1);   
    \draw[->](0,-1)to (1,-1); 
\draw (0.5,1) node[anchor=south] {$\pi$};

    \end{tikzpicture}&
    \begin{tikzpicture}
    \draw[red](2.5,0)--(0.5,0);
    \draw[red](0,2)--(2,2);
    \draw[red](0,1)--(2,1);
    \draw[dashed](0,2)--(0.5,0);
    \draw(0,1)--(0.5,0);
    \vertex{0.5,0}
    \draw(0,2)--(0,1);
    \draw(2,2)--(2,1);\vertex{2.5,0}
    \draw(2.5,0)--(2,1);\draw(2.5,0)--(2,2);
    \foreach \i in {0,2} {\foreach \j in {1,2} {\vertex{\i, \j}}}
    \draw[->](1.25,-0.3) to (1.25,-0.8);
    \draw (1.25,-0.5) node[anchor=west] {$\tilde \varphi_1$};
    \draw (-0.2,2.2) node[anchor=south ] {$\tilde G'_1$};
    \draw (0.05,-0.8) node[anchor=south ] {$\tilde T_1$};
    \draw[red](0.25,-1)--(2.25,-1);
    \vertex{0.25,-1}\vertex{2.25,-1}
    \end{tikzpicture}\\
        \begin{tikzpicture}
    \draw[blue](-0.5,0)--(0.5,0);
    \draw[blue](-0.4,0.2)--(0,1);
    \draw(-0.4,0.2)--(-0.5,0);\vertex{-0.4,0.2}\vertex{-0.5,0}
    \draw[blue](-0.3,0.8)--(0,2);\vertex{-0.3,0.8}\vertex{2.5,0}
    \draw(-0.3,0.8)--(-0.5,0);
\vertex{0.5,0}
    \draw[red](2.5,0)--(0.5,0);
    \draw[red](0,2)--(2,2);
    \draw[red](0,1)--(2,1);
    \draw(0,2)--(0,1);
    \draw(2,2)--(2,1);
    \draw(2.5,0)--(2,1);\draw(2.5,0)--(2,2);
    \foreach \i in {0,2} {\foreach \j in {1,2} {\vertex{\i, \j}}}
    \draw[->](1,-0.3) to (1,-0.8);
    \draw (1,-0.5) node[anchor=west] {$\varphi$};
    \draw (-0.2,2.2) node[anchor=south ] {$G'$};
    \draw (-0.7,-0.8) node[anchor=south ] {$T$};
    \draw[blue](-0.5,-1)--(0.5,-1);
    \draw[red](0.5,-1)--(2.5,-1);
    \vertex{-0.5,-1}\vertex{0.5,-1}\vertex{2.5,-1}
    \end{tikzpicture}
    &\begin{tikzpicture}
        \draw[->](0,1)to (1,1);   
    \draw[->](0,-1)to (1,-1); 
\draw (0.5,1) node[anchor=south] {$\pi$};
    \end{tikzpicture}&
    \begin{tikzpicture}
    \draw[blue](-0.5,0)--(0.5,0);
    \draw[blue](-0.4,0.2)--(0,1);
    \draw(-0.4,0.2)--(-0.5,0);\vertex{-0.4,0.2}\vertex{-0.5,0}
    \draw[blue](-0.3,0.8)--(0,2);\vertex{-0.3,0.8}\vertex{2.5,0}
    \draw(-0.3,0.8)--(-0.5,0);
\vertex{0.5,0}
    \draw[red](2.5,0)--(0.5,0);
    \draw[red](0,2)--(2,1);
    \draw[red](0,1)--(2,1);
    \draw(0,2)--(0,1);
    \draw(2.5,0)--(2,1);
    \draw(2.5,0)[out=90, in=350]to(2,1);
    \foreach \i in {0,2} {\vertex{\i, 1}}
    \vertex{0,2}
    \draw[->](1,-0.3) to (1,-0.8);
    \draw (1,-0.5) node[anchor=west] {$\tilde \varphi_2$};
    \draw (-0.2,2.2) node[anchor=south ] {$\tilde G'_2$};
    \draw (-0.7,-0.8) node[anchor=south ] {$\tilde T_2=T$};
    \draw[blue](-0.5,-1)--(0.5,-1);
    \draw[red](0.5,-1)--(2.5,-1);
    \vertex{-0.5,-1}\vertex{0.5,-1}\vertex{2.5,-1}
    \end{tikzpicture}
    \end{tikzcd}\caption{Two examples of $\varphi$-contractions from the same graph.}\label{fg:contr}
\end{figure}
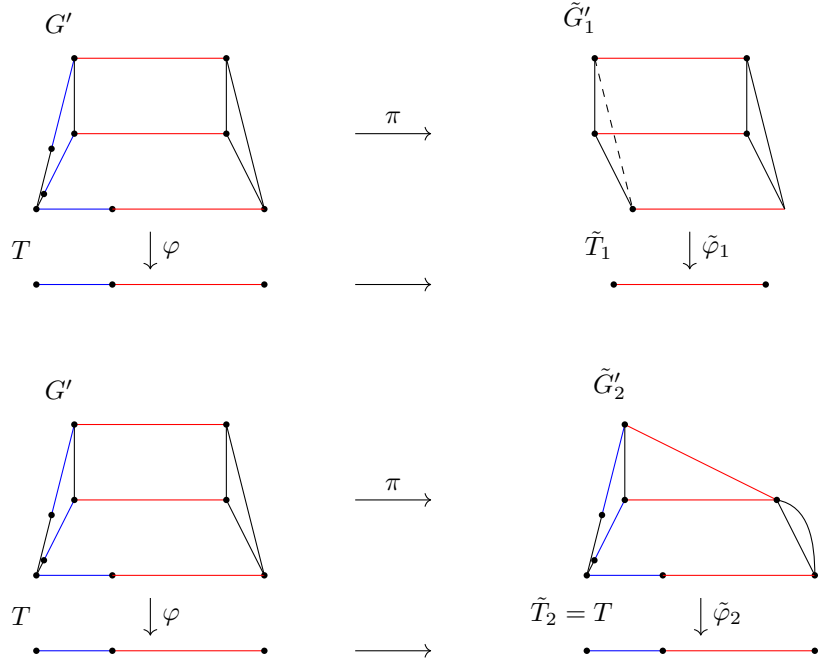

In any case, notice that a $\varphi$-contraction between trigonal types $(G,w,\varphi)\to (\tilde G,\tilde w, \tilde \varphi)$ induces a weighted contraction of stable graphs $(G,w)\to(\tilde G,\tilde w)$ (which could very well be the identity).

\begin{defin}\label{def:trig_contr}
    Let $(G,w,\varphi)$ be a $3$-edge connected trigonal type and let $(\tilde G,\tilde w,\tilde \varphi)$ be the $3$-edge connected trigonal type obtained via a $\varphi$-contraction. A \textbf{trigonal contraction} $(G,w)\to (\tilde G,\tilde w)$ is the weighted contraction induced by the $\varphi$-contraction.
\end{defin}

\begin{defin}
    We denote by $\mathcal{HT}_{g}^{(3)}$, the category whose objects are $3$-edge connected trigonal types of genus $g$, whose morphisms are $\varphi$-contractions.
    Moreover, we denote by $\mathcal{T}_{g}^{(3)}$ the category whose objects are $3$-edge connected trigonal graphs of genus $g$ and whose morphisms are trigonal contractions. 
\end{defin}

In particular, the objects of $\mathcal{T}_{g}^{(3)}$ are stable graphs of genus $g$ and the morphisms are special types of weighted contractions, so $\mathcal T_g^{(3)}$ is a subcategory of the category of stable graphs.

Let $(G,w,\varphi)$ be a $3$-edge connected trigonal type. We have already observed that the morphism $\varphi:G_{\varphi}\to T$ induces an equivalence relation $\sim_{\varphi}$ on $E(G_{\varphi}),$ which in turn induces an order relation on $E(G)$:
\begin{align*}
e_1\leq_{\varphi} e_2 \text{ in } E(G)\qquad&\text{if}\qquad \exists\, e\in E(G_{\varphi});\,e\subseteq e_2 \text{ and } e_1\sim_{\varphi}e;\\
e_1=_{\varphi} e_2 \text{ in } E(G)\qquad&\text{if}\qquad e_1,e_2\in E(G_{\varphi})\text{ and }e_1\sim_{\varphi} e_2 \text{ in } E(G_{\varphi}).
\end{align*}

Here $e\subseteq e_2$ means that $e$ is an edge that can be obtained as a refinement of $e_2.$

\begin{rema}
Given a $3$-edge connected trigonal type $(G,w,\varphi)$ one can consider the induced morphism on a metric graph of which $G_\varphi$ is a model.
We can think of the equivalence relation $\sim_{\varphi}$ as imposing a condition on the lengths of the edges in the same equivalence class: $\mu_{\varphi}(e_i)l(e_i)=\mu_{\varphi}(e_j)l(e_j)$ for $e_i,e_j$ in the same equivalence class and therefore they have same length. Similarly, we can think of the order relation $\leq_{\varphi}$ as imposing an ordering on the lengths of the edges: $e_1\leq_{\varphi} e_2$ implies $l(e_1)\leq l(e_2).$ 
\end{rema}

Notice that a $3$-edge connected trigonal graph $(G,w)$ might be obtained from two distinct trigonal types $(G,w,\varphi_1)$ and $(G, w,\varphi_2)$, which might yield different order relations on $E(G).$ An example is given in Figure~\ref{fg:equiv}. 

\begin{figure}[ht]
\centering
\begin{tikzcd}
\begin{tikzpicture}[scale=0.9]
        \draw (1,2) node[anchor=south] {$e_1$};
    \draw (1,0) node[anchor=south] {$e_3$};
    \draw (1,1) node[anchor=south] {$e_2$};
    \draw(-0.5,0)--(0,2);
    \draw (0.1,0.2) node[anchor=south] {$l_2$};
    \draw (-0.5,0.8) node[anchor=south] {$l_1$};
    \vertex{-0.5,0}
    \vertex{2.5,0}
    \draw(0,1)--(-0.5,0);
    \draw(-0.5,0)--(2.5,0);
    \draw(0,2)--(2,2);
    \draw(0,1)--(2,1);
    \draw(0,2)--(0,1);
    \draw(2,2)--(2,1);
    \draw(2,2)--(2.5,0);
    \draw(2.5,0)--(2.2,0.6);
    \draw(2,1)--(2.5,0);
    \foreach \i in {0,2} {\foreach \j in {1,2} {\vertex{\i, \j}}}
    \draw (-0.2,2.2) node[anchor=south ] {$G$};
    \end{tikzpicture}&
\begin{tikzpicture}[scale=0.9]
    \draw[blue](-0.5,0)--(0,0);
    \draw[blue](-0.15,1.4)--(0,2);
    \draw(-0.15,1.4)--(-0.5,0);\vertex{-0.15,1.4}\vertex{-0.5,0}
    \draw[blue](-0.2,0.6)--(0,1);\vertex{-0.2,0.6}\vertex{2.5,0}
    \draw(-0.2,0.6)--(-0.5,0);
    \draw[red](0,0)--(2,0);
    \draw[red](0,2)--(2,2);

    \draw[red](0,1)--(2,1);
    \draw(0,2)--(0,1);
    \draw(2,2)--(2,1);
    \draw[verde](2,0)--(2.5,0);
    \draw[verde](2,1)--(2.2,0.6);
    \draw(2.5,0)--(2.2,0.6);
    \draw[verde](2.15,1.4)--(2,2);
    \draw(2.15,1.4)--(2.5,0);
    \foreach \i in {0,2} {\foreach \j in {0,1,2,-1} {\vertex{\i, \j}}}
    \vertex{2.2,0.6}\vertex{2.15,1.4}
    \draw[->](1,-0.3) to (1,-0.8);
    \draw (1,-0.5) node[anchor=west] {$\varphi_1$};
    \draw (-0.2,2.2) node[anchor=south ] {$G_{\varphi_1}$};
    \draw (2,2) node[anchor=south ] {$x$};
    \draw (2,1) node[anchor=north ] {$x_1$};
    \draw (2,0) node[anchor=north ] {$x_2$};
    \draw (-0.7,-0.8) node[anchor=south ] {$T_1$};
    \draw[blue](-0.5,-1)--(0,-1);
    \draw[red](0,-1)--(2,-1);
    \draw[verde](2,-1)--(2.5,-1);
    \vertex{-0.5,-1}\vertex{2.5,-1}
    \divisor{2,2}\divisor{2,1}\divisor{2,0}
    \end{tikzpicture}&
    \begin{tikzpicture}[scale=0.9]
    \draw[blue](-0.5,0)--(0.5,0);
    \draw[blue](-0.5,0)--(0,1);
    \vertex{-0.5,0}
    \draw[blue](-0.3,0.8)--(0,2);\vertex{-0.3,0.8}\vertex{2.5,0}
    \draw(-0.3,0.8)--(-0.5,0);
\vertex{0.5,0}
    \draw[red](2.5,0)--(0.5,0);
    \draw[red](0,2)--(2,2);
    \draw[red](0,1)--(2,1);
    \draw(0,2)--(0,1);
    \draw(2,2)--(2,1);
    \draw(2.5,0)--(2,1);\draw(2.5,0)--(2,2);
    \foreach \i in {0,2} {\foreach \j in {1,2} {\vertex{\i, \j}}}
    \draw[->](1,-0.3) to (1,-0.8);
    \draw (1,-0.5) node[anchor=west] {$\varphi_2$};
    \draw (-0.2,2.2) node[anchor=south ] {$G_{\varphi_2}$};
    \draw (-0.7,-0.8) node[anchor=south ] {$T_2$};
    \draw[blue](-0.5,-1)--(0.5,-1);
    \draw[red](0.5,-1)--(2.5,-1);
    \draw (2,2) node[anchor=south ] {$x$};
    \divisor{2,2}\divisor{2,1}\divisor{2.5,0}
    \vertex{-0.5,-1}\vertex{0.5,-1}\vertex{2.5,-1}
    \draw (2,1) node[anchor=north ] {$x_1'$};
    \draw (2.5,0) node[anchor=north ] {$x_2'$};
    \end{tikzpicture}
    \end{tikzcd}\caption{A stable graph $G$ on the left, together with two distinct harmonic morphisms $\varphi_1:G_{\varphi_1}\to T_1$ and $\varphi_2:G_{\varphi_2}\to T_2.$}\label{fg:equiv}
\end{figure}
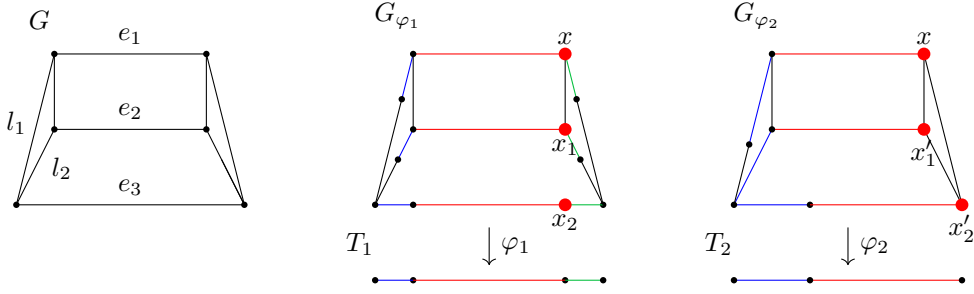

\begin{defin}
Let $(G,w,\varphi)$ be a $3$-edge connected trigonal type. We define a rational polyhedral cone 
$$\tau_{(G,w,\varphi)}=\mathbb R_{>0}^{|E(G)/\leq_{\varphi}|}$$
parametrizing length functions on a stable graph $(G,w)$, hence tropical curves, satisfying the constraints given by the order relation $\leq_\varphi$. 
\end{defin}

The space defined above is clearly a rational polyhedral cone: it is obtained as the intersection of a rational polyhedral cone with half-planes through the origin with generators defined by integral vectors. In fact the locus $\tau_{(G,w,\varphi)}\subseteq \sigma_G=\mathbb R_{>0}^{|E(G)|}$ is defined by all $|E(G)|$-tuples $\{x_1,\dots x_{|E(G)|}\}\in\sigma_G$ satisfying the constraints given by the order relation $\leq_\varphi$. In other words, for all $(G,w,\varphi)$ trigonal type, $\tau_{(G,w,\varphi)}$ is obtained by taking the intersection of $\sigma_G$ and all the half-planes $\{x_i\leq x_j\}$ for $e_i\leq_{\varphi} e_j.$ 

The faces of its closure $\overline{\tau}_{(G,w,\varphi)}$ parametrize degenerations of trigonal tropical curves in $\tau_{(G,w,\varphi)}$ obtained via a $\varphi$-contraction as in Proposition~\ref{prp:contr}.

For instance, let us consider the trigonal type $(G,\underline{0},\varphi_2)$ in Figure~\ref{fg:equiv}. The order relation induced by $\varphi_2$ gives $e_1=_{\varphi_2}e_2\leq_{\varphi_2}e_3$, $l_2\leq_{\varphi_2}l_1 $ and $l_2\leq_{\varphi_2}e_3 $. The cone~$\sigma_G=\mathbb R^9_{>0}$ parametrizes all possible lengths, also the ones not admitted, of the edges $e_1,e_2,e_3,l_1,l_2,y_1,,\dots,y_4$ of $G,$ where $y_i$ denotes the edges with no conditions on their lengths, and the associated cone ${\tau}_{(G,\underline{0},\varphi_2)}$ is given by the intersection of $\sigma_G$ with the half-planes $\{x_1=x_2\}$, $\{x_2\leq x_3\},$ $\{x_5\leq x_4\}$ and $\{x_5\leq x_3\}$. 
The faces of the closure $\overline{\tau}_{(G,\underline{0},\varphi_1)}$ are then obtained either by allowing the coordinates to be $0$ or by considering the points in the plane, corresponding to the case where equalities on the lengths are attained. 

Let us consider a $\varphi$-contraction  $\pi:(G,w,\varphi)\to (\tilde G,\tilde w,\tilde \varphi).$ 
As noticed in the proof of Proposition~\ref{prp:contr}, $(\tilde G,\tilde w)$ is either isomorphic to $(G,w),$ or it is obtained as a weighted contraction of it. Therefore a $\varphi$-contraction always induces an inclusion $i_{\pi}: E(\tilde G)\to E(G).$  
These inclusions then naturally yield face morphisms. Notice that the face morphism associated to $\pi$ induces an automorphism of the cone whenever no edge in~$G$ is contracted.

\begin{defin}
We define the \textbf{moduli spaces of $3$-edge connected trigonal covers} as the colimit
$$H_{g,3}^{\operatorname{trop},(3)}:=\varprojlim_{(G,w,\varphi)\in \mathcal{HT}_{g}^{(3)}}\overline{\tau}_{(G,w,\varphi)}. $$
\end{defin}

\begin{defin}
Given a $3$-edge connected trigonal graph $(G,w)$ we define
$$\tau_{G}:=\bigcup_{\varphi;(G,w,\varphi) \in\mathcal{HT}_g^{(3)}}\tau_{(G,w,\varphi)}=\bigcup_{\varphi;(G,w,\varphi) \in\mathcal{HT}_g^{(3)}}\mathbb R_{>0}^{|E(G)/\leq_{\varphi}|}$$
which parametrizes length functions on trigonal graphs $(G,w)$, hence trigonal tropical curves. 
\end{defin}

The space $\tau_{G}$ is a union of cones and the faces of its closure $\overline{\tau}_{G}$ parametrize the degenerations of the tropical trigonal curves in $\tau_G$ obtained through trigonal contractions as in Definition~\ref{def:trig_contr}.

\begin{defin}
We define the \textbf{moduli spaces of $3$-edge connected tropical trigonal curves} as the colimit
$$T_{g}^{\operatorname{trop},(3)}:=\varprojlim_{(G,w)\in \mathcal{T}_{g}^{(3)}}\overline{\tau}_{G}. $$
\end{defin}

The above construction is similar to that of the moduli of ($2$-edge connected) tropical hyperelliptic curves $H_{g}^{\rm trop,(2)}$ , which, as remarked in~\cite[Remark 4.10]{MC}, is a subset of $M_{g}^{\rm trop},$ but it is not a stacky subfan because its structure is more refined. 
Similarly, $T_g^{\rm trop, (3)}\subset M_{g}^{\rm trop}$ as a set and its structure as a fan is not the one induced by $M_{g}^{\rm trop}$. 

The space $H_{g,3}^{\operatorname{trop},(3)} $ has a natural projection to the moduli space $M_g^{\operatorname{trop}}$ of tropical curves of genus $g$. Notice that the moduli space $T_g^{\operatorname{trop},(3)}$ can also be thought as the image of $H_{g,3}^{\operatorname{trop},(3)}$ via its natural projection to $M_g^{\operatorname{trop}}$.

In the theory of algebraic curves,  it is well known that the Catalan number $C_n$ counts the covers $C\to\mathbb P^1$ of
minimal degree $n+1$ from a general curve $C$ of genus~$2n.$ This was already known to Castelnuovo,~\cite{C89}.
The tropical analog has been proved by Vargas in~\cite[Theorem C]{V}.
When $g=4,$ then $3=\lceil g/2\rceil+1,$ which is the gonality bound computed by Draisma and Vargas in~\cite[Theorem 2]{DV}, for tropical curves. Then, the degree of the forgetful map $H_{g,3}^{\operatorname{trop},(3)} \to T_g^{\operatorname{trop},(3)}$ is the $2$nd Catalan number $C_2=2$.

\subsection{Maximal cells for the moduli space of \texorpdfstring{$3$}{3}-edge connected trigonal curves 
}\label{sc:max_cells}
We would like to give a description of the maximal $3$-edge connected trigonal graphs of fixed genus $g$, i.e. of those graphs which cannot be obtained as a trigonal contraction of any other trigonal graph of the same genus.

Let $T$ be a non-trivial tree with vertices of valence smaller or equal than $3$.
Starting with $T$, we will construct stable and $3$-edge connected trigonal graphs $G$ with refinement $G_T$. We will later prove that such graphs are maximal in the above sense.

The graph $G_T$ constructed from the tree $T$ is constructed as follows.
Take $3$ disjoint copies of $T:$ $T^{(1)},T^{(2)},T^{(3)}$ with ordered vertices
$V(T^{(i)})=\{v_1^{(i)},\dots, v_n^{(i)}\},$ $i=1,2,3,$ such that $v_j^{(1)},v_j^{(2)},v_j^{(3)}$ correspond to the same vertex in $V(T)$, for any $j=1,\dots,n.$

Then, for any $j=1,\dots,n,$ add edges as follows.
\begin{itemize}
\item If $\operatorname{val}(v_j)=2,$ choose $i_1,i_2\in \{1,2,3\}$ and connect $v_j^{(i_1)},v_j^{(i_2)}$ with an edge.

\item If $\operatorname{val}(v_j)=1,$ connect the vertices $v_j^{(i)},i=1,2,3$ by inserting two non-parallel edges.
\end{itemize}
Moreover, if $T$ is a path, if possible, then we require that the added edges over the vertices of valence $2$ cannot have endpoints all contained in the same two copies of the tree.

\begin{defin}
    Let $T$ be a tree and $G_T$ be a graph that arises from the above construction. We call the graph $G_T$ a \textbf{$3$-ladder}.
\end{defin}

Notice that the morphism that maps the edges in $\{T^{(1)},T^{(2)},T^{(3)}\}\subset G_T$ to the corresponding edge in $T$ and contracts the remaining edges is non-degenerate and harmonic of degree $3$. We will denote it by $\varphi_T$. 

An example of the construction of a $3$-ladder from a tree $T$ is shown in Figure~\ref{fg:ex_max}, together with its stable model $G$ in Figure~\ref{fg:ex_max1}. 
\begin{figure}[ht]
\adjustbox{scale=0.7,center}{
\begin{tikzcd}
\begin{tikzpicture}
    \draw (-0.2,1.2) node[anchor=east] {$G_T$};
    \draw[red](0,0) --  (2,0);
    \draw[red](0,1) --  (2,1);
    \draw[red](0,-1) --  (2,-1);
    \draw[] (0,-1) to [out=120, in=240] (0,0);
    \draw[] (0,-1) to [out=120, in=240] (0,1);
    \draw[blue](4,0) --  (2,0);
    \draw[blue](4,1) --  (2,1);
    \draw[blue](4,-1) --  (2,-1);
    \draw[] (2,0) to [out=120, in=240] (2,1);
    \draw[orange](4,0) --  (5,0.5);
    \draw[orange](4,1) --  (5,1.5);
    \draw[orange](4,-1) --  (5,-0.5);
    \draw[] (5,-0.5) to [out=120, in=240] (5,0.5);
    \draw[] (5,0.5) to [out=120, in=240] (5,1.5);
    \draw[verde,dashed](4,0) --  (5,-0.25);
    \draw[verde](5,-0.25) --  (6,-0.5);
    \draw[verde,dashed](4,1) --  (5,0.75);
    \draw[verde](6,0.5) --  (5,0.75);
    \draw[verde](4,-1) --  (6,-1.5);
    \draw[] (6,-1.5) to [out=60, in=300]  (6,0.5);
    \draw[] (6,-0.5) to [out=60, in=300] (6,0.5);
    \foreach \i in {0,2,4} {
        \foreach \j in {0,1,-1}{
            \vertex{\i,\j}}
    }  
    \foreach \j in {1.5,0.5,-0.5} {
        \vertex{5,\j}
    }  
    \foreach \j in {0.5,-0.5,-1.5} {
        \vertex{6,\j}
    }  
    \draw[->] (3,-1.5) to (3,-2);
    \draw (3,-1.75) node[anchor=west] {$\varphi_T$};
    \draw (-0.2,-2.8) node[anchor=east] {$T$};
    \draw[red](0,-3) --  (2,-3);
    \draw[blue](2,-3) --  (4,-3);
    \draw[orange](5,-2.5) --  (4,-3);
    \draw[verde](6,-3.5) --  (4,-3);
    \vertex{0,-3}\vertex{2,-3}\vertex{4,-3}\vertex{6,-3.5}\vertex{5,-2.5}
    \end{tikzpicture}
    \end{tikzcd}}\caption{A $3$-ladder $G_T$ defined by a tree $T$ with the above construction, together with a non-degenerate harmonic morphism of degree $3.$}\label{fg:ex_max}
\end{figure}
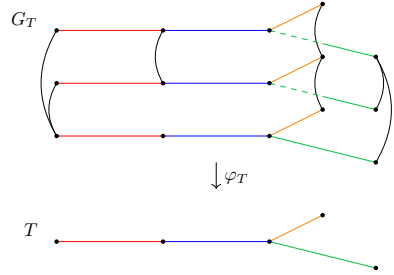

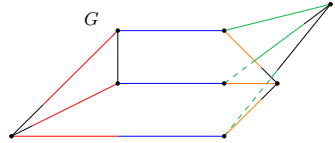
\begin{figure}[ht]
\adjustbox{scale=0.7,center}{
\begin{tikzcd}
\begin{tikzpicture}
    \draw (-0.2,2.2) node[anchor=east] {$G$};
    \draw[blue](0,2) --  (2,2);
    \draw[blue](0,1) --  (2,1);
    \draw[blue](0,0) --  (2,0);
    \draw(0,2)--(0,1);
    \draw[red](0,1)--(-1.6,0.2);
    \draw[red](0,2)--(-1.4,0.6);
    \draw[red](0,0) --  (-2,0);
    \draw[](-2,0) --  (-1.6,0.2);
    \draw[](-2,0) --  (-1.4,0.6);

    \draw[orange](2,1) --  (3,1);
    \draw[orange](2,2) --  (2.7,1.3);
    \draw[orange](2,0) --  (2.7,0.7);
    \draw[](2.7,1.3) --  (3,1);
    \draw[](2.7,0.7) --  (3,1);
    \draw[verde](2,2) --  (4,2.5);
    \draw[verde, dashed](2,1) --  (2.57,1.43);
    \draw[verde](2.57,1.43) --  (3.5,2.125);
    \draw[](3.5,2.125) --  (4,2.5);
    \draw[verde, dashed](2,0) --  (2.89,1.11);
    \draw[verde](2.89,1.11) --  (3,1.25);
    \draw[](3,1.25) --  (4,2.5);
    \vertex{-2,0}
    \vertex{0,1}\vertex{0,2}\vertex{2,2}\vertex{2,1}\vertex{2,0}
    \vertex{4,2.5}\vertex{3,1}
    \end{tikzpicture}
    \end{tikzcd}}\caption{The stable model $G$ of the $3$-ladder $G_T$ represented in Figure~\ref{fg:ex_max} as a refinement.}\label{fg:ex_max1}
\end{figure}

\begin{rema}
Let us remark that, unlike the \emph{ladders} in~\cite{MC}, given a tree $T,$ a $3$-ladder graph $G_T$ is not uniquely determined, but different choices of the edges connecting the copies of vertices of valences 1 or 2 in $T$ might yield different $3$-ladders, up to isomorphism, as in Figure~\ref{fg:numb_max}.
\begin{figure}[ht]
\centering
\begin{tikzcd}
\begin{tikzpicture}
    \draw (-0.8,0.2) node[anchor=east] {$T$};
    \draw(-1,0) --  (0,0);
    \draw(0,0) --  (1,0.5);
    \draw(0,0) --  (1,-0.5);
    \vertex{-1,0}\vertex{0,0}\vertex{1,0.5}\vertex{1,-0.5}

    \draw(-5.5,2) --  (-4.5,2);
    \draw(-4.5,2) --  (-3.5,1.5);
    \draw(-4.5,2) --  (-3.5,2.5);
    \draw(-5.5,2.75) --  (-4.5,2.75);
    \draw(-4.5,2.75) --  (-3.5,2.25);
    \draw(-4.5,2.75) --  (-3.5,3.25);
    \draw(-5.5,3.5) --  (-4.5,3.5);
    \draw(-4.5,3.5) --  (-3.5,3);
    \draw(-4.5,3.5) --  (-3.5,4);
    \draw[verde] (-5.5,2) to [out=120, in=240] (-5.5,2.75);
    \draw[verde] (-5.5,2.75) to [out=120, in=240] (-5.5,3.5);
    \draw[verde] (-3.5,2.5) to [out=120, in=240] (-3.5,3.25);
    \draw[verde] (-3.5,3.25) to [out=120, in=240] (-3.5,4);
    \draw[verde] (-3.5,1.5) to [out=60, in=300] (-3.5,2.25);
    \draw[verde] (-3.5,2.25) to [out=60, in=300] (-3.5,3);

    \draw(-2.5,2) --  (-1.5,2);
    \draw(-1.5,2) --  (-0.5,1.5);
    \draw(-1.5,2) --  (-0.5,2.5);
    \draw(-2.5,2.75) --  (-1.5,2.75);
    \draw(-1.5,2.75) --  (-0.5,2.25);
    \draw(-1.5,2.75) --  (-0.5,3.25);
    \draw(-2.5,3.5) --  (-1.5,3.5);
    \draw(-1.5,3.5) --  (-0.5,3);
    \draw(-1.5,3.5) --  (-0.5,4);
    \draw[verde] (-2.5,2) to [out=120, in=240] (-2.5,2.75);
    \draw[verde] (-2.5,2.75) to [out=120, in=240] (-2.5,3.5);
    \draw[verde] (-0.5,2.5) to [out=120, in=240] (-0.5,3.25);
    \draw[verde] (-0.5,3.25) to [out=120, in=240] (-0.5,4);
    \draw[red] (-0.5,2.25) to [out=60, in=300] (-0.5,3);
    \draw[red] (-0.5,1.5) to [out=60, in=300] (-0.5,3);

    \draw(0.5,2) --  (1.5,2);
    \draw(1.5,2) --  (2.5,1.5);
    \draw(1.5,2) --  (2.5,2.5);
    \draw(0.5,2.75) --  (1.5,2.75);
    \draw(1.5,2.75) --  (2.5,2.25);
    \draw(1.5,2.75) --  (2.5,3.25);
    \draw(0.5,3.5) --  (1.5,3.5);
    \draw(1.5,3.5) --  (2.5,3);
    \draw(1.5,3.5) --  (2.5,4);
    \draw[verde] (0.5,2) to [out=120, in=240] (0.5,2.75);
    \draw[verde] (0.5,2.75) to [out=120, in=240] (0.5,3.5);
    \draw[red] (2.5,2.5) to [out=120, in=240] (2.5,4);
    \draw[red] (2.5,3.25) to [out=120, in=240] (2.5,4);
    \draw[red] (2.5,2.25) to [out=60, in=300] (2.5,3);
    \draw[red] (2.5,1.5) to [out=60, in=300] (2.5,3);

    \draw(3.5,2) --  (4.5,2);
    \draw(4.5,2) --  (5.5,1.5);
    \draw(4.5,2) --  (5.5,2.5);
    \draw(3.5,2.75) --  (4.5,2.75);
    \draw(4.5,2.75) --  (5.5,2.25);
    \draw(4.5,2.75) --  (5.5,3.25);
    \draw(3.5,3.5) --  (4.5,3.5);
    \draw(4.5,3.5) --  (5.5,3);
    \draw(4.5,3.5) --  (5.5,4);
    \draw[verde] (3.5,2) to [out=120, in=240] (3.5,2.75);
    \draw[verde] (3.5,2.75) to [out=120, in=240] (3.5,3.5);
    \draw[red] (5.5,2.5) to [out=120, in=240] (5.5,4);
    \draw[red] (5.5,3.25) to [out=120, in=240] (5.5,4);
    \draw[blue] (5.5,1.5) to [out=60, in=300] (5.5,2.25);
    \draw[blue] (5.5,1.5) to [out=60, in=300] (5.5,3);
    \foreach \i in {-5.5,-4.5,-2.5,-1.5,0.5,1.5,3.5,4.5} {
        \foreach \j in {2,2.75,3.5}{
            \vertex{\i,\j}}
    }  
    \foreach \i in {-3.5,-0.5,2.5,5.5} {
        \foreach \j in {1.5,2.25,2.5,3,3.25,4}{
            \vertex{\i,\j}}
    }  
    \end{tikzpicture}
    \end{tikzcd}\caption{Four non-isomorphic graphs, which are $3$-ladders constructed from the tree $T$.}\label{fg:numb_max}
\end{figure}
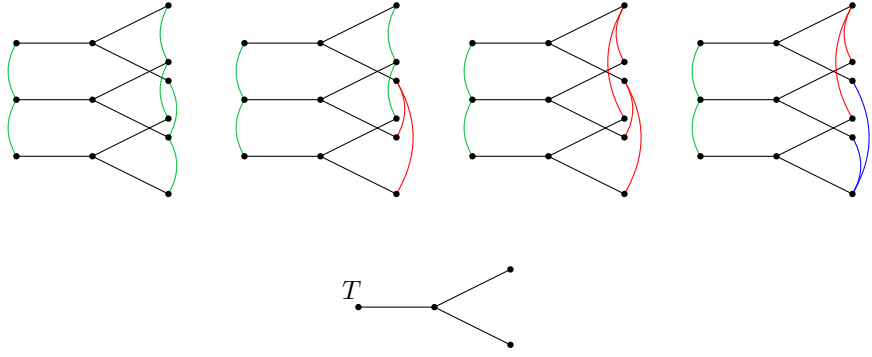
\end{rema}

\begin{rema}
    Consider the case where $T$ is a path. The condition on the edges over vertices of valence $2$ ensures the graph to be maximal with respect to  trigonal contraction. 
    
    We will now show that in fact that if $G$ admits a non-degenerate harmonic morphism of degree $3$ from a graph $G_T'$ where the above condition is not satisfied, then it is not maximal, since it also admits another non-degenerate harmonic morphism of degree $3$ from a graph $G_1',$ which is a trigonal contraction of a~$3$-ladder.
    
    More precisely, if $(G,w,\varphi_T')$ is a trigonal type with a morphism $\varphi_T':G_T'\to T$ with $(G_T')^{\operatorname{st}}=G$ such that $G_T'$ is constructed as $G_T,$ without the condition on the edges over vertices of valence $2$ , then $G$ also admits a degree $3$ harmonic morphism $G_1\to T_1$ with $(G_1)^{\operatorname{st}}=G$ and such that $G_1$ is the trigonal contraction of a $3$-ladder~$G_T.$
    
    In fact, let us consider $G_T'$ whose construction is analogous to that of $3$-ladder $G_T,$ but with the difference that the edges over vertices of valence $2$ are all incident to vertices in the same two copies of $T$. 
    Indeed, for any choice of vertical edges made over the vertices of valence $1,$ we have two possible morphisms. The first one is the one that follows from the usual construction, where all vertical edges are contracted. The second is instead obtained by considering the vertical edges together with their consecutive horizontal one (glued by a single vertex of valence $2$) and contract them to the same vertex that is the image of the unique vertex of valence $3$.
    For instance we can assume the graph $G_T'$ to be as the graph in Figure~\ref{fg:not_max}, with a second harmonic morphism depicted by $G_1'\to T_1,$ both with stable model as the graph in the left in Figure~\ref{fg:equiv}.
    \begin{figure}[ht]
\centering
\begin{tikzcd}
\begin{tikzpicture}
    \draw (-0.2,1.2) node[anchor=south] {$G_T'$};
    \draw[blue](0,0)--(1,0);
    \draw[blue](0,0.5)--(1,0.5);
    \draw[blue](0,1)--(1,1);
    \draw(1,1)--(1,0.5);
    \draw(2,1)--(2,0.5);
    \draw(0,0) to [out=120, in=240] (0,1);
    \draw(0,0) to [out=120, in=240] (0,0.5);
    \draw(3,0) to [out=60, in=300] (3,0.5);
    \draw(3,0) to [out=60, in=300] (3,1);
    \draw[red](2,0)--(1,0);
    \draw[red](2,0.5)--(1,0.5);
    \draw[red](2,1)--(1,1);
    \draw[verde](2,0)--(3,0);
    \draw[verde](2,0.5)--(3,0.5);
    \draw[verde](2,1)--(3,1);
    \draw[->] (1.5,-0.4) to (1.5,-0.7);
    \draw[blue](0,-1)--(1,-1);
    \draw[red](2,-1)--(1,-1);
    \draw[verde](2,-1)--(3,-1);
    \draw (-0.2,-0.8) node[anchor=south] {$T$};
        \foreach \i in {0,1,2,3} {
    \foreach \j in {0,0.5,1,-1} {
    	    \vertex{\i, \j}
         }
    	}
    \end{tikzpicture}&\begin{tikzpicture}
    \draw[<->](0,0.5)to (2,0.5);   
\draw (1,0.5) node[anchor=south] {$(G_T')^{\operatorname{st}}=(G'_1)^{\operatorname{st}}$
};
    \end{tikzpicture}&
    \begin{tikzpicture}
    \draw (-0.2,1.2) node[anchor=south] {$G_1'$};
    \draw[blue](0,0)--(1,0);
    \draw[blue](0,0.5)--(1,0.5);
    \draw[blue](0,1)--(1,1);
    \draw(1,1)--(1,0.5);
    \draw(2,1)--(2,0.5);
    \draw(0,0) to [out=120, in=240] (0,1);
    \draw(0,0) to [out=120, in=240] (0,0.5);
    \draw(2,0) to [out=60, in=300] (2,0.5);
    \draw(2,0) to [out=60, in=300] (2,1);
    \draw[red](2,0)--(1,0);
    \draw[red](2,0.5)--(1,0.5);
    \draw[red](2,1)--(1,1);
    \draw[->] (1,-0.4) to (1,-0.7);
    \draw[blue](0,-1)--(1,-1);
    \draw[red](2,-1)--(1,-1);
    \draw (-0.2,-0.8) node[anchor=south] {$T_1$};
        \foreach \i in {0,1,2} {
    \foreach \j in {0,0.5,1,-1} {
    	    \vertex{\i, \j}
         }
    	}
    \end{tikzpicture}
    \end{tikzcd}\caption{}\label{fg:not_max}
\end{figure}

It is then easy to check that we can obtain $G'_1$ as the trigonal contraction of a $3$-ladder $G_T,$ as in Figure~\ref{fg:not_max2}.    

\begin{figure}[ht]
\centering
\begin{tikzcd}
\begin{tikzpicture}
    \draw (-0.2,1.2) node[anchor=south] {$G_T$};
    \draw[blue](0,0)--(1,0);
    \draw[blue](0,0.5)--(1,0.5);
    \draw[blue](0,1)--(1,1);
    \draw(1,1)--(1,0.5);
    \draw(2,0)--(2,0.5);
    \draw(0,0) to [out=120, in=240] (0,1);
    \draw(0,0) to [out=120, in=240] (0,0.5);
    \draw(3,0.5) to [out=60, in=300] (3,1);
    \draw(3,0) to [out=60, in=300] (3,1);
    \draw[red](2,0)--(1,0);
    \draw[red](2,0.5)--(1,0.5);
    \draw[red](2,1)--(1,1);
    \draw[verde](2,0)--(3,0);
    \draw[verde](2,0.5)--(3,0.5);
    \draw[verde](2,1)--(3,1);
    \draw[->] (1.5,-0.4) to (1.5,-0.7);
    \draw[blue](0,-1)--(1,-1);
    \draw[red](2,-1)--(1,-1);
    \draw[verde](2,-1)--(3,-1);
    \draw (-0.2,-0.8) node[anchor=south] {$T$};
        \foreach \i in {0,1,2,3} {
    \foreach \j in {0,0.5,1,-1} {
    	    \vertex{\i, \j}
         }
    	}
    \end{tikzpicture}&\begin{tikzpicture}
    \draw[->](0,0.5)to (1,0.5);   
\draw (0.5,0.5) node[anchor=south] {$\pi$
};
    \end{tikzpicture}&
    \begin{tikzpicture}
    \draw (-0.2,1.2) node[anchor=south] {$G_1'$};
    \draw[blue](0,0)--(1,0);
    \draw[blue](0,0.5)--(1,0.5);
    \draw[blue](0,1)--(1,1);
    \draw(1,1)--(1,0.5);
    \draw(2,1)--(2,0.5);
    \draw(0,0) to [out=120, in=240] (0,1);
    \draw(0,0) to [out=120, in=240] (0,0.5);
    \draw(2,0) to [out=60, in=300] (2,0.5);
    \draw(2,0) to [out=60, in=300] (2,1);
    \draw[red](2,0)--(1,0);
    \draw[red](2,0.5)--(1,0.5);
    \draw[red](2,1)--(1,1);
    \draw[->] (1,-0.4) to (1,-0.7);
    \draw[blue](0,-1)--(1,-1);
    \draw[red](2,-1)--(1,-1);
    \draw (-0.2,-0.8) node[anchor=south] {$T_1$};
        \foreach \i in {0,1,2} {
    \foreach \j in {0,0.5,1,-1} {
    	    \vertex{\i, \j}
         }
    	}
    \end{tikzpicture}
    \end{tikzcd}\caption{}\label{fg:not_max2}
\end{figure}

\end{rema}

Recall, from the discussion on the previous section, that any non-degenerate harmonic morphism $\varphi_T:G_T\to T$ of degree $3$ induces an equivalence relation $\sim_{\varphi_T}$ on $|E(G_T)|.$ Let us now compute the genus and the number of equivalence classes $|E(G_T)/\sim_{\varphi_T}|$ in terms of the number of vertices $n$ of the tree $T.$

\begin{lemm}\label{lm:2n+1}
    Let $T$ be a tree with $n$ vertices. Then $g(G_T)=n$ and $|E(G_T)/\sim_{\varphi_T}|=2n+1.$
\end{lemm}
\begin{proof}
From our construction, $|V(G_T)|=3|V(T)|=3n.$

Then, for any edge in $T,$ we have $3$ copies in $G_T$, but then we also have to add $2$ edges for each leaf, and one for any vertex of valence $2.$

Let $\mu_k$ denote the number of vertices in $T$ of valence $k.$
We have $|E(G_T)|=3|E(T)|+2\mu_1+\mu_2=3(n-1)+2\mu_1+\mu_2.$
Noticing that $n=\mu_1+\mu_2+\mu_3$ and $\mu_1=\mu_3+2,$ gives $|E(G_T)|=5\mu_1+4\mu_2+3\mu_3-3=4n-1.$
Then $g(G_T)=b_1(G_T)=4n-1-3n+1=n.$

The number of equivalence classes, instead, is precisely the number of edges in the three, plus all the other edges added in the construction. Thus 
$|E(G_T)/\sim_{\varphi_T}|=(n-1)+2\mu_1+\mu_2=3\mu_1+2\mu_2+\mu_3-1=2\mu_1+2\mu_2+2\mu_3+1=2n+1.$
\end{proof}

When we consider the stabilization of $G_T$ the genus is preserved. Moreover we also know that the equivalence relation on $E(G')$ also induced an order relation $\leq_{\varphi_T}$ on the edges of stable graph $G=(G_T)^{\operatorname{st}}$, and thus an equivalence relation $=_{\varphi_T}$. Clearly, $|E(G)/=_{\varphi_T}|\leq|E(G_T)/\sim_{\varphi_T}|$, but we will prove that actually equality holds, with an exception for $g=3.$

\begin{lemm}\label{lm:stab}
Let $G_T$ be such that $g(G_T)>3.$ Then, if $G=(G_T)^{\operatorname{st}}$, we have $|E(G_T)/\sim_{\varphi_T}|=|E(G)/=_{\varphi_T}|$.

If $g(G_T)=3,$ then $|E(G)/=_{\varphi_T}|=|E(G_T)/\sim_{\varphi_T}|-1=6.$
\end{lemm}
\begin{proof}

Let us count the number of equivalence classes in $E(G)$, with respect to the equivalence relation induced by $\varphi_T.$

For any leaf-edge in $T$, we have that its pre-image via $\varphi_T$ gives precisely three edges contained in precisely three distinct edges in $G.$ 

Then $G'_T,$ the pre-image in $G_T$ of the tree $T'$ obtained by removing all leaves from~$T$, with the usual construction. 
The tree $T'$ will thus have $n-\mu_1-1$ edges.

By construction, the pre-image of an edge incident to a vertex of valence $2$ in $T'$
is either an edge in $E(G_T),$ or an edge which is contained in one edge in $G,$
whose length is uniquely determined by the lengths of the edges in the tree. 

Moreover, the pre-image of the edges incident to vertices in $T'$ of valence $3$ and the vertical edges are edges in the stable graph $G.$

Hence if $n-\mu_1-1>0$ (i.e. $T'$ containes of at least an edge), we have that $|E(G)/=_{\varphi_T}|=3\mu_1+(n-\mu_1-1)+\mu_2=3\mu_1+2\mu_2+\mu_3-1=2n+1.$
Notice also that the condition $n-\mu_1-1>0$ is equivalent to $g>3.$

If instead $n-\mu_1-1=0$ we have precisely that $g=3,$ in this case, the tree with the leaves removed has no edges therefore the above argument cannot be applied.
In this case if suffices to count the number of edges in $G$, which is $6$, and check that they all belong to different equivalence classes.

Notice that this is precisely the case where $T$ is a path. Indeed, more generally, in the case where all vertical edges in $T$ are incident to the same copies of the tree, then there is an edge mapping the whole tree, as in Figure~\ref{fg:not_max}. The length of this edge is then counted twice when we count the pre-images over the leaves, therefore the dimension would drop by $1.$ 
\end{proof}

\begin{rema}\label{rk:max}
We would like to conclude by saying that the dimension of $T_g^{\operatorname{trop},(3)}$ is determined by $E(G_T)$ and the associated morphism. 
However, let us recall that a trigonal graph might allow multiple morphisms, which thus could define a different dimension. Let us now show that any other morphism, if it exists, would induce a smaller dimension. 

Let $G$ be the stabilization of a $3$-ladder $G_T,$ and assume that there is a second morphism $\varphi':G'\to T',$ with $(G')^{\operatorname{st}}=G$ and $T'$ some other tree.

Since $G$ has no points of valence greater than $3,$ the same must hold for $T'.$ By definition of trigonal type, we also require that $\varphi'$ is such that the preimage of any $t\in V(T)$ contains (at least) a vertex of valence $3.$ In particular $G$ is $3$-regular, i.e. its vertices have all valence $3.$ Then, by $3$-edge connectivity and $3$-regularity of the stabilization, this means precisely that the pre-image of vertices of valence $2$ contains precisely one edge, and that of vertices of valence $1$ contains two of them. Therefore a similar argument to that of the proof of Lemma~\ref{lm:2n+1} can be applied.
Thus,  $$|E(G)/=_{\varphi'}|\leq|E(G')/\sim_{\varphi'}|=2g+1=|E(G_T)/\sim_{\varphi_T}|.$$
\end{rema}

\begin{theo}\label{th:max_cells}
    Let $g\geq3.$ The maximal cells of $T_g^{\operatorname{trop},(3)}$ are associated to graphs $(G,\underline{0})$ with $G=(G_T)^{\operatorname{st}}$ with $G_T$ a $3$-ladder constructed from $T$ a tree with $g$ vertices.
\end{theo}
\begin{proof}
Let $(G,w)$ be a trigonal graph and we assume it to be maximal with respect to trigonal contractions.

First notice that $w$ has to be uniformly $0,$ otherwise we could write $(G,w)$ as the contraction of the loops in $(G^{w},\underline{0}).$ 

Let us then show that $|V(G)|\not\in\{2,3\}.$ This is done by showing that in any of these cases, the graph $G$ can be obtained as the trigonal-contraction of a trigonal graph.
If~$|V(G)|\in\{2,3\}$ then $G$ is a graph with $g+1$ or $g+2$ edges. 

Clearly a graph with 2 vertices can always be obtained from a graph with 3 vertices with two of the vertices connected by exactly one edge $e$, via the contraction of $e$, as shown in Figure~\ref{fg:v2}.
\begin{figure}[ht]
\centering
\begin{tikzcd}
\begin{tikzpicture}
    \vertex{0,0}\vertex{1,1}\vertex{2,0}
    \draw[] (0,0) to [out=340, in=200] (2,0);
    \draw[] (0,0) to [out=80, in=190] (1,1);
    \draw[] (0,0) to [out=60, in=210] (1,1);
    \draw[] (0,0) to [out=20, in=260] (1,1);
    \draw[] (2,0) to [out=110, in=350] (1,1);
    \draw[] (2,0) to [out=160, in=280] (1,1);
    \draw (1,-0.5) node[anchor=north] {$e$};
    \end{tikzpicture}\arrow[r,"\pi"]&
    \begin{tikzpicture}
    \vertex{0,0}\vertex{0,1}
    \draw[] (0,0) to [out=150, in=210] (0,1);
    \draw[] (0,0) to [out=30, in=330] (0,1);
    \draw[] (0,0)--(0,1);
    \draw[] (0,0) to [out=120, in=240] (0,1);
    \draw[] (0,0) to [out=60, in=300] (0,1);
    \draw (0,-0.5) node[anchor=north] {$\pi(e)$};
    \end{tikzpicture}
    \end{tikzcd}\caption{}\label{fg:v2}
\end{figure}

If instead $|V(G)|=3$ then $G$ can be obtained from a graph $G'$ with $6$ vertices: $v_1,v_2,v_3\in V(G)$ and a copy of each of them $v_1',v_2',v_3'$, with an edge connecting each pair $v_i,v_i'$ and a number of edges between $v_i,v_j$ and $v_i',v_j'$ such that their sum is equal to the number of edges between $v_i,v_j$ in $V(G).$ The morphism sending the three edge $v_iv_i'$ and contracting everything else yields a harmonic morphism of degree $3$ to $T=K_2$ makes $G'$ trigonal and the contraction of the $3$ edges, which is a trigonal contraction, gives $G$, as shown in Figure~\ref{fg:v3}.
\begin{figure}[ht]
\centering
\begin{tikzcd}
\begin{tikzpicture}
    \draw[blue](0,0)--(2,0.5);
    \draw[blue](1,1)--(3,1.5);
    \draw[blue](2,0)--(4,0.5);
    \vertex{2,0.5}\vertex{3,1.5}\vertex{4,0.5}
    \draw[] (2,0.5) to [out=60, in=210] (3,1.5);
    \draw[] (4,0.5) to [out=160, in=280] (3,1.5);
        \draw (0,0) node[anchor=north] {$v_1$};
    \draw (2,0) node[anchor=north] {$v_3$};
    \draw (1,1) node[anchor=south east] {$v_2$};
    \vertex{0,0}\vertex{1,1}\vertex{2,0}
    \draw[] (0,0) to [out=340, in=200] (2,0);
    \draw[] (0,0) to [out=80, in=190] (1,1);
    \draw[] (0,0) to [out=20, in=260] (1,1);
    \draw[] (2,0) to [out=110, in=350] (1,1);
        \draw (0,0) node[anchor=north] {$v_1$};
    \draw (2,0) node[anchor=north] {$v_3$};
    \draw (1,1) node[anchor=south east] {$v_2$};
    
    \end{tikzpicture}\arrow[r,"\pi"]&
    \begin{tikzpicture}
        \vertex{0,0}\vertex{1,1}\vertex{2,0}
    \draw[] (0,0) to [out=340, in=200] (2,0);
    \draw[] (0,0) to [out=80, in=190] (1,1);
    \draw[] (0,0) to [out=60, in=210] (1,1);
    \draw[] (0,0) to [out=20, in=260] (1,1);
    \draw[] (2,0) to [out=110, in=350] (1,1);
    \draw[] (2,0) to [out=160, in=280] (1,1);
    \draw (0,0) node[anchor=north] {$v_1$};
    \draw (2,0) node[anchor=north] {$v_3$};
    \draw (1,1) node[anchor=south] {$v_2$};
    \end{tikzpicture}
    \end{tikzcd}\caption{}\label{fg:v3}
\end{figure}

Let us now consider the case $|V(G)|>3$ and consider the non-degenerate harmonic morphism 
$\varphi:G'\to T$ of degree $3$ with $(G')^{\operatorname{st}}=G$ and $T$ a tree.
We want to show that $G'=G_T$ and this will be done using the properties of harmonic morphisms.

We first claim that for any $v\in V(G'),$ $m_{\varphi}(v)=1.$
Assume by contradiction that there is a vertex $v$ with $m_{\varphi}(v)=3.$
By $3$-edge connectivity of $G,$ we have $\mu_{\varphi}(e)=1$ for any $e\in E_v(G')$ and therefore $\operatorname{val}(v)\geq6$. Then we can replace $v$ with two vertices~$v_1,v_2$ such that $v_1$ is incident to $2$ of any of the $3$ edges sent to an edge via $\varphi,$ and $v_2$ is incident to the remaining ones. Moreover, $v_1,v_2$ are connected to an edge whose contraction gives us back $G'$.
Clearly the resulting graph still admits a non-degenerate harmonic morphism of degree $3$ to the same tree.

If instead $m_{\varphi}(v)=2,$ then we are for instance in the previous case whit $v=v_1$ and $\operatorname{val}(v)\geq 4.$ Replace $v_1$ with two vertices, $v_1',v_2'$ as shown in Figure~\ref{fg:mult1}.

Also in this case, $v_1',v_2'$ are connected to an edge whose contraction gives us back the previous case and again the resulting graph still admits a non-degenerate harmonic morphism of degree $3$ to the same tree.

\begin{figure}[ht]
\centering
\begin{tikzcd}
\begin{tikzpicture}
    \draw[blue](0,0)--(1,0);
    \draw[blue](0,0.5)--(1,0.5);
    \draw[blue](0,1)--(1,1);
    \draw[red](2,0)--(1,0);
    \draw[red](2,0.5)--(1,0.5);
    \draw[red](2,1)--(1,1);
    \draw[->] (1,-0.4) to (1,-0.7);
    \draw[blue](0,-1)--(1,-1);
    \draw[red](2,-1)--(1,-1);
    \draw(1,1)--(1,0.5);
    \draw(1,0)--(1,0.5);
    \foreach \i in {0,1,2} {
    \foreach \j in {0,0.5,1,-1} {
    	    \vertex{\i, \j}
         }
    	}
     \draw (0.9,0.9) node[anchor=south west] {$v_1'$};
     \draw (0.9,0.4) node[anchor=south west] {$v_2'$};
     \draw (1,0) node[anchor=north] {$v_2$};
    \end{tikzpicture}\arrow[r,"\pi_2"]&
    \begin{tikzpicture}
    \draw[blue](0,0)--(1,0);
    \draw[blue](0,0.5)--(1,0.5);
    \draw[blue](0,1)--(1,0.5);
    \draw[red](2,0)--(1,0);
    \draw[red](2,0.5)--(1,0.5);
    \draw[red](2,1)--(1,0.5);
    \draw[->] (1,-0.4) to (1,-0.7);
    \draw[blue](0,-1)--(1,-1);
    \draw[red](2,-1)--(1,-1);
    
    \draw(1,0)--(1,0.5);
    \foreach \i in {0,1,2} {
    \foreach \j in {0,0.5,-1} {
    	    \vertex{\i, \j}
         }
    	}
     \vertex{0,1}\vertex{2,1}
     \draw (1,0.5) node[anchor=south ] {$v_1$};
     \draw (1,0) node[anchor=north] {$v_2$};
    \end{tikzpicture}\arrow[r,"\pi_1"]&
    \begin{tikzpicture}
    \draw[blue](0,0)--(1,0.5);
    \draw[blue](0,0.5)--(1,0.5);
    \draw[blue](0,1)--(1,0.5);
    \draw[red](2,0)--(1,0.5);
    \draw[red](2,0.5)--(1,0.5);
    \draw[red](2,1)--(1,0.5);
    \draw[->] (1,-0.4) to (1,-0.7);
    \draw[blue](0,-1)--(1,-1);
    \draw[red](2,-1)--(1,-1);
    \foreach \i in {0,1,2} {
    \foreach \j in {0.5,-1} {
    	    \vertex{\i, \j}
         }
    	}
     \vertex{0,1}\vertex{2,1}\vertex{0,0}\vertex{2,0}
    \draw (1,0.5) node[anchor=south] {$v$};
    \end{tikzpicture}
    \end{tikzcd}\caption{}\label{fg:mult1}
\end{figure}

Notice that this argument can also be applied to exclude the case where there are loops in $G'$. 

From the fact that $m_{\varphi}(v)=1$ for any vertex $v,$ the graph $G'$ contains three disjoint copies of the tree $T$.

The valence of the vertices in $T$ cannot exceed $3.$ If $T$ contains a vertex $v$ with $\operatorname{val}(v)>3,$ then it is sufficient to consider the tree  obtained by replacing $v$ with two vertices $v_1,v_2,$ connected via an edge $e,$ such that $\operatorname{val}(v_1)=3,$ and $\operatorname{val}(v_2)=\operatorname{val}(v)-1\geq 3.$ Then consider the graph where all preimages of $v$ are also replaced by the above construction and the equivalence relation that is equal to that induced by~$\varphi$ over the edges that where already in $G'$ and such that the pre-images of the added edges belong to the same equivalence class, as in Figure~\ref{fg:tree}. If $\operatorname{val}(v_2)>3,$ repeat the construction until both vertices have valence $3.$ 

\begin{figure}[ht]
\centering
\begin{tikzcd}
\begin{tikzpicture}
    \draw[blue](0,1.9)--(1,1.9);
    \draw[orange](0,2.4)--(1,2.4);
    \draw[cyan](1,1.9)--(1,2.4);
    \draw[red](2,1.9)--(1,1.9);
    \draw[verde](2,2.4)--(1,2.4); 
    
    \draw[blue](0,1.2)--(1,1.2);
    \draw[orange](0,1.7)--(1,1.7);
    \draw[cyan](1,1.2)--(1,1.7);
    \draw[red](2,1.2)--(1,1.2);
    \draw[verde](2,1.7)--(1,1.7); 
    
    \draw[blue](0,0.5)--(1,0.5);
    \draw[orange](0,1)--(1,1);
    \draw[cyan](1,0.5)--(1,1);
    \draw[red](2,0.5)--(1,0.5);
    \draw[verde](2,1)--(1,1);  
    
    \draw[->](1,0.3)to(1,0);
    \draw[blue](0,-1)--(1,-1);
    \draw[orange](0,-0.5)--(1,-0.5);
    \draw[red](2,-1)--(1,-1);
    \draw[verde](2,-0.5)--(1,-0.5);
    \draw[cyan](1,-1)--(1,-0.5);
    \foreach \i in {0,1,2} {
    \foreach \j in {-1,-0.5,0.5,1,1.2,1.7,1.9,2.4} {
    	    \vertex{\i, \j}
         }
    	}
     \draw (1,-0.5) node[anchor=south ] {$v_1$};
     \draw (1,-1) node[anchor=north] {$v_2$};
    \end{tikzpicture}\arrow[r,"\pi"]&
    \begin{tikzpicture}
    \draw[blue](0,1.9)--(1,2.1);
    \draw[orange](0,2.4)--(1,2.1);
    \draw[red](2,1.9)--(1,2.1);
    \draw[verde](2,2.4)--(1,2.1); 
    
    \draw[blue](0,1.2)--(1,1.45);
    \draw[orange](0,1.7)--(1,1.45);
    \draw[red](2,1.2)--(1,1.45);
    \draw[verde](2,1.7)--(1,1.45); 
    
    \draw[blue](0,0.5)--(1,0.75);
    \draw[orange](0,1)--(1,0.75);
    \draw[red](2,0.5)--(1,0.75);
    \draw[verde](2,1)--(1,0.75);  
    
    \draw[->](1,0.3)to(1,0);
    \draw[blue](0,-1)--(1,-0.75);
    \draw[orange](0,-0.5)--(1,-0.75);
    \draw[red](2,-1)--(1,-0.75);
    \draw[verde](2,-0.5)--(1,-0.75);
    \foreach \i in {0,2} {
    \foreach \j in {-1,-0.5,0.5,1,1.2,1.7,1.9,2.4} {
    	    \vertex{\i, \j}
         }
    	}
    \foreach \j in {-0.75,0.75,1.45,2.1} {
    	    \vertex{1, \j} }
     \draw (1,-0.75) node[anchor=south ] {$v$};
    \end{tikzpicture}
    \end{tikzcd}\caption{}\label{fg:tree}
\end{figure}

Let us now consider all the possibilities given by the valence of the vertices on the tree, which now we have proved to be at most $3$.

If $\operatorname{val}(v)=3,$ then we claim that its pre-image via $\varphi$ does not contain any edge. 

If this would be the case, then it would be sufficient to replace any copy of $T$ with a graph where $v$ is split into two vertices connected via an edge such that all copied of the same edge belong to the same equivalence class. Repeat as many time as many edges are in $\varphi^{-1}(v),$ as in Figure~\ref{fg:tree3}, and also notice that we can always choose the vertical edge over consecutive vertices of valence two, such that they connect distinct copies of the tree. The resulting graph is such that the pre-image of a vertex of valence~$3$ via the harmonic morphism does not contain edges.
\begin{figure}[ht]
\centering
\begin{tikzcd}
\begin{tikzpicture}
    \draw[blue](0,1.9)--(1,2.15);
    \draw[red](0,2.4)--(1,2.15);
    \draw[cyan](1,2.15)--(2,2.15);
    \draw[verde](2,2.15)--(3,2.15);
    \draw(2,2.15)--(2,1.45);
    \draw[blue](0,1.2)--(1,1.45);
    \draw[red](0,1.7)--(1,1.45);
    \draw[cyan](1,1.45)--(2,1.45);
    \draw[verde](2,1.45)--(3,1.45); 
    
    \draw[blue](0,0.5)--(1,0.75);
    \draw[red](0,1)--(1,0.75);
    \draw[cyan](1,0.75)--(2,0.75);
    \draw[verde](2,0.75)--(3,0.75);
    
    \draw[->](1.5,0.3)to(1.5,0);
    \draw[blue](0,-1)--(1,-0.75);
    \draw[red](0,-0.5)--(1,-0.75);
    \draw[cyan](1,-0.75)--(2,-0.75);
    \draw[verde](2,-0.75)--(3,-0.75);
    \foreach \j in {-1,-0.5,0.5,1,1.2,1.7,1.9,2.4} {
    	    \vertex{0, \j}
         }
    \foreach \i in {1,2,3} {
    \foreach \j in {-0.75,0.75,1.45,2.15}{
    	    \vertex{\i, \j} }}
    \end{tikzpicture}\arrow[r,"\pi"]&
    \begin{tikzpicture}
    \draw[blue](0,1.9)--(1,2.15);
    \draw[red](0,2.4)--(1,2.15);
    \draw[verde](1,2.15)--(2,2.15);
    \draw(1,2.15)--(1,1.45);
    \draw[blue](0,1.2)--(1,1.45);
    \draw[red](0,1.7)--(1,1.45);
    \draw[verde](1,1.45)--(2,1.45);
    
    \draw[blue](0,0.5)--(1,0.75);
    \draw[red](0,1)--(1,0.75);
    \draw[verde](1,0.75)--(2,0.75);
    
    \draw[->](1,0.3)to(1,0);
    \draw[blue](0,-1)--(1,-0.75);
    \draw[red](0,-0.5)--(1,-0.75);
    \draw[verde](1,-0.75)--(2,-0.75);
    \foreach \j in {-1,-0.5,0.5,1,1.2,1.7,1.9,2.4} {
    	    \vertex{0, \j}
         }
    \foreach \i in {1,2} {
    \foreach \j in {-0.75,0.75,1.45,2.15}{
    	    \vertex{\i, \j} }}
    \end{tikzpicture}
    \end{tikzcd}\caption{}\label{fg:tree3}
\end{figure}

If $\operatorname{val}(v)=2,$ and $\varphi^{-1}(v)$ does not contain any edge, then the graph defines the same type obtained by contacting one of the two equivalence classes defined by the incident edge to $v.$ 
Thus we may assume that $\varphi^{-1}(v)$ contains at least one edge, and we show that it is exactly one if the graph is maximal.
Similarly to the previous case, in fact, if there at least two edges in $\varphi^{-1}(v)$, then we could just repeat the above procedure for any additional edge, this is represented in Figure~\ref{fg:tree2}.

\begin{figure}[ht]
\centering
\begin{tikzcd}
\begin{tikzpicture}
    \draw[red](0,2.15)--(1,2.15);
    \draw[cyan](1,2.15)--(2,2.15);
    \draw[blue](2,2.15)--(3,2.15);
    \draw(2,2.15)--(2,1.45);
    \draw[red](0,1.45)--(1,1.45);
    \draw[cyan](1,1.45)--(2,1.45);
    \draw[blue](2,1.45)--(3,1.45); 
    \draw(1,1.45)--(1,0.75);
    \draw[red](0,0.75)--(1,0.75);
    \draw[cyan](1,0.75)--(2,0.75);
    \draw[blue](2,0.75)--(3,0.75);
    
    \draw[->](1.5,0.3)to(1.5,0);
    \draw[red](0,-0.75)--(1,-0.75);
    \draw[cyan](1,-0.75)--(2,-0.75);
    \draw[blue](2,-0.75)--(3,-0.75);

    \foreach \i in {0,1,2,3} {
    \foreach \j in {-0.75,0.75,1.45,2.15}{
    	    \vertex{\i, \j} }}
    \end{tikzpicture}\arrow[r,"\pi"]&
    \begin{tikzpicture}
    \draw[red](0,2.15)--(1,2.15);
    \draw[blue](1,2.15)--(2,2.15);
    \draw(1,2.15)--(1,1.45);
    \draw[red](0,1.45)--(1,1.45);
    \draw[blue](1,1.45)--(2,1.45);
    \draw(1,1.45)--(1,0.75);
    \draw[red](0,0.75)--(1,0.75);
    \draw[blue](1,0.75)--(2,0.75);
    
    \draw[->](1,0.3)to(1,0);
    \draw[red](0,-0.75)--(1,-0.75);
    \draw[blue](1,-0.75)--(2,-0.75);

    \foreach \i in {0,1,2} {
    \foreach \j in {-0.75,0.75,1.45,2.15}{
    	    \vertex{\i, \j} }}
    \end{tikzpicture}
    \end{tikzcd}\caption{}\label{fg:tree2}
\end{figure}

Finally if $\operatorname{val}(v)=1,$ we have at least $2$ vertical edges, by $3$-edge connectivity, and we show that they are exactly $2$.

Assume by contradiction that the number of vertical edges is $3,$ then, again by $3$-edge connectivity, not all with the same endpoints.
Then we can construct a graph $G''$ that splits any copy of $v$ into two vertices $v_1,v_2,$ with vertical edges as in Figure~\ref{fg:tree1}.

Then $G''$ is again trigonal and it is the contraction of the pre-image of added edge in $T$. 
\begin{figure}[ht]
\centering
\begin{tikzcd}
\begin{tikzpicture}
    \draw[cyan](0,2.15)--(1,2.15);
    \draw[red](1,2.15)--(2,2.15);
    \draw (1,1.45) to [out=60, in=300] (1,2.15);
    \draw[cyan](0,1.45)--(1,1.45);
    \draw[red](1,1.45)--(2,1.45);
    \draw[cyan](0,0.75)--(1,0.75);
    \draw[red](1,0.75)--(2,0.75);
    \draw (0,0.75) to [out=60, in=300] (0,1.45);
    \draw (0,0.75) to [out=120, in=240] (0,2.15);
    \draw[->](1,0.3)to(1,0);
    \draw[cyan](0,-0.75)--(1,-0.75);
    \draw[red](1,-0.75)--(2,-0.75);
    \draw (0,-0.75) node[anchor=north] {$v_1$};
    \draw (1,-0.75) node[anchor=north] {$v_2$};
    \foreach \i in {0,1,2} {
    \foreach \j in {-0.75,0.75,1.45,2.15}{
    	    \vertex{\i, \j} }}
\end{tikzpicture}\arrow[r,"\pi"]&\begin{tikzpicture}
    \draw[red](0,2.15)--(1,2.15);
    \draw (0,1.45) to [out=60, in=300] (0,2.15);
    \draw[red](0,1.45)--(1,1.45);
    \draw[red](0,0.75)--(1,0.75);
    \draw (0,0.75) to [out=60, in=300] (0,1.45);
    \draw (0,0.75) to [out=120, in=240] (0,2.15);
    \draw[->](0.5,0.3)to(0.5,0);
    \draw[red](0,-0.75)--(1,-0.75);
    \draw (0,-0.75) node[anchor=north] {$v$};
    \foreach \i in {0,1} {
    \foreach \j in {-0.75,0.75,1.45,2.15}{
    	    \vertex{\i, \j} }}
    \end{tikzpicture}
    \end{tikzcd}\caption{}\label{fg:tree1}
\end{figure}

If the number of edges over $v$ is bigger that $3,$ then it is sufficient to repeat the above procedure.

To sum up, for any $v\in V(T),$ we have considered the minimum number of vertical edges in its pre-image, depending on the valence, and proved that this is precisely the number of edges allowed in order for $G'$ to be maximal. Comparing with our construction of the graph $G_T$ we conclude that $G'=G_T.$

The proof also shows that $\varphi=\varphi_T$
and by Remark~\ref{rk:max} we can conclude $(G,\underline{0})$ is a maximal trigonal graph, since any other morphism would define a bigger number of equivalence classes.
\end{proof}

Lemmas~\ref{lm:2n+1},~\ref{lm:stab} and Theorem~\ref{th:max_cells} then show that the dimension of the moduli of $3$-edge connected trigonal tropical curves of genus $g$ coincides with that of the moduli of genus $g$ algebraic trigonal curves.
\begin{rema}
Let us recall that any stable graph is the limit, via weighted contractions, of a pure and $3$-regular graph, which also represents the maximal cells in the moduli of tropical curves.

Similarly, the proof of Theorem~\ref{th:max_cells} also shows that any $3$-edge connected trigonal graph is a graph which has a refinement that is the limit, via trigonal contractions, of a $3$-ladder $G_T$.
\end{rema} 

\begin{coro}
The moduli space $T_g^{\operatorname{trop},(3)}$ is of pure dimension $6$ for $g=3$ and of pure dimension $2g+1$ for $g>3$.
\end{coro}

This also agrees with~\cite{CD}. Indeed, in Remark~\ref{rk:tree_gon} we have already observed that in the $3$-edge connected case tropical curves of gonality $3$ also have tree gonality $3.$

If instead we drop the $3$-edge connectivity assumption then we need to take also into account hyperelliptic curves. Therefore, the dimension of the corresponding moduli space would be the same as the space $H_g^{tr}$ constructed in~\cite{MC}, which is $3g-3$ and therefore different from the one of the trigonal locus in the algebraic case. 

\begin{coro}
The moduli space $T_g^{\operatorname{trop},(3)}$ is connected through codimension $1$.\end{coro}\begin{proof}
In order to prove that $T_g^{\operatorname{trop},(3)}$ is connected through codimension $1$ we need to show that, for any pair of $3$-ladders $G_{T_1},G_{T_2}$,  there exists a sequence $G_{T_1}=G^0,\dots,G^{2k}=G_{T_2}$ such that $G^i,G^{i+1}$ differ by the contraction of a single class in the equivalence relation defined by the associated morphism to a tree and $G^i$ is a $3$-ladder for $i$ even.

Firstly, let us consider the case where $T_1=T_2.$ Clearly, over the vertices of valence $3$ the two graphs are the same, since there are no vertical edges. If $G_{T_1}, G_{T_2}$ differ by the vertical edges on any vertex of valence $1$, it is sufficient to contract the vertical edge $G_{T_1}$ and not in $G_{T_2},$ and then add the one in $G_{T_2}.$
Instead, if the two graphs differ over a vertex of valence $2,$ then one can obtain from $G_{T_1}$ the corresponding vertical edge in $G_{T_2}$ via a sequence of contraction and addition of triples of edges with the same image via the non-degenerate harmonic morphism to a tree. In fact, let us observe that any two consecutive vertical edges, not with endpoint in the same copies of the tree, can be swapped by contracting the $3$ edge-cut incident to both vertical edges and by adding it again with the vertical edges now glued to the opposite endpoints of the edge-cut, as in Figure~\ref{fg:codim1_1}. With more iterations, one can see that this can be done also if between the two vertical edges there is a vertex of valence $3.$ 

\begin{figure}[ht]
\centering
\begin{tikzcd}
\begin{tikzpicture}
    \draw[red](0,2.15)--(1,2.15);
    \draw[cyan](1,2.15)--(2,2.15);
    \draw[blue](2,2.15)--(3,2.15);
    \draw(2,2.15)--(2,1.45);
    \draw[red](0,1.45)--(1,1.45);
    \draw[cyan](1,1.45)--(2,1.45);
    \draw[blue](2,1.45)--(3,1.45); 
    \draw(1,1.45)--(1,0.75);
    \draw[red](0,0.75)--(1,0.75);
    \draw[cyan](1,0.75)--(2,0.75);
    \draw[blue](2,0.75)--(3,0.75);
    
    \draw[->](1.5,0.3)to(1.5,0);
    \draw[red](0,-0.75)--(1,-0.75);
    \draw[cyan](1,-0.75)--(2,-0.75);
    \draw[blue](2,-0.75)--(3,-0.75);

    \foreach \i in {0,1,2,3} {
    \foreach \j in {-0.75,0.75,1.45,2.15}{
    	    \vertex{\i, \j} }}
    \end{tikzpicture}\arrow[r]&
    \begin{tikzpicture}
    \draw[red](0,2.15)--(1,2.15);
    \draw[blue](1,2.15)--(2,2.15);
    \draw(1,2.15)--(1,1.45);
    \draw[red](0,1.45)--(1,1.45);
    \draw[blue](1,1.45)--(2,1.45);
    \draw(1,1.45)--(1,0.75);
    \draw[red](0,0.75)--(1,0.75);
    \draw[blue](1,0.75)--(2,0.75);
    
    \draw[->](1,0.3)to(1,0);
    \draw[red](0,-0.75)--(1,-0.75);
    \draw[blue](1,-0.75)--(2,-0.75);

    \foreach \i in {0,1,2} {
    \foreach \j in {-0.75,0.75,1.45,2.15}{
    	    \vertex{\i, \j} }}
\end{tikzpicture}&\arrow[l]
    \begin{tikzpicture}
    \draw[red](0,2.15)--(1,2.15);
    \draw[cyan](1,2.15)--(2,2.15);
    \draw[blue](2,2.15)--(3,2.15);
    \draw(2,1.45)--(2,0.75);
    \draw[red](0,1.45)--(1,1.45);
    \draw[cyan](1,1.45)--(2,1.45);
    \draw[blue](2,1.45)--(3,1.45); 
    \draw(1,1.45)--(1,2.15);
    \draw[red](0,0.75)--(1,0.75);
    \draw[cyan](1,0.75)--(2,0.75);
    \draw[blue](2,0.75)--(3,0.75);
    
    \draw[->](1.5,0.3)to(1.5,0);
    \draw[red](0,-0.75)--(1,-0.75);
    \draw[cyan](1,-0.75)--(2,-0.75);
    \draw[blue](2,-0.75)--(3,-0.75);

    \foreach \i in {0,1,2,3} {
    \foreach \j in {-0.75,0.75,1.45,2.15}{
    	    \vertex{\i, \j} }}
    \end{tikzpicture}
    \end{tikzcd}\caption{}\label{fg:codim1_1}
\end{figure}

Then, using the argument for the vertical edges over the vertices of valence $1$ one can always produce a pair of vertical edges with endpoints not in the same copies of the tree, such that one is the vertical edge in $G_{T_2}$ that we one to recover from $G_{T_1}$, as in Figure~\ref{fg:codim1_2}. Iterating these two operations then yields a sequence from $G_{T_1}$ to $G_{T_2}$ with the desired properties. 

\begin{figure}[ht]
\centering
\begin{tikzcd}[sep=small]
\begin{tikzpicture}[scale=0.9]
    \draw(0,2.15)--(0,1.45);\draw(0,0.75)--(0,1.45);
    \draw[red](0,2.15)--(1,2.15);
    \draw[cyan](1,2.15)--(2,2.15);
    \draw[blue](2,2.15)--(3,2.15);
    \draw(2,2.15)--(2,1.45);
    \draw[red](0,1.45)--(1,1.45);
    \draw[cyan](1,1.45)--(2,1.45);
    \draw[blue](2,1.45)--(3,1.45); 
    \draw(1,1.45)--(1,2.15);
    \draw[red](0,0.75)--(1,0.75);
    \draw[cyan](1,0.75)--(2,0.75);
    \draw[blue](2,0.75)--(3,0.75);
    
    \draw[->](1.5,0.3)to(1.5,0);
    \draw[red](0,-0.75)--(1,-0.75);
    \draw[cyan](1,-0.75)--(2,-0.75);
    \draw[blue](2,-0.75)--(3,-0.75);

    \foreach \i in {0,1,2,3} {
    \foreach \j in {-0.75,0.75,1.45,2.15}{
    	    \vertex{\i, \j} }}
    \end{tikzpicture}\arrow[r]&\dots\arrow[r]&
    \begin{tikzpicture}[scale=0.9]
    \draw (0,0.75) to [out=60, in=300] (0,1.45);
    \draw (0,0.75) to [out=120, in=240] (0,2.15);
    \draw[red](0,2.15)--(1,2.15);
    \draw[cyan](1,2.15)--(2,2.15);
    \draw[blue](2,2.15)--(3,2.15);
    \draw(2,2.15)--(2,1.45);
    \draw[red](0,1.45)--(1,1.45);
    \draw[cyan](1,1.45)--(2,1.45);
    \draw[blue](2,1.45)--(3,1.45); 
    \draw(1,1.45)--(1,2.15);
    \draw[red](0,0.75)--(1,0.75);
    \draw[cyan](1,0.75)--(2,0.75);
    \draw[blue](2,0.75)--(3,0.75);
    
    \draw[->](1.5,0.3)to(1.5,0);
    \draw[red](0,-0.75)--(1,-0.75);
    \draw[cyan](1,-0.75)--(2,-0.75);
    \draw[blue](2,-0.75)--(3,-0.75);

    \foreach \i in {0,1,2,3} {
    \foreach \j in {-0.75,0.75,1.45,2.15}{
    	    \vertex{\i, \j} }}
\end{tikzpicture}\arrow[r]&\dots\arrow[r]&
    \begin{tikzpicture}[scale=0.9]
    \draw (0,1.45) to [out=60, in=300] (0,2.15);
    \draw (0,0.75) to [out=120, in=240] (0,2.15);
    \draw[red](0,2.15)--(1,2.15);
    \draw[cyan](1,2.15)--(2,2.15);
    \draw[blue](2,2.15)--(3,2.15);
    \draw(2,2.15)--(2,1.45);
    \draw[red](0,1.45)--(1,1.45);
    \draw[cyan](1,1.45)--(2,1.45);
    \draw[blue](2,1.45)--(3,1.45); 
    \draw(1,1.45)--(1,0.75);
    \draw[red](0,0.75)--(1,0.75);
    \draw[cyan](1,0.75)--(2,0.75);
    \draw[blue](2,0.75)--(3,0.75);
    
    \draw[->](1.5,0.3)to(1.5,0);
    \draw[red](0,-0.75)--(1,-0.75);
    \draw[cyan](1,-0.75)--(2,-0.75);
    \draw[blue](2,-0.75)--(3,-0.75);

    \foreach \i in {0,1,2,3} {
    \foreach \j in {-0.75,0.75,1.45,2.15}{
    	    \vertex{\i, \j} }}
\end{tikzpicture}
    \end{tikzcd}\caption{}\label{fg:codim1_2}
\end{figure}

Let us then consider the case where $T_1\neq T_2.$ By definition, they are trees over the same number of vertices, which we can make correspond such that they share the maximum number of edges. Starting from $T_1$ by removing any leaf-edge which is not in $T_2$ and then add instead the edge in $T_2$ connecting the leaf with the common path and repeating until we obtain $T_2$ defines a sequence between the two trees ${T_1}=T^0,\dots,T^{2h}={T_2}$ such that $T^i,T^{i+1}$ differ by a single edge contraction and $T^i$ has the maximum number of edges for $i$ even.
Any contraction of an edge in $T_1$ clearly corresponds to the contraction of three edges in $G_{T_1}$ having the same edge as an image. Moreover, as in Figures~\ref{fg:tree},\ref{fg:tree3},~\ref{fg:tree2},~\ref{fg:tree1}, any addition of an edge, also corresponds to the addition of tree edges which yield a $3$-ladder. Therefore given $G_{T_1},G_{T_2},$ with $T_1\neq T_2,$ one can always construct a sequence from $G_{T_1}$ to a $3$-ladder over $T_2$ of contractions and additions of triples of edges in the same equivalence class and then use the first argument to extend the sequence to $G_{T_2}.$
\end{proof}

Finally, we would like to relate a $3$-edge connected trigonal type $(G,w,\varphi)$, with the dual graph of an admissible cover. The construction of the latter can be found in~\cite[Section 2.3]{LC}.

In particular, nodes map to nodes, therefore no contraction of edges is allowed, but as we proved in Proposition~\ref{prp: harm_contractions} we can consider instead its tropical modification with no contractions.

We have indeed already observed in Remark~\ref{rk:tree_gon} that such a tropical modification is a tropical morphism, for which the Riemann-Hurwitz inequality holds, and we will now prove that equality holds without further tropical modifications.

We conclude by relating the moduli space of $3$-edge connected trigonal covers with that of tropical admissible covers of degree $3$, as defined in~\cite{CMR}. 
In particular, they define in~\cite[Definition 16]{CMR} a \emph{tropical admissible cover} of tropical curves  $\psi:\Gamma_{\operatorname{src}}=(G_{\operatorname{src}},l_{\operatorname{src}})\to \Gamma_{\operatorname{tgt}}=(G_{\operatorname{tgt}},l_{\operatorname{tgt}})$ as an harmonic morphism satisfying the \emph{local Riemann-Hurwitz equation} at any point, i.e. such that for any $v\in V(G_{\operatorname{src}})$ if~$v'=\psi(v),$
\begin{equation}\label{eq:RH}
    2-2w(v)=m_{\psi}(v)(2-2w(v'))-\sum_{e\in E_v(G_{\operatorname{src}})}(\mu_\psi(e)-1).
\end{equation}

\begin{prop}\label{prp:adm}
Let $(G,w,\tilde{\varphi}_T)$ be a $3$-edge connected trigonal type with $G$ such that $(G_T)^{\operatorname{st}}=G$ where $G_T$ is a $3$-ladder and $\tilde{\varphi}_T: \tilde G_T\to\tilde T$ its tropical modification with no contractions of $\varphi_T$.

Then, the trigonal type $(G,w,\tilde{\varphi_T})$ and any $\tilde{\varphi_T}$-contraction of it is the combinatorial type of a tropical admissible cover.
\end{prop}

\begin{proof}
Since the local Riemann-Hurwitz equation is closed by contraction, we can always assume our graph to be weightless and the stabilization of a $3$-ladder $G_T$.
Then the left hand-side of~\eqref{eq:RH} is $2$
while the right hand side is, 
$$
    2m_{\psi}(v)-\sum_{e\in E_v(G_{\operatorname{src}})}(\mu_\psi(e)-1).
$$
We also observed, in the proof of Theorem~\ref{th:max_cells}, that maximal $3$-edge connected trigonal graphs are such that $m_{\varphi}(v)=1$ for any vertex. 

Moreover, from our construction of the harmonic morphism in the $3$-edge connected case we have $\mu_\psi(e)=1$ for any $e$ which is sent to an edge of the tree. Since $\psi$ has no contraction, then $\mu_\psi(e)=1$ for any $e\in E_x(G_{\operatorname{src}})$, hence equality holds.
\end{proof}
\bibliographystyle{alpha}
\bibliography{bib}
\end{document}